\documentclass[11pt]{article}
\usepackage{fullpage,amsfonts,amsmath,amssymb,amsthm,bm,xspace,thmtools,thm-restate,asymptote,mathrsfs,appendix,array}
\usepackage{multirow,graphicx,
  algorithm,algorithmic,caption,subcaption,url,cite,float}
\usepackage{comment}
\usepackage{mathtools}
\usepackage{todonotes}
\usepackage{array} 
\usepackage[noabbrev]{cleveref}
\usepackage{listings}
\usepackage{xpatch}
\usepackage{thmtools}
\usepackage{apptools}

\makeatletter
\xpatchcmd{\thmt@restatable}{\csname #2\@xa\endcsname\ifx\@nx#1\@nx\else[{#1}]\fi}%
{\IfAppendix{\csname #2\@xa\endcsname}{\csname #2\@xa\endcsname[{#1}]}}
\makeatother

\newcolumntype{C}{ >{\centering\arraybackslash} m{7cm} }
\newcolumntype{D}{ >{\centering\arraybackslash} m{2cm} }
\newcolumntype{E}{ >{\centering\arraybackslash} m{1.5cm} }
\newcolumntype{F}{ >{\centering\arraybackslash} m{8.3cm} }
\newcolumntype{G}{ >{\centering\arraybackslash} m{8cm} }

\newcommand{\alignmulticolumn}[1]{\omit\rlap{$\displaystyle{#1}$}}
\newtheorem{theorem}{Theorem}[section]

\newtheorem{lemma}[theorem]{Lemma}
\newtheorem{corollary}[theorem]{Corollary}

\newtheorem{definition}[theorem]{Definition}

\theoremstyle{remark}

\newtheorem{example}[theorem]{Example}

\numberwithin{equation}{section}

\newcommand{\smo}[1]{#1^{\downarrow}_\infty}

\newcommand{\sym}[1]{#1_\infty}

\newcommand{\wid}{\tau}
\newcommand{\DeltaG}{3.751}
\newcommand{\DeltaR}{5.056}
\newcommand{\botwid}{0.065}
\newcommand{\topwid}{0.2375}
\newcommand{\botwidG}{\botwid}
\newcommand{\topwidG}{\topwid}
\newcommand{\botwidR}{0.0775}
\newcommand{\topwidR}{0.165}
\newcommand{\etaG}{1.08732}
\newcommand{\etaR}{1.12597}
\newcommand{\psiG}{\etaG}
\newcommand{\psiR}{\etaR}
\newcommand{\zetaG}{0.087317}
\newcommand{\zetaR}{0.12597}
\newcommand{\betaG}{0.118824}
\newcommand{\betaR}{0.261070}
\newcommand{\alphaG}{1.09206}
\newcommand{\alphaR}{1.12633}
\newcommand{\alphajG}{0.907937}
\newcommand{\alphajR}{0.873678}
\newcommand{\QsiG}{1.39862}
\newcommand{\QsiR}{1.19352}
\newcommand{\waveG}{1.83193}
\newcommand{\waveR}{1.14055}

\newcommand{\KG}{K^{\mathcal{G}}}
\newcommand{\KR}{K^{\mathcal{R}}}
\newcommand{\DG}{D^{\mathcal{G}}}
\newcommand{\DR}{D^{\mathcal{R}}}

\newcommand{\wt}[1]{\widetilde{ #1 } }
\newcommand{\wh}[1]{\widehat{ #1 } }

\newcommand{\LB}{\mathrm{LB}}

\newcommand{\CCC}{\mathcal{C}}
\newcommand{\R}{\mathbb{R}}
\newcommand{\RR}{\mathbb{R}}
\newcommand{\BB}{\mathcal{B}}
\newcommand{\DBB}{\mathcal{B}^{\brac{1}}}
\newcommand{\cL}{\mathcal{L}}

\newcommand{\bxi}{\bar{\xi}}
\newcommand{\cC}{\mathcal{C}}

\newcommand{\cN}{\mathcal{N}}
\newcommand{\cI}{\mathcal{I}}
\newcommand{\Noi}{\cN}
\newcommand{\sB}{\mathscr{B}}

\newcommand{\sD}{\mathscr{D}}
\newcommand{\WW}{\mathcal{W}}
\newcommand{\DWW}{\mathcal{W}^{\brac{1}}}
\newcommand{\sW}{\mathscr{W}}
\newcommand{\II}{\mathcal{I}}

\newcommand{\ZZ}{\mathbb{Z}}

\newcommand{\qj}{q(j)}
\newcommand{\Qaux}{Q_{\text{aux}}}
\newcommand{\QI}{Q_{\ml{I}}}
\newcommand{\QC}{Q_{\ml{C}}}
\newcommand{\RQ}{R}
\newcommand{\IS}{\ml{I}}
\newcommand{\CS}{\ml{C}}
\newcommand{\RC}{R_{C}}
\newcommand{\nearNoise}{\eta}

\newcommand{\sign}{\operatorname*{sign}}

\newcommand{\bdb}[1]{#1}

\newcommand{\MAT}[1]{\begin{bmatrix} #1 \end{bmatrix}}

\newcommand{\abs}[1]{\left| #1 \right|}
\newcommand{\keys}[1]{\left\{ #1 \right\}}
\newcommand{\sqbr}[1]{\left[ #1 \right]}
\newcommand{\brac}[1]{\left( #1 \right) }
\newcommand{\ml}[1]{\mathcal{ #1 } }
\newcommand{\op}[1]{ \operatorname{#1} }
\newcommand{\normInf}[1]{\left\| #1 \right\| _{\infty}}
\newcommand{\normTwo}[1]{\left\| #1 \right\| _{2}}
\newcommand{\normOne}[1]{\left\| #1 \right\| _{1}}
\newcommand{\normTV}[1]{\left\| #1 \right\| _{\op{TV}}}

\newcommand{\diff}[1]{ \, \text{d} #1 }

\pdfoutput = 1
\title{Deconvolution of Point Sources:\\ A Sampling Theorem and Robustness Guarantees}

\author{Brett Bernstein\thanks{Courant Institute of Mathematical Sciences, New York University} \hspace{0.4cm} Carlos Fernandez-Granda\footnotemark[1] \thanks{Center for Data Science,
    New York University}}

\date{May 2018}

\begin{document}

\maketitle

\vspace{-0.3in}

\begin{abstract}
\noindent

In this work we analyze a convex-programming method for estimating superpositions of point sources or spikes from nonuniform samples of their convolution with a known kernel. We consider a one-dimensional model where the kernel is either a Gaussian function or a Ricker wavelet, inspired by applications in geophysics and imaging. Our analysis establishes that minimizing a continuous counterpart of the $\ell_1$ norm achieves exact recovery of the original spikes as long as (1) the signal support satisfies a minimum-separation condition and (2) there are at least two samples close to every spike. In addition, we derive theoretical guarantees on the robustness of the approach to both dense and sparse additive noise.

\end{abstract}


\section{Introduction}

The problem of deconvolution consists of estimating a signal $\mu$ from data $y$ modeled as the convolution of $\mu$ with a kernel $K$ sampled at a finite number of locations $s_1$, $s_2$, \ldots, $s_n \in \RR$
\begin{align}
y_i := \brac{ K \ast \mu} \brac{s_i} , \quad i = 1,2,\ldots , n.
\end{align}
This problem arises in multiple domains of the applied sciences,
including spectroscopy~\cite{kauppinen1981fourier,carley1979application}, neuroscience~\cite{vogelstein2010fast}, geophysics~\cite{mendel2013optimal}, ultrasound~\cite{abeyratne1995higher,jensen1992deconvolution}, optics~\cite{broxton2013wave}, where $K$ is the point-spread function of the
imaging system, and signal processing~\cite{mallat1999wavelet}, where $K$ is the
impulse response of a linear system. In many of these applications, the signal $\mu$ is well modeled as a superposition of point sources or spikes, which could correspond to heat sources~\cite{li2014heat}, fluorescent probes in microscopy~\cite{palm,fpalm}, celestial bodies in astronomy~\cite{astronomy_puschmann} or neural action potentials in neuroscience~\cite{rieke_spikes}. Mathematically, the spikes can be represented as a superposition of Dirac measures supported on a finite set $T \subseteq \R$ 
\begin{align}
\mu & := \sum_{t_j \in T} a_j\delta_{t_j}, \label{eq:signal}
\end{align} 
with amplitudes $a_1, a_2, \ldots \in\RR$. For such signals, the data correspond to samples from a linear combination of shifted and scaled copies of the convolution kernel
\begin{align}
y_i = \sum_{t_j \in T} a_j K \brac{s_i-t_j} , \quad i = 1,2,\ldots,n. \label{eq:data_dirac}
\end{align}

\begin{figure}[t]
\centering
\begin{tabular}{ECC}
&&\hspace{-1.5cm}  Spectrum \\ \\ 
Gaussian&\includegraphics{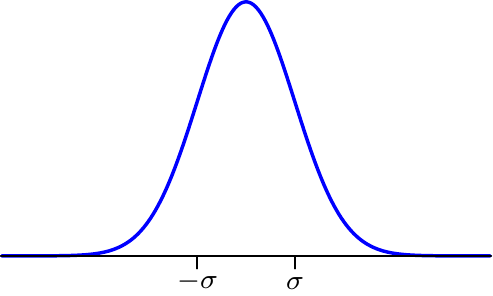}& \hspace{-1.5cm} \includegraphics{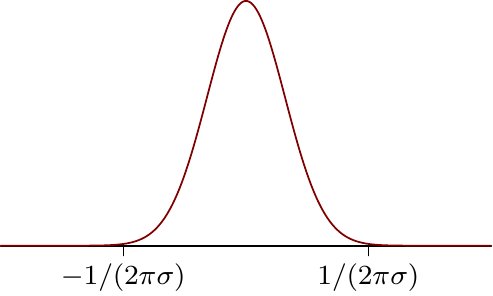}\\
Ricker&\includegraphics{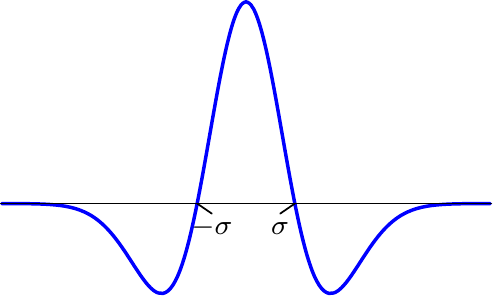}& \hspace{-1.5cm} \includegraphics{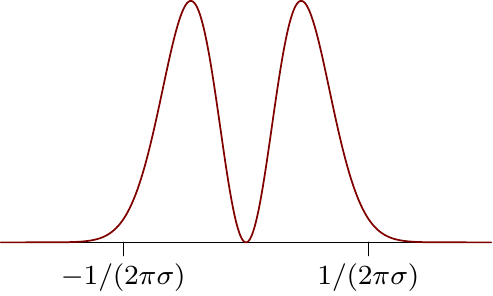}
\end{tabular}
\caption{The Gaussian and Ricker kernels with their respective spectra. }
\label{fig:Kernels}
\end{figure}

To make our analysis concrete, we focus on two specific convolution kernels, depicted in \Cref{fig:Kernels}, that are particularly popular in deconvolution applications: the Gaussian kernel,
\begin{align}
\KG(t) & := \exp\left(- \frac{t^2}{2\sigma^2}\right),
\end{align}
and the Ricker wavelet
\begin{align}
\KR(t) & := \brac{1-\frac{t^2}{\sigma^2}}\exp\left(-
\frac{t^2}{2\sigma^2}\right),
\end{align}
which equals the second derivative of the Gaussian kernel up to a
multiplicative constant. In both cases $\sigma > 0$ is a real-valued parameter that determines the spread of the
kernel. \Cref{fig:Examples} illustrates the measurement model described by \cref{eq:data_dirac} for the Gaussian and Ricker kernels. Gaussian kernels are used to model the point-spread function of diffraction-limited imaging systems such as microscopes~\cite{palm,zhang2007gaussian} and telescopes~\cite{aniano2011common}, blurring operators in computer vision~\cite{hummel1987deblurring} and the Green's function of systems governed by the diffusion equation~\cite{li2014heat,murray2015estimating}. The Ricker kernel is relevant to reflection seismology, a technique to estimate properties of the subsurface of the Earth 
by transmitting a short pulse into the ground and then measuring the reflected signal \cite{sheriff1995exploration}. Processing these measurements can be reduced to a one-dimensional deconvolution problem of the form~\eqref{eq:data_dirac}~\cite{pereg2017seismic}, where the signal $\mu$ represents the reflection coefficients at the interface between adjacent underground layers and the convolution kernel corresponds to the pulse, which can be modeled using a Ricker wavelet~\cite{ricker1953form}. 

\begin{figure}[tp]
\centering
\includegraphics{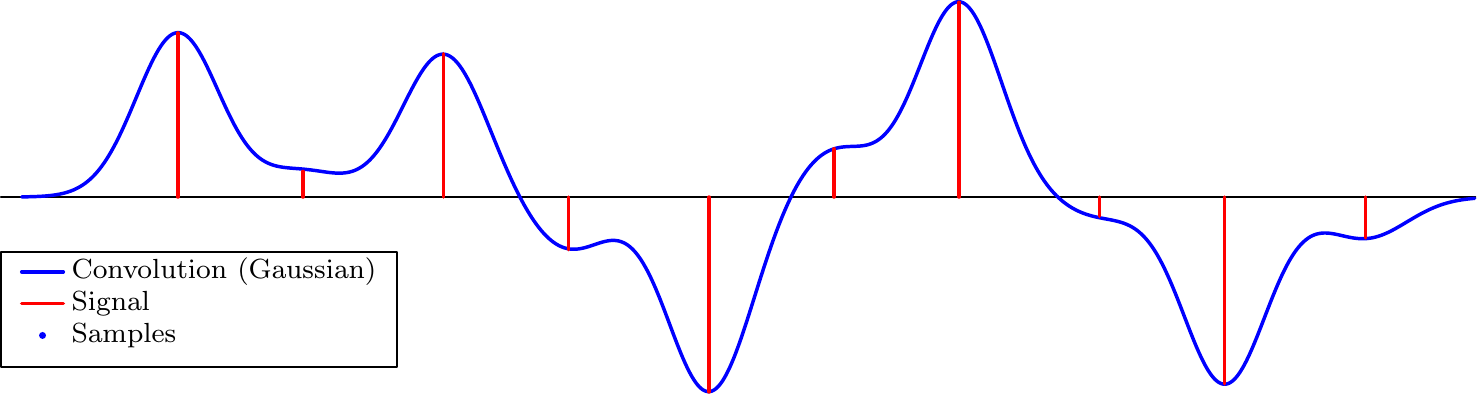}
\vspace{.5cm}
\includegraphics{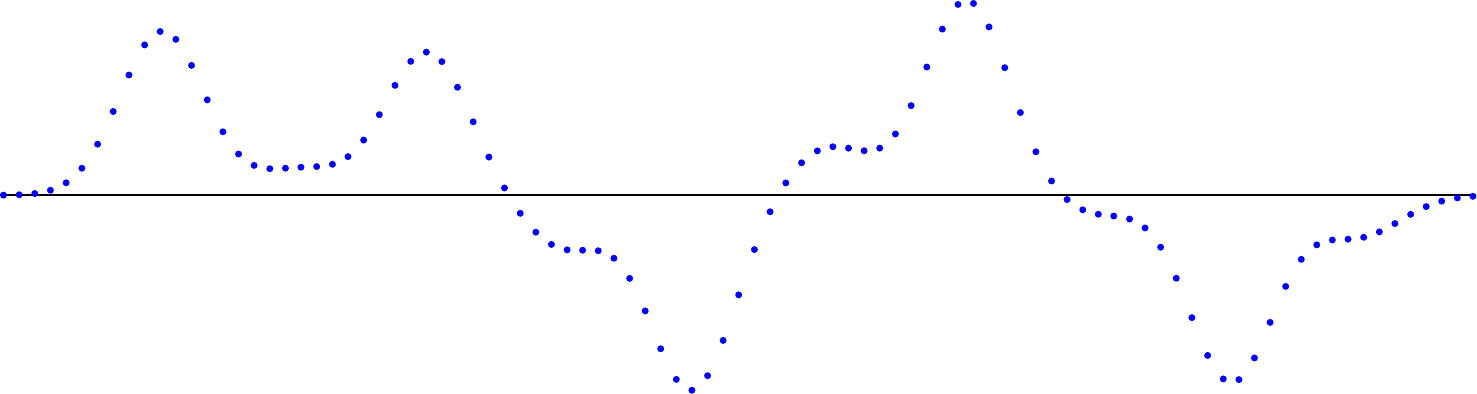}
\vspace{.5cm}
\includegraphics{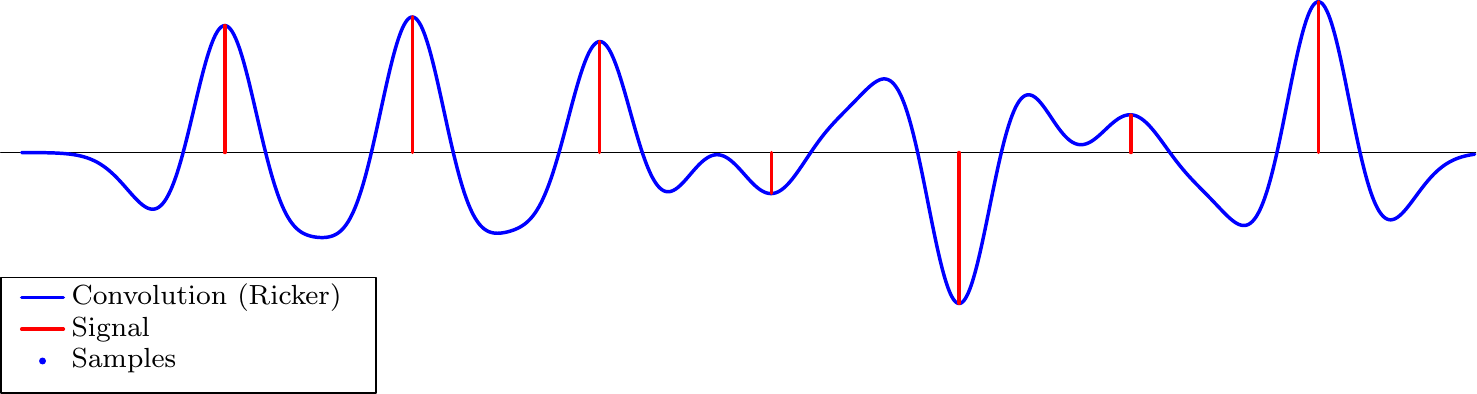} 
\vspace{.5cm}
\includegraphics{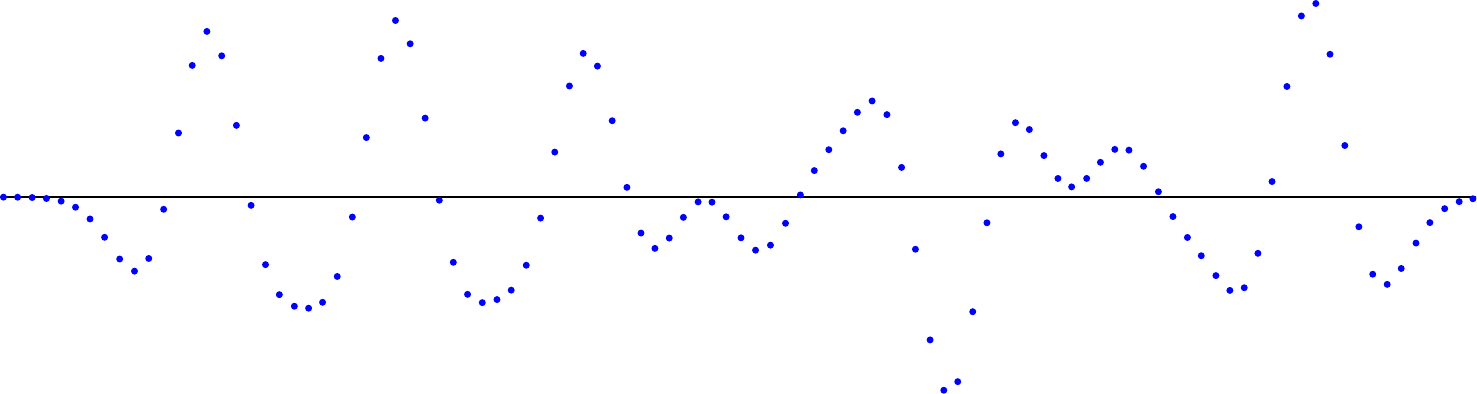}
\caption{Measurement model where a superposition of spikes is convolved with a Gaussian (top) or Ricker (bottom) kernel and then sampled to produce the data.}
\label{fig:Examples}
\end{figure}

In the 1970s and 1980s, geophysicists working on reflection seismology proposed to tackle the spike-deconvolution problem by solving a regularized least-squares problem incorporating an $\ell_1$-norm penalty to promote sparse solutions~\cite{taylor1979deconvolution,claerbout,levy,santosa,debeye1990lp}. Empirically, the method often recovers the spikes exactly from noiseless data-- this is the case for instance for the data shown in \Cref{fig:Examples}-- and produces robust estimates when the data are perturbed by noise. This optimization-based approach has also been successfully applied to deconvolution problems arising in other domains such as processing of marine seismic data~\cite{chapman1983deconvolution}, signal processing~\cite{drachman1984two}, ultrasound imaging~\cite{o1994recovery,bendory64stable} and estimation of initial conditions for systems governed by the heat equation~\cite{li2014heat}.   

In this paper we provide a theoretical framework to analyze deconvolution via $\ell_1$-norm minimization when the sampling pattern is nonuniform. In order to make our analysis independent of the discretization of the signal support, we consider a continuous counterpart of the $\ell_1$ norm, called the total-variation (TV) norm~\cite[Section~3.1]{folland2013real}, which should not be interpreted as the total variation of a piecewise-constant function~\cite{tv}, but rather as the analog of the $\ell_1$ norm for atomic measures. In fact, the TV norm of an atomic measure of the form $\sum_{i=1}^n b_i\delta_{t_i}$, $b_1, \ldots, b_n \in\RR$, equals the $\ell_1$ norm of the coefficients $ \sum_{i=1}^n \abs{b_i}$. Our goal is to characterize under what conditions minimizing the TV norm subject to measurement constraints is an accurate and robust deconvolution method. In the noiseless case, we consider the solution to the problem 
\begin{equation}
  \begin{aligned}
    \underset{\tilde{\mu}}{\op{minimize}} \quad& \normTV{ \tilde{\mu} } \\ \text{subject to}
    \quad& (K*\tilde{\mu})(s_i) = y_i, \quad i=1,\ldots,n.
  \end{aligned}
  \label{pr:TVnorm}
\end{equation}
Our main result is a sampling theorem for deconvolution via convex programming. We prove that solving Problem~\eqref{pr:TVnorm} achieves exact recovery as long as the signal support is not too clustered-- a condition which is necessary for the problem to be well posed-- and the set of samples contains at least two samples that are close to each spike. Our analysis holds for any nonuniform sampling pattern satisfying this condition, except for some pathological cases described in \Cref{sec:proximity}. 
In addition, the method is shown to be robust to dense additive noise in terms of support estimation. Finally, we illustrate the potential of this framework for incorporating additional assumptions by providing exact-recovery guarantees when the data are corrupted by impulsive perturbations. 
 
 The paper is organized as follows: In Section~\ref{sec:MainResults} we describe our main results and discuss related work. Section~\ref{sec:ExactRecoveryProof} is devoted to the proof of the deconvolution sampling theorem, which is based on a novel dual-certificate construction that is our main technical contribution. Sections~\ref{sec:RobustRecoveryProof} and~\ref{sec:SparseNoise} establish robustness to dense and sparse noise respectively. In Section~\ref{sec:numerical} we report the results of several experiments illustrating the numerical performance of our methods. We conclude the paper in Section~\ref{sec:conclusions}, where we discuss several future research directions. Code implementing all the algorithms discussed in this work is available online\footnote{\url{http://www.cims.nyu.edu/~cfgranda/scripts/deconvolution_simulations.zip}}.

\section{Main Results}
\label{sec:MainResults}

\subsection{Conditions on the Signal Support and the Sample Locations}


In this section we introduce conditions to characterize the class of signals and sampling patterns that we consider in our analysis of the point-source deconvolution problem. 

\subsubsection{Minimum Separation}
\label{sec:minsep}
A remarkable insight that underlies compressed-sensing theory is that randomized linear measurements preserve the energy of most sparse signals (more technically, they satisfy conditions such as the restricted isometry property~\cite{Candes:2005cs}). As a result, it is possible to robustly recover arbitrary sparse signals from such measurements, even if the problem is underdetermined~\cite{candes2006robust,donoho2006compressed}. Unfortunately, this is not the case for the deconvolution problem, unless the convolution kernel is random~\cite{hauptToeplitz,romberg2009compressive}. For smooth deterministic kernels, the data corresponding to different sparse signals can be almost identical, even if we have access to their full convolution with the kernel. Figure~\ref{fig:minsep_illposed} shows an example of this phenomenon, where a pair of distinct sparse signals become essentially indistinguishable after they are convolved with a Gaussian kernel (the same happens if we use a Ricker kernel instead). No algorithm would be able to tell the two signals apart from noisy samples of the convolution even at very high signal-to-noise ratios. The reason why the data are so similar becomes apparent in the frequency domain: most of the energy of the convolution kernel is concentrated in a low-pass band, whereas the energy of the difference between the signals is concentrated in the high end of the spectrum. Since convolution is equivalent to pointwise multiplication in frequency, the difference is suppressed by the measurement process. 

\begin{figure}[tp]
\begin{tabular}{ECC}
  & &Spectrum (magnitude)\\ \\
Signals & \includegraphics[scale=0.85]{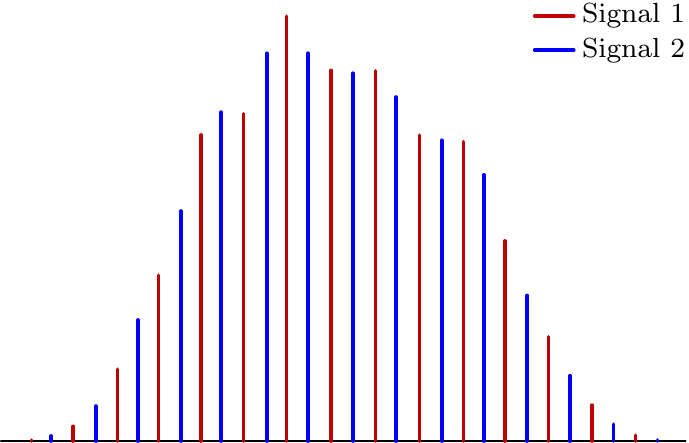} &
\includegraphics[scale=0.95]{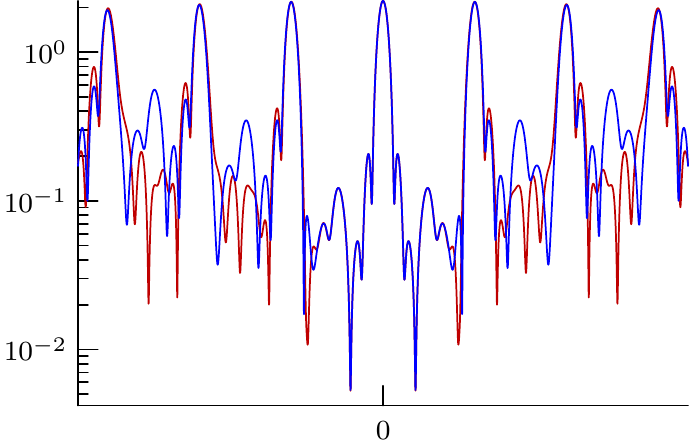}\\\\
Difference & \includegraphics[scale=0.85]{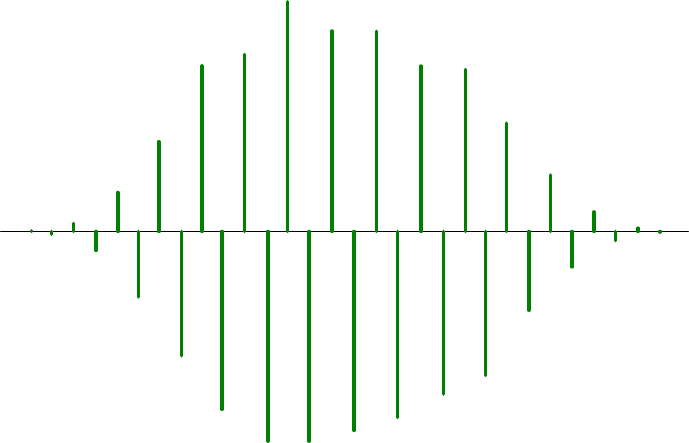} &
\includegraphics[scale=0.95]{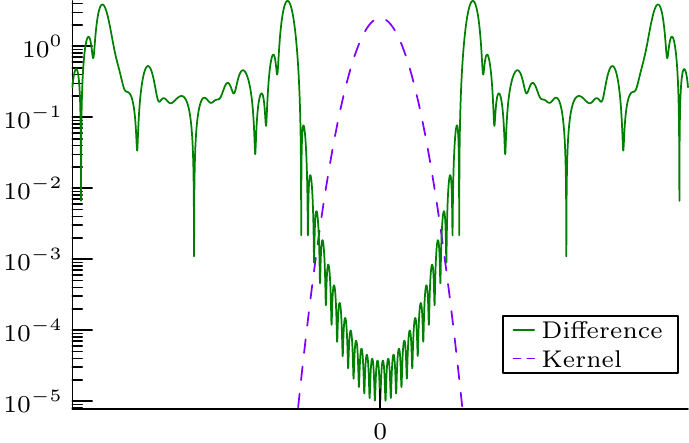}\\\\
Data & \includegraphics[scale=0.85]{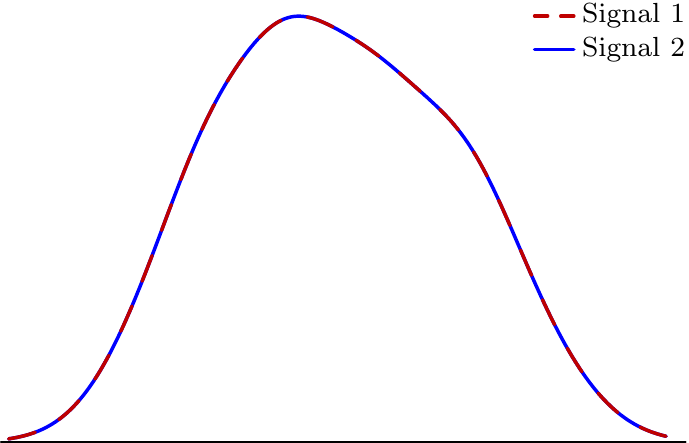} & \includegraphics[scale=0.95]{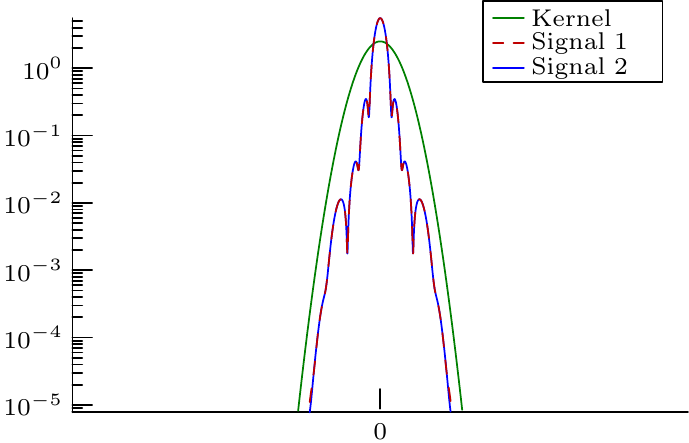}\\
\end{tabular}
\caption{The top row shows two signals with a minimum separation of $\sigma$ which have disjoint supports (no two spikes are closer than $\sigma/2$), but are very similar in the low-pass band of the spectrum. The difference between them is approximately high pass, which is problematic if the convolution kernel is approximately low pass (center row). The bottom row shows that the convolution of both signals with a Gaussian kernel are almost identical.}
\label{fig:minsep_illposed}
\end{figure}

In order to restrict our attention to a class of signals for which deconvolution is a well-posed problem for our convolution kernels of interest, we constrain the support to have a certain minimum separation.

\begin{definition}[Minimum Separation]
  \label{def:MinSep}
  The minimum separation of the support $T = \{t_1,\ldots,t_m\}$ of a signal is
\begin{align}
\label{eq:min_separation}
\Delta(T) = \min_{i\neq i'} \abs{ t_{i}-t_{i'}}.
\end{align} 
\end{definition}

This quantity was introduced in~\cite{superres} in the context of super-resolution of point sources from low-pass measurements (see also~\cite{donoho1992superresolution}). In that case, the minimum separation necessary to ensure that two signals do not produce very similar measurements can be characterized theoretically by applying Slepian's seminal work on prolate spheroidal sequences~\cite{slepian} (see also~\cite{moitra_superres} and Section 3.2 in \cite{superres}). In Section~\ref{sec:conditioning} we report numerical experiments indicating that for Gaussian and Ricker convolution kernels the minimum separation must be at least around $\sigma$ to avoid situations like the one in \Cref{fig:minsep_illposed}. 

\subsubsection{Sample Proximity}
\label{sec:proximity}

In order to estimate a signal of the form~\eqref{eq:signal} we need to determine two uncoupled parameters for each spike: its location and its amplitude. As a result, at least two samples per spike are necessary to deconvolve the signal from data given by~\eqref{eq:data_dirac}. In addition, the amplitude of our convolution kernels of interest decay quite sharply away from the origin, so most of the energy in the convolved signal corresponding to a particular spike is concentrated around that location. We therefore need two samples that are near each spike to estimate the signal effectively. To quantify to what extent this is the case for a fixed sampling pattern $S$ and signal support $T$ we define the \emph{sample proximity} of $S$ and $T$. In words, $S$ and $T$ have sample proximity $\gamma(S,T)$ if for each spike in $T$ there are at least two samples that are $\gamma(S,T)$ close to it. 

\begin{definition}[Sample Proximity]
  \label{def:ProxSep}
  Fix a set of sample locations $S$ and a signal support $T$.  Then
  $S$ has sample proximity $\gamma(S,T)>0$ if for every spike location
  $t_i\in T$ there exist two distinct sample locations $s,s' \in S$ such that
\begin{align}
\abs{t_i-s} \leq \gamma(S,T) \quad \text{and} \quad \abs{t_i-s'} \leq \gamma(S,T).
\end{align}
\end{definition}

To derive an upper bound on the sample proximity necessary for TV-norm minimization to identify a spike we consider a signal consisting of a single spike, $T:=\keys{t_1}$, and a sampling pattern with two samples $S
:=\keys{s_1,s_2}$, where $s_1:=t_1 - \gamma_0 $ and $s_2 :=
t_1+ \gamma_0 $. \Cref{fig:FarSamples} shows two alternative
explanations for these data: a single spike at $t_1$ or two spikes at
$s_1$ and $s_2$. For the Gaussian kernel, minimizing the TV norm chooses the
two-spike option when $\gamma_0  > \sqrt{2\log2}\sigma \approx
1.18\sigma$. For the Ricker wavelet, minimizing the TV norm chooses the
two-spike option when $\gamma_0 > 0.7811\sigma$. 

\begin{figure}[t]
  \centering 
  \includegraphics{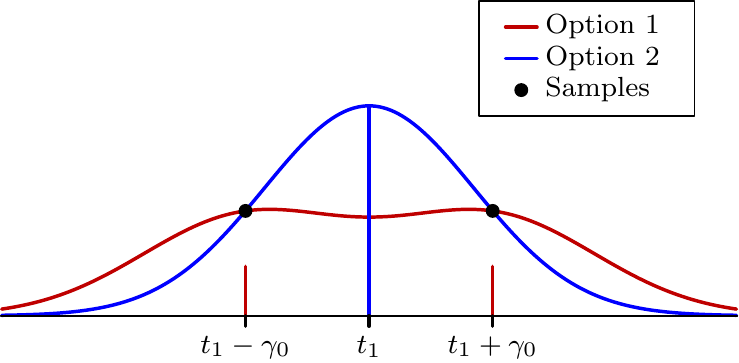}
  \includegraphics{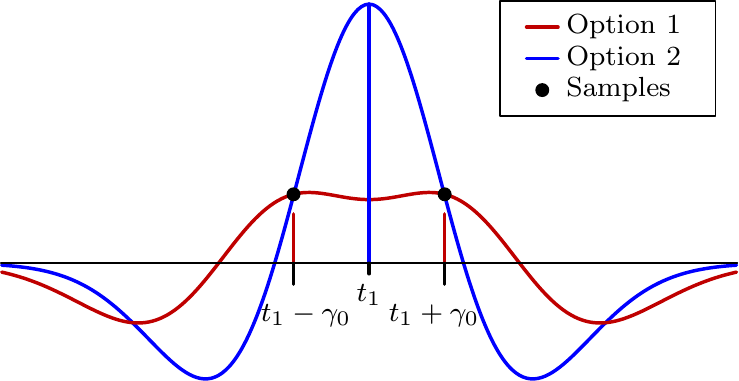}
  \caption{The two measurements in this example can be interpreted as samples from one spike situated at $t_1$ or two spikes situated at $t_1-\gamma_0$ and $t_1+\gamma_0$. When $\gamma_0$ is larger than $\sqrt{2\log2}\sigma$ for the Gaussian kernel (left) and $0.7811\sigma$ for the Ricker wavelet (right), the two-spike option has a smaller TV norm.}
  \label{fig:FarSamples}
\end{figure}

If we describe a sampling pattern only through its sample proximity, we allow for the following situation: there could be \emph{only} two samples close to a spike which also happen to be very close to each other. We cannot expect to derive robust-deconvolution guarantees for such patterns: a small perturbation in the measurements can have a dramatic effect on the estimated spike location \emph{for any recovery method} due to the local smoothness of the convolution kernels of interest. This is illustrated in \Cref{fig:NearSamples}. The good news is that this situation is pathological and would never arise in practice: placing the samples in this way requires knowing the location of the spikes beforehand! In order to rule out these patterns in our analysis, we define the sample separation associated to a given sample proximity and fix it to a small value.

\begin{definition}[Sample Separation]
  \label{def:SampSep}
  Fix a set of samples $S$ and a signal support $T$ with sample proximity $\gamma\brac{S,T}$.  $S$ and $T$ have sample separation $ \kappa(S,T)$ for every spike location $t_i \in T$ there exist at least two samples locations $s$ and $s'$  such that
\begin{align}
\abs{t_i-s} \leq \gamma(S,T),  \quad \abs{t_i-s'} \leq \gamma(S,T),\quad \text{and} \quad\abs{s-s'} \geq \kappa\brac{S,T}.
\end{align}
\end{definition}

\begin{figure}[t]
  \centering 
  \includegraphics[scale=0.7]{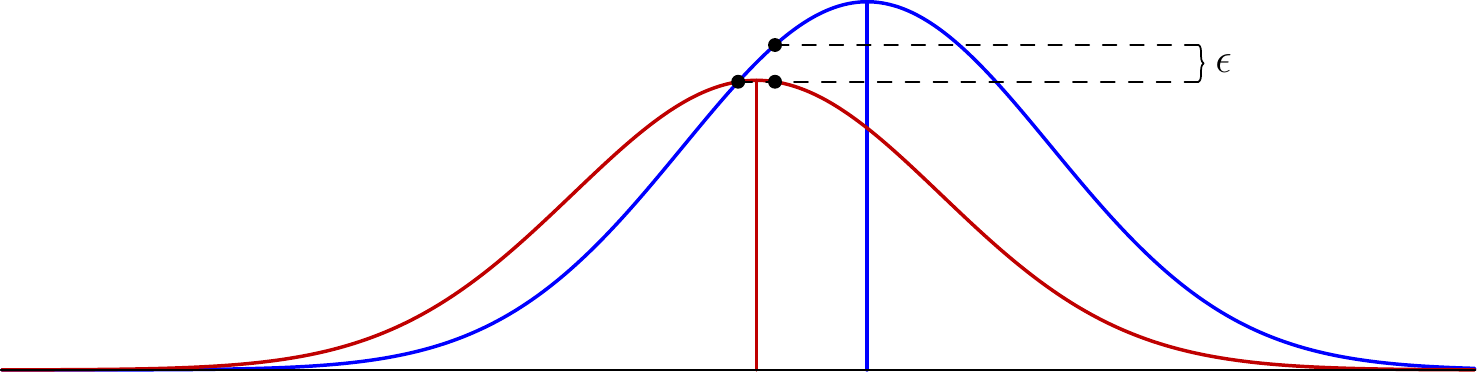} \vspace{0.5cm}
  \caption{A small sample separation may make it impossible to estimate the spike location robustly. In this example the original signal is shown in blue. A small perturbation of $\epsilon$ on one of the samples makes the red spike more plausible, producing a large error in the estimate of the support.}
  \label{fig:NearSamples}
\end{figure}

\subsection{Sampling Theorem for Exact Recovery}
\label{sec:exact_recovery}

Our main result is a sampling theorem for deconvolution via convex optimization from nonuniform samples. We show that solving problem~\eqref{pr:TVnorm} achieves exact recovery for the Gaussian and Ricker kernels under minimum-separation and sample-proximity conditions on the support and sample locations.

\begin{theorem}\label{thm:Exact}
Let $\mu$ be a signal defined by~\eqref{eq:signal} and assume that the data are of the form~\eqref{eq:data_dirac}, where $K$ is the Gaussian kernel or the Ricker wavelet. Then $\mu$ is the unique solution to problem~\eqref{pr:TVnorm}, as long as the signal support $T$ and the sampling pattern $S$ have a minimum separation $\Delta(T)$ and a sample proximity $\gamma(S,T)$ lying in the orange region depicted in the corresponding plot of \Cref{fig:ExactRecoveryRegion}, and the sample separation $\kappa(S,T)$ is greater than $\sigma/20$.
\end{theorem}
As explained in \Cref{sec:proximity}, we fix a small sample separation to rule out pathological sampling patterns. Except for such cases, the result holds for \emph{any nonuniform sampling pattern} satisfying the sample-proximity condition: among the samples there must be two close to each spike, but the rest can be situated at any arbitrary location. The proof of the theorem, presented in \Cref{sec:ExactRecoveryProof}, relies on a novel construction of a dual certificate, which is our main technical contribution and can be applied to derive recovery guarantees for other convolution kernels.

\begin{figure}[bp]
\centering
\includegraphics{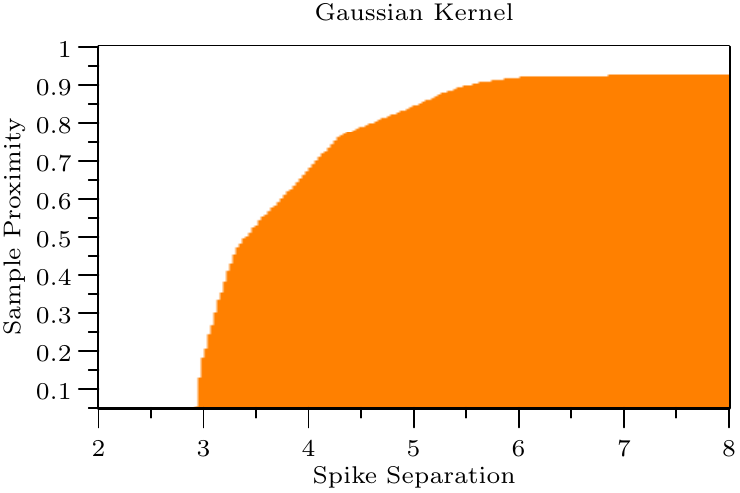}\qquad
\includegraphics{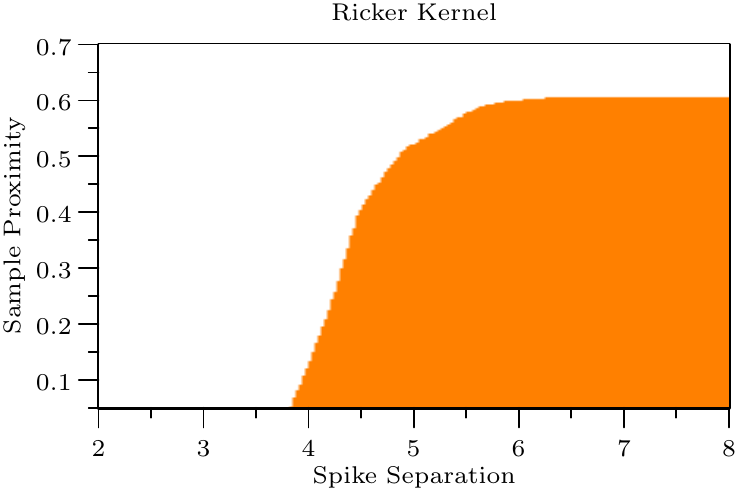}
\caption{We prove exact recovery for the region marked in orange. The sample-proximity and minimum-separation samples are in units of $\sigma$.}
\label{fig:ExactRecoveryRegion}
\end{figure}

In order to evaluate the accuracy of our theoretical analysis, we compare our results to the numerical performance of the method for a sampling pattern containing only two samples close to each spike placed at different sample proximities. This sampling pattern is not realistic, because placing only two samples close to every spike requires knowing the support of the signal. However, it provides a worst-case characterization of the method for a given sample proximity, since incorporating additional samples can only increase the chances of achieving exact recovery. The simulation, described in more detail in \Cref{sec:numexactrecovery}, confirms that the method succeeds as long as the minimum separation is not too small and the sample proximity is not too large. \Cref{fig:ProofVsCVX} shows that the numerical exact-recovery region contains the theoretical exact-recovery region established by \Cref{thm:Exact}. 

\begin{figure}[tp]
\centering
\includegraphics{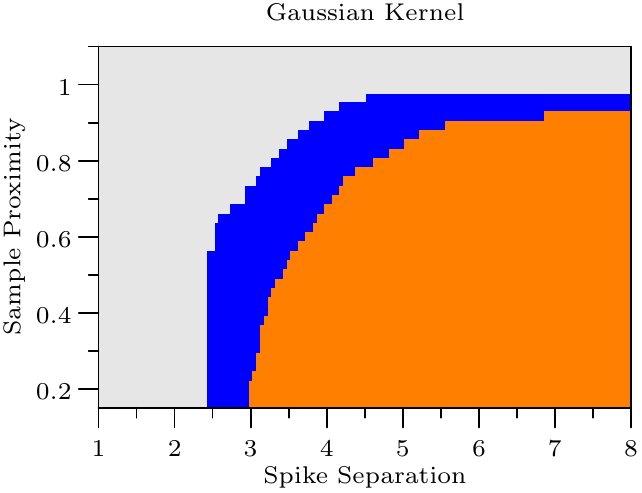}\quad
\includegraphics{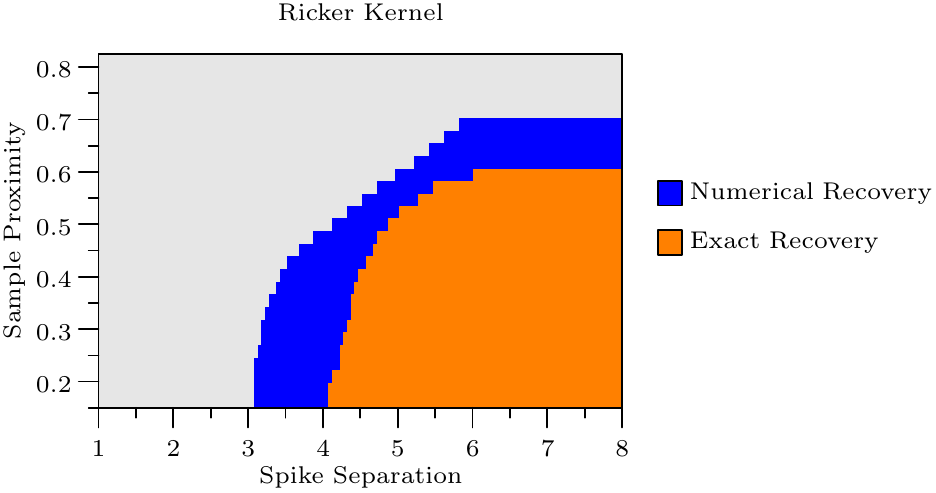}
\caption{Comparison of the region where we prove exact recovery
  (orange) with the region where we observe exact recovery in the
  numerical experiment using 60 spikes described in \Cref{fig:CVXExactRecovery}
  (blue). Every blue point indicates that exact recovery occurs for
  all larger minimum separations and smaller sample proximities. 
  }
\label{fig:ProofVsCVX}
\end{figure}

Our analysis focuses on the continuous TV-norm minimization
problem. In practice, this can be implemented by discretizing the
domain and then minimizing the $\ell_1$-norm of the discretized
estimate. \Cref{thm:Exact} implies that this procedure succeeds if the
original signal is supported on the discretization grid.
\bdb{As we explain at the end of \Cref{sec:RobustnessDense}
  the techniques we present to derive robustness guarantees provide
  some control over the error incurred
  if the signal is not
supported on the grid.  Performing a more detailed analysis of
discretization error is an interesting
direction for future research.}

\begin{corollary}
  \label{cor:L1Exact}
  Assume that the support of the measure $\mu$ in \cref{eq:signal}
  lies on a known discretized grid $G\subset\RR$. Then as long as its support $T$ and the set of samples $S$
  satisfy the conditions of \Cref{thm:Exact}, the true coefficients $a_1$, \ldots, $a_{\abs{G}}$ are the unique solution to the $\ell_1$-norm minimization problem
  \begin{equation}
  \begin{aligned}
    \underset{\tilde{a}\in\RR^{|G|}}{\op{minimize}} \quad& \normOne{ \tilde{a} } 
    \\ \text{subject to}
    \quad& \sum_{t_j\in T} \tilde{a}_jK(s_i-t_j) = y_i, \quad i=1,\ldots,n.
  \end{aligned}
  \label{pr:TVnormDisc}
\end{equation}
\end{corollary}
\begin{proof}
  Problem~\eqref{pr:TVnormDisc} is
  equivalent to problem~\eqref{pr:TVnorm} restricted to measures supported on $G$.
  As a result, any respective solutions $\hat{a}$ and $\hat{\mu}$ must satisfy
  $\normOne{\hat{a}}\geq\normTV{\hat{\mu}}$.  By \Cref{thm:Exact},
  \eqref{pr:TVnorm} is uniquely minimized by $\mu$.
  Since $\mu$ is supported on $T$, $a$ is the unique solution of \eqref{pr:TVnormDisc}.
\end{proof}
When the original spike does not lie exactly on the grid the approximation error can be controlled to some extent using the results on additive dense noise in the following section. A more specific analysis of the effect of discretization in the spirit of~\cite{duval2017sparse} is an interesting topic for future research.

\subsection{Robustness to dense noise}
\label{sec:RobustnessDense}
\begin{figure}[tp]
\centering
\includegraphics{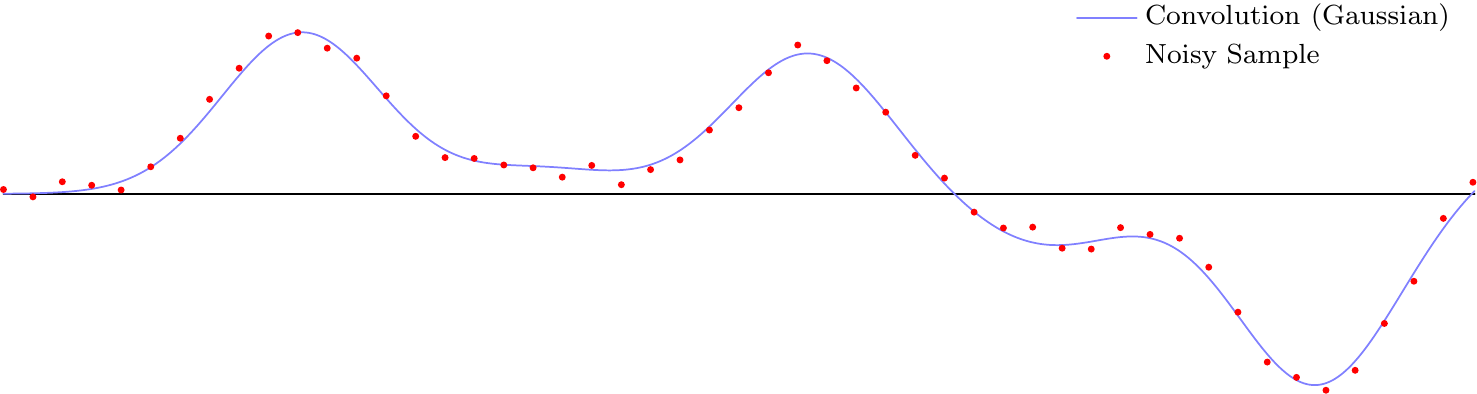}\vspace{.25cm}
\includegraphics{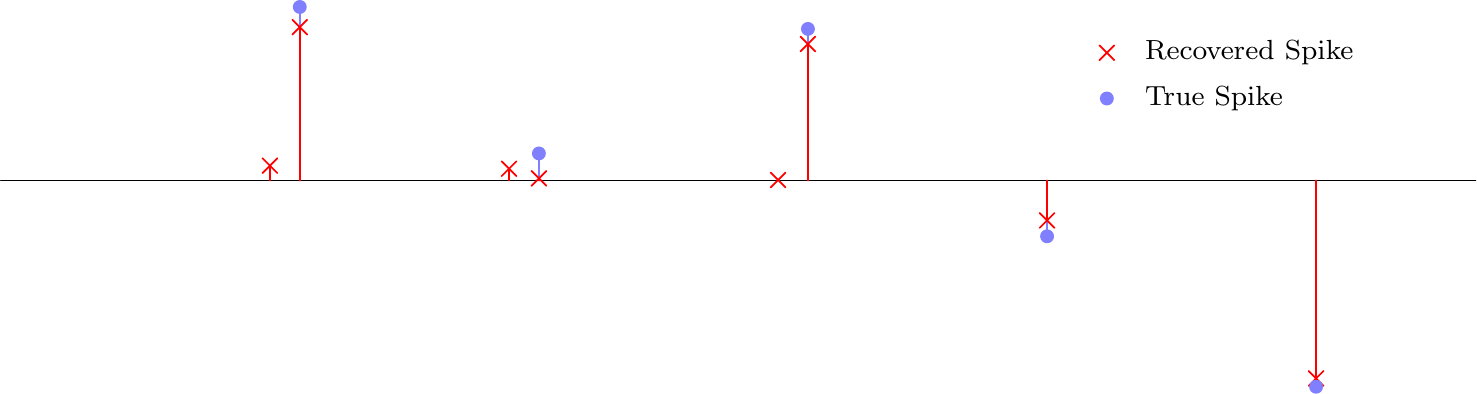}\vspace{.25cm}
\includegraphics{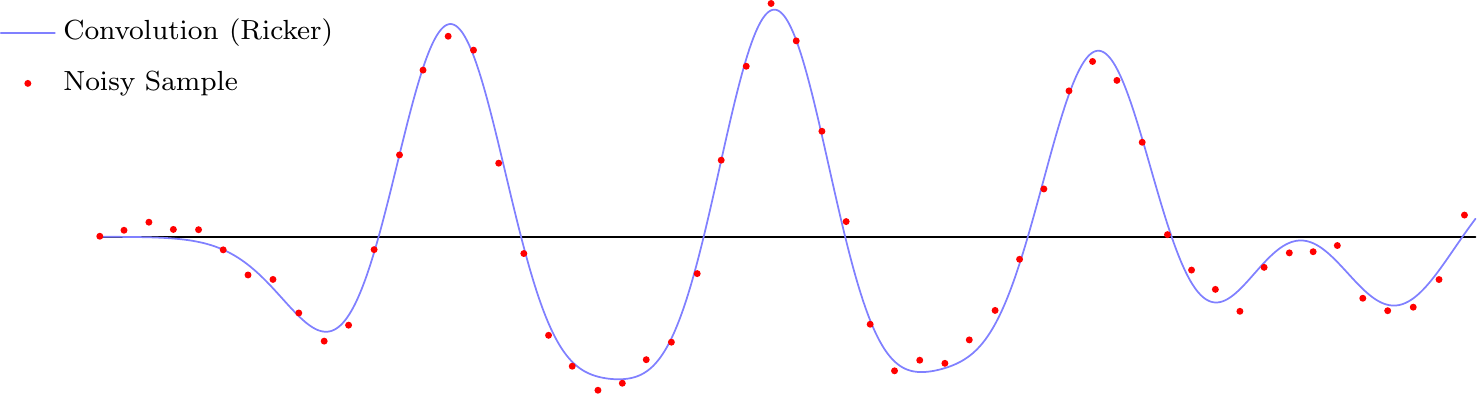}\vspace{.25cm}
\includegraphics{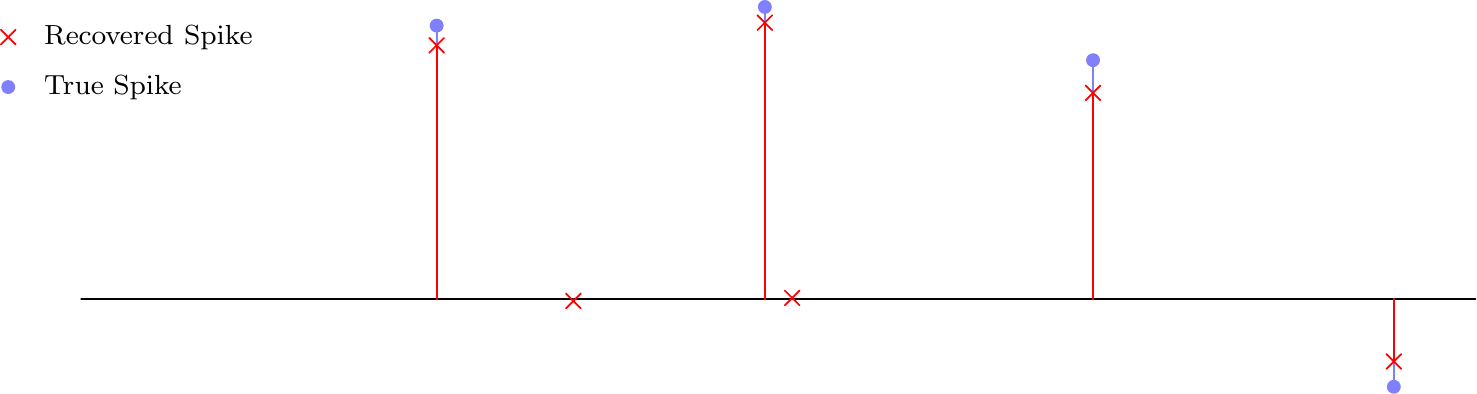}
\caption{Deconvolution from noisy samples (in red) with a
  signal-to-noise ratio of 21.6 dB for the Gaussian kernel (top) and
  20.7 dB for the Ricker wavelet (bottom). The noise is iid
  Gaussian. The recovered signals are obtained by solving
  \cref{pr:NoiseRecovery} on a fine grid which contains the location
  of the original spikes. Here we set $\bar{\xi}:=1.25\xi$.}
\label{fig:NoiseExperiment}
\end{figure}

In this section we consider the deconvolution problem when additive noise perturbs the data,
\begin{align}
y_i := \brac{ K \ast \mu}\brac{s_i } + z_i, \quad i = 1,2,\ldots,n. \label{eq:noisy_data}
\end{align}
Our only assumption on the perturbation $z \in \RR^n$ is that it has bounded $\ell_2$ norm, i.e. $\normTwo{z}< \xi$ for some noise level $\xi$. Otherwise, the error is arbitrary and can be adversarial. To adapt the TV-norm minimization problem~\eqref{pr:TVnorm} to this scenario we relax the data-consistency constraint from an equality to an inequality involving a known upper bound on the noise level $\bxi \geq \xi$,
\begin{equation}
\label{pr:NoiseRecovery}
\begin{aligned}
  \underset{\tilde{\mu}}{\op{minimize}} \quad& \normTV{ \tilde{ \mu } }\\ 
  \text{subject to} \quad & \sum_{i=1}^{n}\brac{ y_i -
    (K*\tilde{\mu})(s_i)}^2 \leq \bxi^2.
\end{aligned}
\end{equation}
For real analytic kernels that decay at infinity, which includes the Gaussian kernel and the Ricker wavelet, the solution to this problem is atomic.
\begin{restatable}[Proof in
    \Cref{sec:NoiseAtomicProof}]{lemma}{NoiseAtomic}
\label{lem:NoiseAtomic}
Assume $y$ satisfies \eqref{eq:noisy_data} with $\normTwo{z}<\bxi$.
If $K$ is real analytic with $K(t)\to 0$ as $t\to\infty$ then the
solution $\hat{\mu}$ to problem~\eqref{pr:NoiseRecovery} is an atomic
measure with finite support $\widehat{T} \subseteq\RR$ of the form
\begin{align}
\hat{\mu} & = \sum_{\hat{t}_k\in\widehat{T}}
\hat{a}_k\delta_{\hat{t}_k}, \quad \hat{a}_1,
\ldots,\hat{a}_{\abs{\widehat{T}}} \in\RR.
\end{align}  
\end{restatable}
Empirically, we observe that the estimated support is an accurate estimate of the true support when the noise is not too
high, as illustrated in \Cref{fig:NoiseExperiment}. The following theorem shows that this is guaranteed to be the case under the assumptions of
\Cref{thm:Exact}.

\begin{theorem}[Proof in \Cref{sec:RobustRecoveryProof}]
\label{thm:RobustRecovery}
Let $\mu$ be a signal defined by~\eqref{eq:signal} and assume that the data are of the form~\eqref{eq:noisy_data}, where $K$ is the Gaussian kernel or the Ricker wavelet and  $\normTwo{z}< \xi$. If the sample separation $\kappa(S)$ is at least $\sigma/20$ and the minimum separation and sample proximity lie in the exact-recovery regions depicted in
  \Cref{fig:ExactRecoveryRegion}, the solution
  \begin{equation}
    \label{eq:denseproofmu}
    \hat{ \mu } := \sum_{\hat{t}_k\in\hat{T}}
    \hat{a}_k\delta_{\hat{t}_k}
  \end{equation}
  to problem~\eqref{pr:NoiseRecovery} has the following properties:
\begin{align}
\abs{ a_j - \sum_{\keys{\hat{t}_l\in\hat{T}: |
      \hat{t}_l-t_j|\leq \nearNoise\sigma } }\hat{a}_l } & \leq C_1\bxi\sqrt{|T|} \quad
\text{for all $t_j\in T$,} \label{eq:RobustRecovery_1}
\\ 
\sum_{\keys{\hat{t}_l\in\hat{T},t_j\in T:
    |\hat{t}_l-t_j|\leq \nearNoise\sigma } } \abs{ \hat{a}_l} \brac{
  \hat{t}_l-t_j}^2 & \leq C_2\bxi\sqrt{|T|} , \label{eq:RobustRecovery_2}
\\ 
\sum_{\keys{\hat{t}_l\in\hat{T}:
    |\hat{t}_l-t_j|> \nearNoise\sigma,\forall t_j\in T } } \abs{ \hat{a}_l } & \leq C_3\bxi\sqrt{|T|}
, \label{eq:RobustRecovery_3}
  \end{align}
  where $\nearNoise,C_1,C_2,C_3>0$ depend only on $\Delta(T)$, $\gamma(S,T)$,
  and $\kappa(S)$.  Upper bounds on the values of $\nearNoise$ are given in
  \Cref{fig:C0Bounds}.  
\end{theorem}
\begin{figure}[t]
\centering 
\includegraphics{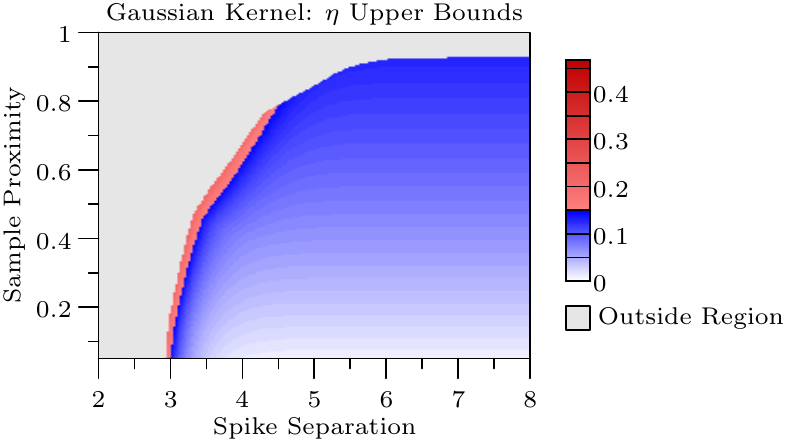}
\includegraphics{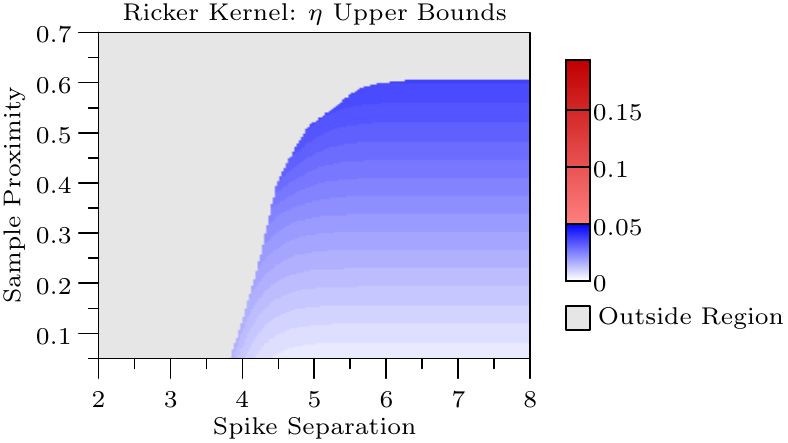}
  \caption{Upper bound on the values of $\nearNoise$ in
    \Cref{thm:RobustRecovery} for the corresponding values of the
    sample proximity and minimum separation.  Sample proximity and
    spike separation are given in units of~$\sigma$.}
  \label{fig:C0Bounds}
\end{figure}

Property \eqref{eq:RobustRecovery_1} implies that the amplitude of each spike is well
approximated by the sum of the estimated spikes close to it, where
\emph{close} means in a radius of $\nearNoise \sigma< \Delta(T)/2$. $\nearNoise$ is
smaller than $0.15$ for the Gaussian kernel and $0.05$ for the Ricker
kernel for most values of the minimum separation and sample proximity
for which exact recovery occurs, as shown in
\Cref{fig:C0Bounds}. Property \eqref{eq:RobustRecovery_2} implies that the estimated support
clusters around the support of the true signal. 
Property \eqref{eq:RobustRecovery_3} shows that
spurious spikes detected away from the support have small
magnitude. An interesting consequence of the theorem is that the error
incurred by the method at each spike can be bounded by a quantity that
does not depend on the amplitude of other spikes.
\begin{corollary}
  \label{cor:densesupport}
  Under the conditions of Theorem~\ref{thm:RobustRecovery}, for any
  element $t_i$ in the support of $\mu$ such that $|a_i|>C_1\bxi\sqrt{|T|}$
  there exists an element $\hat{t}_i$ in the support of the estimate
  $\hat{ \mu }$ satisfying
  \begin{equation}
    \label{eq:RobustRecovery_Supp}
    |t_i-\hat{t}_i| \leq \sqrt{\frac{C_2\bxi\sqrt{|T|}}{|a_i|-C_1\bxi\sqrt{|T|}}}.
  \end{equation}
\end{corollary}
These guarantees demonstrate that the method is robust in a small-noise regime: the error tends to zero as the noise decreases. This occurs under the same conditions that we require for exact recovery. In particular, we only assume that the sampling pattern has two samples close to each spike, but we do not account for the rest of the samples. Characterizing how additional samples improve the estimation accuracy of the algorithm in the presence of noise is an interesting problem for future research. 

\bdb{
  The results of this section allow us to control the error of our
  recovered signal when the measurements are corrupted by noise.
  Another source of measurement error is due to discretization.
  As mentioned earlier, in practice we often solve the continuous TV-norm
  minimization problem by discretizing and solving an $\ell_1$-norm
  minimization problem.  In general the true signal does not
  necessarily lie on the grid, which means the model is not completely accurate.
  We can account for this by solving the following discretized
  version of problem~\eqref{pr:NoiseRecovery}
  \begin{equation}
    \label{pr:L1NoiseRecovery}
    \begin{aligned}
      \underset{\tilde{a}\in\RR^{|G|}}{\op{minimize}} \quad&
      \normOne{ \tilde{a} }\\ 
      \text{subject to} \quad & \sum_{i=1}^{n}\brac{ y_i -
        \sum_{g_j\in G}\tilde{a}_jK(s_i-g_j)}^2 \leq \bxi^2,
    \end{aligned}
  \end{equation}
  where $G$ is some fixed discretization grid.  By bounding the effect
  of shifting the true measure $\mu$ onto $G$
  we can obtain results analogous to
  \Cref{thm:RobustRecovery} and \Cref{cor:densesupport} for
  $\bxi\geq L\eta\normTV{\mu}\sqrt{|S|}$. Here $\eta$ is the grid
  separation and $L$ is the Lipschitz constant of the kernel.
  Improving these somewhat crude error bounds is an
  interesting direction for future work.
}
\subsection{Sparse Noise}
\label{sec:sparse_noise_results}
In this section we consider data contaminated by outliers, where some of the samples are completely corrupted by impulsive noise,
\begin{align}
y_i := \brac{ K \ast \mu}\brac{s_i } + w_i \quad i = 1,2,\ldots,n. \label{eq:sparse_noise}
\end{align}
Here $w \in \R^n$ is a sparse vector with arbitrary amplitudes. In order to account for the presence of the sparse perturbation, we incorporate an additional variable to Problem~\eqref{pr:NoiseRecovery} and add a corresponding sparsity-inducing penalty term to the cost function:
\begin{equation}
  \begin{aligned}
    \underset{\tilde{\mu},\tilde{w}}{\op{minimize}} \quad& \normTV{ \tilde{\mu} } +\lambda\normOne{\tilde{w}}\\ 
    \text{subject to}
    \quad& (K*\tilde{ \mu})(s_i) +\tilde{w}_i = y_i, \quad i=1,\ldots,n,
  \end{aligned}
  \label{pr:SparseNoiseRecovery}
\end{equation}
where $\lambda >0$ is a regularization parameter. An equivalent approach was previously proposed in~\cite{gholami2012fast} without theoretical guarantees. In~\cite{fernandez2016demixing} an analogous formulation is applied to spectral super-resolution in the presence of outliers, where the impulsive noise occurs in the spectral domain and is consequently very incoherent with the signal. 
In our case, the sparse noise and the clean samples are much less incoherent: most of the energy corresponding to a given spike in the signal is concentrated in a few samples around it, which could be mistaken for sparse noise. For the Gaussian and Ricker kernels, we observe that solving problem~\eqref{pr:SparseNoiseRecovery} is effective as long as there are not too many contiguous noisy samples. \Cref{fig:sparseNoise} shows an example where the solution to the optimization problem estimates both the signal and the sparse corruptions exactly. 

\begin{figure}[tp]
\centering
\includegraphics{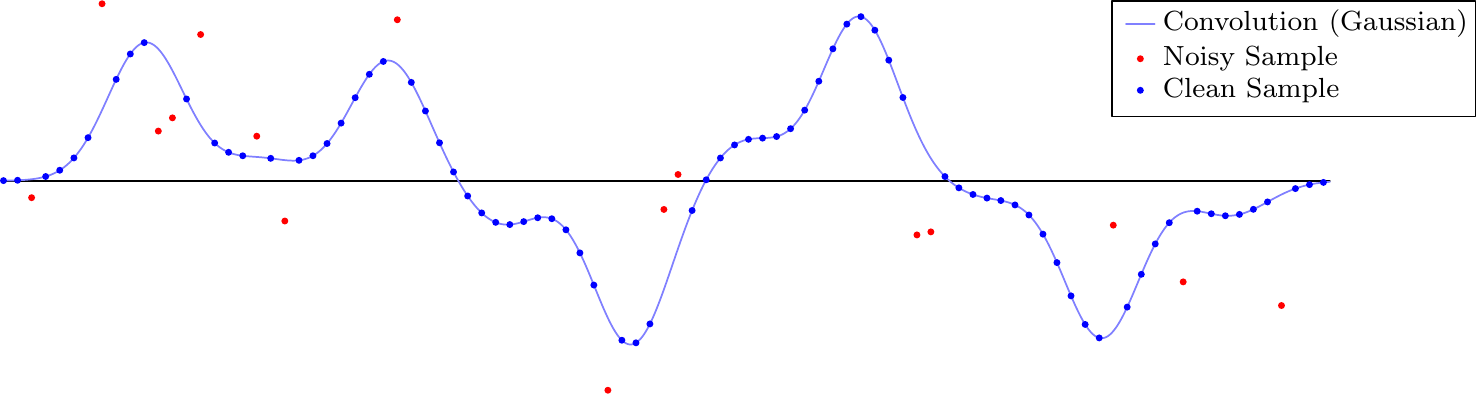}\vspace{.5cm}
\includegraphics{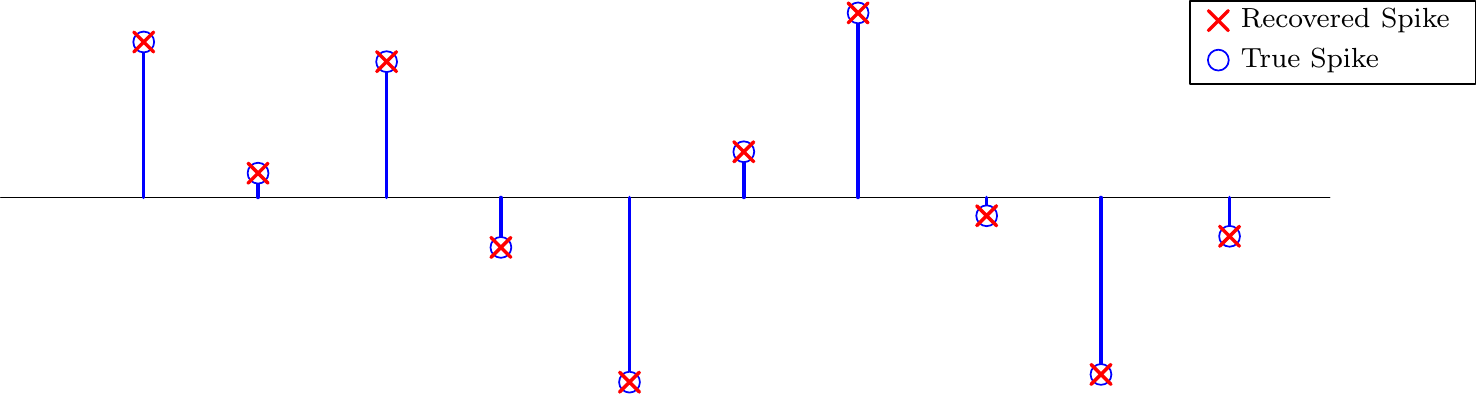}\vspace{.5cm}
\includegraphics{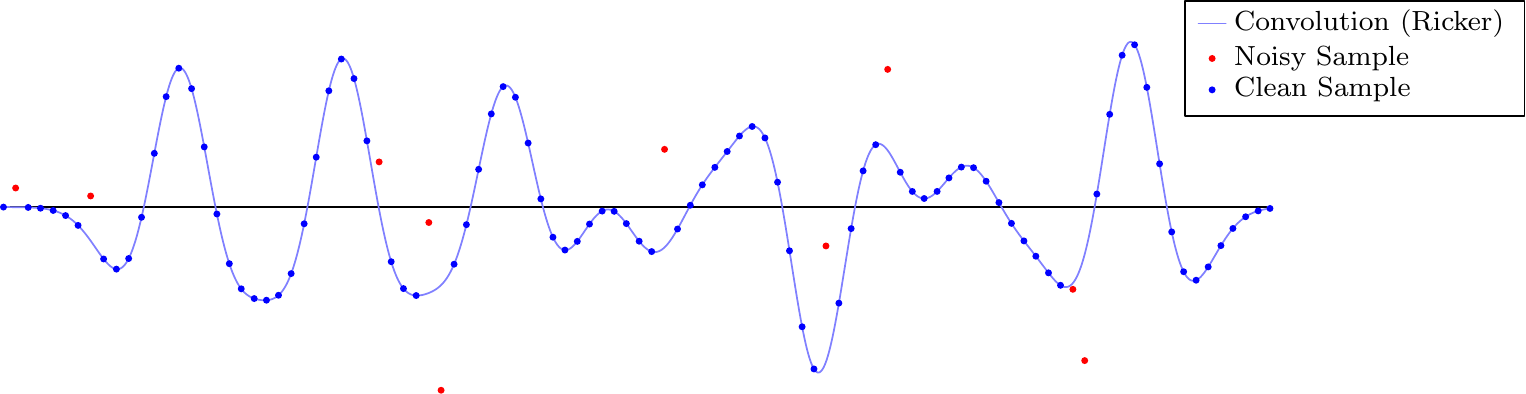}\vspace{.5cm}
\includegraphics{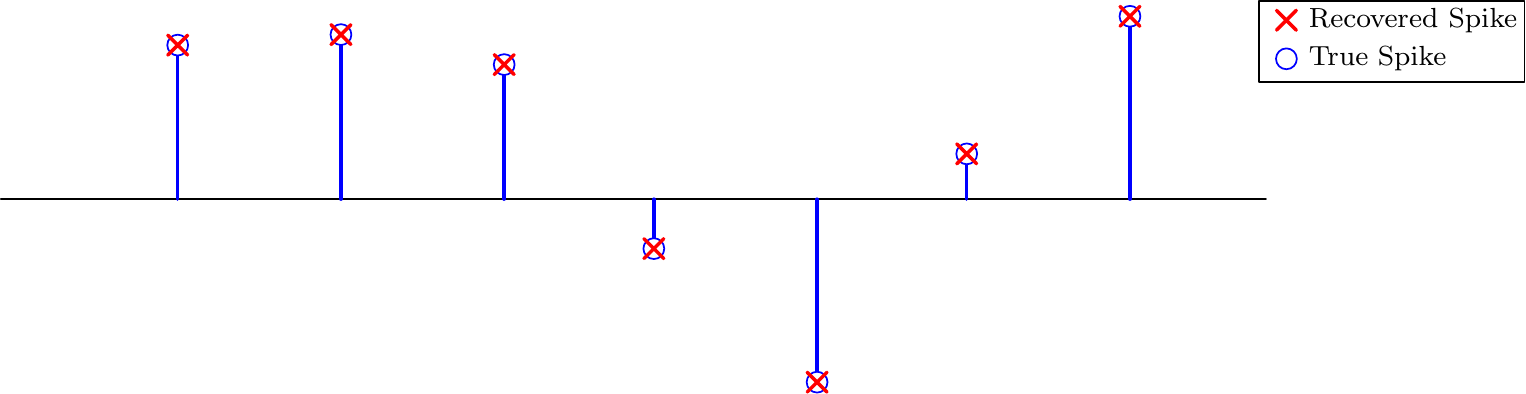}
\caption{Deconvolution from convolution samples (blue) corrupted by large-amplitude impulsive noise (red) for the Gaussian kernel (top) and the Ricker wavelet (bottom). The recovered signals are obtained by solving \cref{pr:SparseNoiseRecovery} on a fine grid, which achieves exact recovery.}
\label{fig:sparseNoise}
\end{figure}
The following theorem provides exact-recovery guarantees in the
presence of sparse noise. For simplicity, we assume that the samples
lie on a regular grid with step size $\tau$, so that $S :=
\{s_1,\ldots,s_n\}$ where $s_{i+1}=s_i+\wid$, instead of requiring a
sample-proximity condition. In addition, we impose a
minimum-separation condition on the support of the signal, as well as
on the set of noise locations $\Noi\subseteq S$ to preclude it from being too clustered. Finally, we require that the samples surrounding the spikes and the sparse noise, denoted by
\begin{align}
  \cI &:= \{s_i\in S\setminus\Noi\mid \exists t_j\in T, s_i\leq t_j < s_{i+1}\}\cup
  \{s_{i+1}\in S\setminus\Noi\mid \exists t_j\in T, s_i\leq t_j < s_{i+1}\},\\
  \cC &:= \{s_{i-1}\in S\setminus\Noi \mid s_i\in\Noi\}\cup
  \{s_{i+1}\in S\setminus\Noi \mid s_i\in\Noi\}, 
\end{align}
respectively, be unaffected by the noise and non-overlapping (so that there are at least two clean samples between every corruption and every spike).

\begin{theorem}[Proof in \Cref{sec:SparseNoise}]
\label{thm:SparseNoiseRecovery}
Let $\mu$ be a signal defined by~\eqref{eq:signal} and $y$ be data of the form~\eqref{eq:sparse_noise}, where $K$ is the
Gaussian kernel and the samples lie on a grid with step size
$\wid\in[\botwidG\sigma,\topwidG\sigma]$. Assume the spikes and
samples are both surrounded by clean samples, so that:
\begin{equation}
\label{eq:sparse_conditions}
  |\II| = 2|T|,\quad |\cC|=2|\Noi|,\quad \text{and}\quad \cC\cap\II = \emptyset.
\end{equation}
If $\mu$ has minimum
separation $\Delta(T)\geq\DeltaG\sigma$, and the noise locations $\Noi$
have minimum separation $\Delta(\Noi)\geq\DeltaG\sigma$ 
then solving problem~\eqref{pr:SparseNoiseRecovery}
with parameter $\lambda:=2$ recovers $\mu$ and $w$ exactly.
\end{theorem}
\begin{theorem}[Proof in \Cref{sec:SparseNoise}]
\label{thm:SparseNoiseRecoveryRicker}
Let $\mu$ be a signal defined by~\eqref{eq:signal} and $y$ be data of the form~\eqref{eq:sparse_noise}, where $K$ is the Ricker wavelet and the samples lie on a grid with step size $\wid\in[\botwidR\sigma,\topwidR\sigma]$. If~\eqref{eq:sparse_conditions} holds, $\mu$ has minimum separation $\Delta(T)\geq\DeltaR\sigma$, and the
noise locations $\Noi$ have minimum separation
$\Delta(\Noi)\geq\DeltaR\sigma$ then solving problem~\eqref{pr:SparseNoiseRecovery}
with parameter $\lambda:=2$ recovers $\mu$ and $w$ exactly.
\end{theorem}

\Cref{sec:sparse_recovery} reports the results of several numerical
experiments showing that the method achieves exact recovery under
weaker assumptions than those in \Cref{thm:SparseNoiseRecovery,thm:SparseNoiseRecoveryRicker}. Characterizing the performance of the method more generally is an interesting research problem. These experiments also suggest that the approach is quite robust to the choice of the regularization parameter $\lambda$. The following lemma, which is based on Theorem 2.2 in~\cite{candes2011robust}, provides some theoretical justification: it establishes that if the method recovers a signal and a sparse noise realization for a certain value of $\lambda$, then it also achieves exact recovery for data corresponding to any trimmed versions of the signal or the noise for that same value of $\lambda$. Note that $\Omega$ denotes the nonzero entries of $w$, as opposed to $\Noi$ which denotes the corresponding sample locations.  

\begin{restatable}[Proof in
    \Cref{sec:SparseTrimProof}]{lemma}{SparseTrim}
\label{lem:SparseTrim}
  Let $w$ be a vector with support $\Omega$ and let $\mu$ be an
  arbitrary measure such that 
  \begin{equation}
    y_i = (K*\mu)(s_i)+w_i,
  \end{equation}
  for $i=1,\ldots,n$.  Assume that the pair $(\mu,w)$ is the unique
  solution to Problem~\eqref{pr:SparseNoiseRecovery} for the data $y$,
  and consider the
  data
  \begin{equation}
    y_i' = (K*\mu')(s_i)+w_i',
  \end{equation}
  for $i=1,\ldots,n$.  Here $\mu'$ is a trimmed version of $\mu$: it
  is equal to $\mu$ on a subset $T'\subseteq T$ of its support.
  Similarly, the support $\Omega'$ of $w'$ satisfies
  $\Omega'\subseteq\Omega$.  For any choice of $T'$ and $\Omega'$, the
  pair $(\mu',w')$ is the unique solution to
  Problem~\eqref{pr:SparseNoiseRecovery} if we set the data vector to
  equal $y'$ for the same value of $\lambda$.  
\end{restatable}

\subsection{Related work}

There are two previous works that derive exact recovery guarantees for the Gaussian kernel in one dimension. In~\cite{schiebinger2015superresolution} the authors prove exact deconvolution of $k$ spikes from samples at $2k+1$ arbitrary locations using a weighted version of the TV norm that avoids having to impose a sample proximity condition. The result only holds for positive spikes, but does not require a minimum-separation condition. In~\cite{vetterli2002sampling} the authors prove exact recovery of $k$ spikes from $2k$ samples lying on a uniform grid, also without a minimum-separation condition. The method can be interpreted as an extension of Prony's method~\cite{deProny:tg}. In both cases, no robustness guarantees are provided (such guarantees would require conditions on the support of the signal, as explained in \Cref{sec:minsep}). 

As mentioned in the introduction, to the best of our knowledge $\ell_1$-norm minimization for deconvolution was originally proposed by researchers in geophysics from the 1970s~\cite{taylor1979deconvolution,claerbout,levy,santosa,debeye1990lp}. Initial theoretical works focused on analyzing random convolution kernels~\cite{hauptToeplitz,romberg2009compressive} using techniques from compressed sensing~\cite{Candes:2005cs,donoho2006compressed}. A series of recent papers analyze TV-norm minimization for recovering point sources from low-pass measurements~\cite{supportPursuit,superres,superres_new} and derive stability guarantees~\cite{superres_noisy, support_detection, azais2015spike, tang2014minimax,peyreduval,li2016approximate}. This framework has been extended to obtain both exact-recovery and robustness guarantees for deconvolution via convex programming for bump-like kernels including the Gaussian in~\cite{bendory64stable,bendory2016robust} and specifically Ricker kernels in~\cite{pereg2017seismic} under the assumption that the complete convolution between a train of spikes and the kernel of interest is known (i.e. without sampling). 

When the convolution kernel is a low-pass function, as is approximately the case for the Gaussian kernel, one can reformulate the spike deconvolution as a super-resolution problem. First, the spectrum of the convolved signal is estimated from the data. Then the low-frequency component of the spikes is approximated by dividing the result by the spectrum of the convolution kernel. Finally, the spikes are recovered from the estimate of their low-pass spectrum using spectral super-resolution techniques based on convex programming~\cite{superres} or on Prony's method~\cite{deProny:tg,Stoica:2005wf}, as in the finite-rate-of-innovation (FRI) framework~\cite{vetterli2002sampling,dragotti2007sampling}. This framework can also be applied to arbitrary non-bandlimited convolution kernels~\cite{uriguen2013fri} and nonuniform sampling patterns~\cite{pan2016towards}, but without exact-recovery guarantees.

\section{Proof of Exact Recovery (\Cref{thm:Exact})}
\label{sec:ExactRecoveryProof}
In the proof of \Cref{thm:Exact} we use standardized versions of our
kernels where $\sigma=1$:  
\begin{equation}
  \KG(t):= \exp\brac{-\frac{t^2}{2}}\quad\text{and}\quad
  \KR(t):= (1-t^2)\exp\brac{-\frac{t^2}{2}} = -(\KG)^{(2)}(t),
\end{equation}
without loss of generality. This is equivalent to expressing $t$ in units of $\sigma$. Auxiliary Mathematica code to perform the computations needed for the proof is available online\footnote{\url{http://www.cims.nyu.edu/~cfgranda/scripts/deconvolution_proof.zip}}.
\subsection{Dual Certificate}
We prove \Cref{thm:Exact} by establishing the
existence of a \emph{certificate} that guarantees exact recovery.

\begin{restatable}[Proof in \Cref{sec:DualCertProof}]{proposition}{DualCert}
\label{prop:DualCert}
Let $T \subseteq \RR$ be the nonzero support of a signal $\mu$ of the
form~\eqref{eq:signal}. If for any sign pattern $\rho\in\{-1,1\}^{|T|}$
there exists a function of the form
\begin{align}
Q(t) := \sum_{i=1}^n q_i K(s_i-t) \label{eq:dualcomb_K}
\end{align}
satisfying
\begin{alignat}{2}
 & Q(t_j) = \rho_j, \qquad && \forall t_j \in
  T, \label{eq:conditionQT}\\ 
  & \abs{Q(t)} < 1, && \forall t \in
  T^c, \label{eq:conditionQTc}
\end{alignat}
then the unique solution to Problem~\eqref{pr:TVnorm} is $\mu$.
\end{restatable}
In words, to prove exact recovery we need to show that is possible to
interpolate the sign of the amplitudes of any superposition of spikes, which we denote $\rho$, on the support of the spikes using
scaled copies of the convolution kernel centered at the location of
the samples.

\bdb{The vector $q$ is known as a \emph{dual certificate} in the
literature~\cite{Candes:2005cs} because it certifies recovery and is a solution
to the Lagrange dual of Problem~\eqref{pr:TVnorm}:
  \begin{equation}
    \label{pr:TVnormDual}
    \begin{aligned}
      \underset{q}{\op{maximize}} \quad& q^Ty\\ 
      \text{subject to} \quad&
      \sup_{t}\left|\sum_{i=1}^n q_iK(s_i-t)\right|\leq 1.
    \end{aligned}
  \end{equation}
}
Dual certificates have mostly been used to derive guarantees for inverse
problems involving random measurements, including compressed
sensing~\cite{Candes:2005cs,candes2011probabilistic}, matrix
completion~\cite{candes2012exact} and phase
retrieval~\cite{candes2015phase}. In such cases, the construction
relies on concentration bounds and other tools from probability
theory~\cite{vershynin2010introduction}. In contrast, our setting is
completely deterministic.

Our techniques are inspired by the deterministic certificate used to
establish exact recovery of an atomic measure from low-pass
measurements in~\cite{superres}. In that case, the certificate
corresponds to a low-pass trigonometric polynomial which also
interpolates an arbitrary sign pattern on the support of the atomic measure that
represents the signal. This polynomial is constructed via
interpolation using a low-pass interpolation kernel with fast
decay. The crucial difference with our setting is that
in~\cite{superres} the interpolating function just needs to be
low-pass, which makes it possible to center the shifted copies of the
interpolation kernel at the elements of the signal support ($T$ in our
notation). Similarly, if the whole convolution $K\ast \mu$ is assumed to be know, as in~\cite{bendory64stable,bendory2016robust,pereg2017seismic}, a certificate can be built by interpolating the sign pattern with the convolution kernel in the same way. In contrast, if we incorporate sampling into the measurement process, we are constrained to use shifted kernels \emph{centered at the sample locations} $s_1,\ldots,s_n$. This requires a new method for constructing the dual certificate, which is our main technical contribution.

By the sample proximity and separation conditions
(\Cref{def:ProxSep}), we can associate each spike location in $T$ with two nearby
samples that are not too close to each other. This allows us to form a subset $\wt{S}\subseteq S$ of size
$2|T|$ that contains pairs of samples separated by $\kappa(S,T)$ that are $\gamma(S,T)$-close to each spike location in the
support~$T$. We construct the \textit{dual combination} $Q$ using
kernels centered at the elements of $\wt{S}$,
\begin{equation}
\label{eq:DualCombination}
Q(t) := \sum_{\tilde{s}_i \in \wt{S}} q_i K( \tilde{s}_i - t).
\end{equation}
In order to satisfy condition~\eqref{eq:conditionQT}, $Q$ must
interpolate the sign pattern $\rho$ on $T$. To satisfy
condition~\eqref{eq:conditionQTc}, $Q$ must have a local
extremum at each element of $T$. Otherwise the construction will
violate condition~\eqref{eq:conditionQTc} close to the support of the
signal as shown in~\Cref{fig:onlybump}.  To favor local extrema on
$T$, we constrain the derivative of $Q$ to vanish at those points. The
sign-interpolation and zero-derivative constraints yield a system of
$2\abs{T}$ equations,
\begin{equation}
\label{eq:Interpolation}
\begin{aligned}
Q(t_i) & = \rho_i,\\ Q'(t_i) & = 0, \quad \text{for all $t_i\in T$}.
\end{aligned}
\end{equation}
The main insight underlying our proof is that the solution to this system is amenable
to analysis once $Q$ is reparametrized appropriately. The
reparametrization, described in the next section, allows us to prove
that the system is invertible and to control the amplitude of $Q$ on
$T^c$.

\subsection{Interpolation With Bumps and Waves}
\label{sec:bumps_waves}
We construct the dual combination $Q$ defined in \Cref{prop:DualCert}
by interpolating an arbitrary sign pattern $\rho$ using modified
interpolation kernels, which are linear combinations of shifted copies
of~$K$. The construction is a reparametrization of
\eqref{eq:DualCombination} of the form
\begin{align} \label{eq:reparametrization}
Q(t) = \sum_{t_i\in T}\alpha_i
B_{t_i}(t,\tilde{s}_{i,1},\tilde{s}_{i,2})+\beta_i
W_{t_i}(t,\tilde{s}_{i,1},\tilde{s}_{i,2}),
\end{align} 
where for each $t_i \in T$ we define a pair of modified kernels $B_{t_i}$ and $W_{t_i}$. The motivation is that controlling the
coefficients $\alpha_i$ and $\beta_i$ is considerably simpler than
controlling $q_1$, \ldots, $q_{\abs{\wt{S}}}$ directly.

We call the first interpolation kernel $B_{t_i}$ a \emph{bump},
defined as
\begin{align}
\label{eq:bump}
B_{t_i}(t,\tilde{s}_{i,1},\tilde{s}_{i,2}) & :=
b_{i,1}K(\tilde{s}_{i,1}-t)+ b_{i,2}K(\tilde{s}_{i,2}-t),
\end{align}
where $\tilde{s}_{i,1}$ and $\tilde{s}_{i,2}$ are the two elements of
$\wt{S}$ that are closest to $t_i$.  The coefficients $b_{i,1}$ and
$b_{i,2}$ are adjusted so that $B_{t_i}$ has a local extremum equal to
one at $t_i$:
\begin{align}
B_{t_i}(t_i,\tilde{s}_{i,1},\tilde{s}_{i,2}) & =1,
\\ \frac{\partial}{\partial
  t}B_{t_i}(t_i,\tilde{s}_{i,1},\tilde{s}_{i,2})& =0.
\end{align} 
If $\tilde{s}_{i,1}$ and $\tilde{s}_{i,2}$ are close enough to $t_i$,
the extremum is a maximum and $B_{t_i}$ is indeed a bump-like kernel
centered at $t_i$. \Cref{fig:bump} shows examples of bumps when $K$ is
the Gaussian kernel and the Ricker wavelet.

\begin{figure}[t]
\centering
\includegraphics{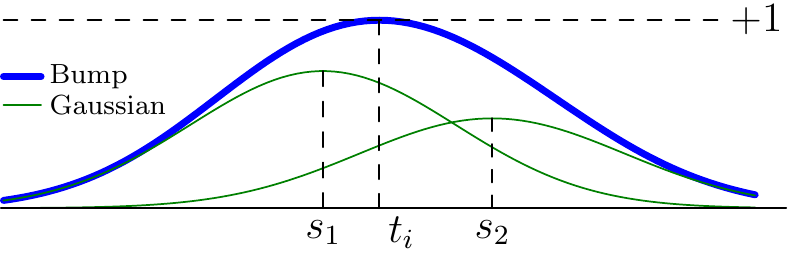}
\includegraphics{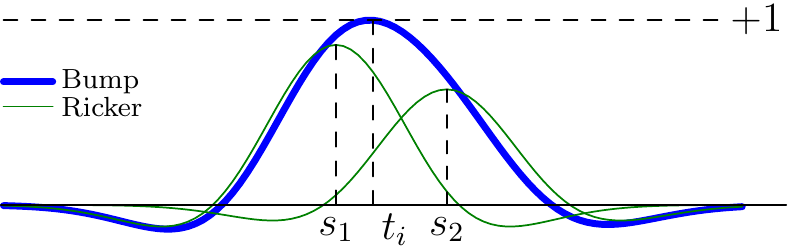}
  \caption{The bump interpolation function $B_{t_i}(t,s_1,s_2)$ is a
    linear combination of kernels centered at
    $s_1,s_2$. The figure shows examples for the Gaussian kernel (left) and the Ricker wavelet (right). }
  \label{fig:bump}
\end{figure}

The bumps are linear combinations of shifted copies of the
interpolation kernel $K$ centered at the samples, so we can construct
a dual combination $Q$ of the form~\eqref{eq:dualcomb_K} interpolating
an arbitrary sign pattern $\rho$ by weighting them appropriately. The
problem is that the resulting construction does not satisfy the
zero-derivative condition in~\eqref{eq:Interpolation} and therefore
tends to violate condition~\eqref{eq:conditionQT} close to the
elements of $T$, as shown in \Cref{fig:onlybump}. This is the reason
why our dual-combination candidate~\eqref{eq:reparametrization}
contains an additional interpolation kernel, which we call a
\emph{wave},
\begin{align}
W_{t_i}(t,\tilde{s}_{i,1},\tilde{s}_{i,2}) & = w_{i,1}
K(\tilde{s}_{i,1}-t)+ w_{i,2} K(\tilde{s}_{i,2}-t).
\end{align}

\begin{figure}[b]
\centering 
\includegraphics{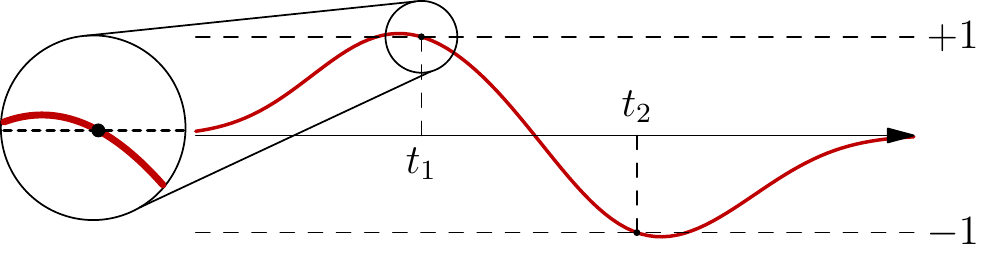}
\caption{A linear combination of two bumps interpolating the signs
  $+1$ at $t_1$ and $-1$ at $t_2$.  Since there are no constraints on its derivative, the red
  curve does not have local extrema at $t_1,t_2$.
  The construction is not a valid certificate because 
  its magnitude is not bounded by one.}
\label{fig:onlybump}
\end{figure}

In this case, the coefficients are adjusted so that
\begin{align}
W_{t_i}(t_i,\tilde{s}_{i,1},\tilde{s}_{i,2}) & =0,
\\ \frac{\partial}{\partial
  t}W_{t_i}(t_i,\tilde{s}_{i,1},\tilde{s}_{i,2}) & =1.
\end{align}  
Each $W_{t_i}$ is a wave-shaped function centered at
$t_i$. \Cref{fig:wave} shows examples of waves when $K$ is the
Gaussian kernel and the Ricker wavelet. The additional degrees of
freedom in~\eqref{eq:reparametrization} allow us to enforce the
zero-derivative condition and obtain a valid dual combination. The
role of the wave in our analysis is analogous to the role of the
derivative of the interpolation kernel in the dual-certificate
construction for super-resolution from low-pass
data~\cite{superres}. Note that here we cannot use the derivative of
the bumps or of $K$, as the resulting construction would not be a
linear combination of shifted copies of $K$.

\begin{figure}[t]
\includegraphics{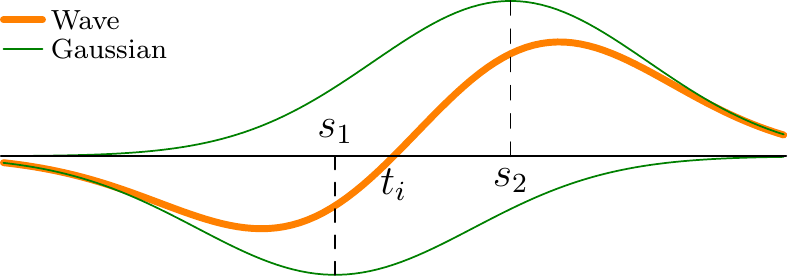}
\includegraphics{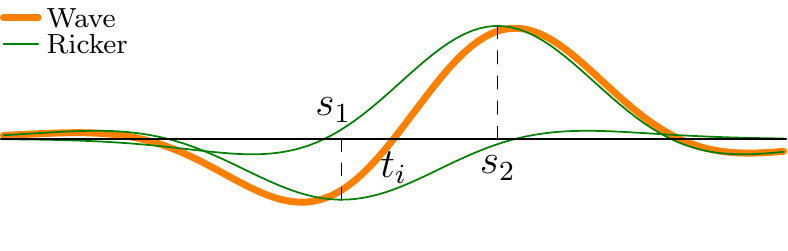}
  \caption{The wave interpolation function $W_{t_i}(t,s_1,s_2)$ is a
    linear combination of kernels centered at $s_1,s_2$. The figure shows examples for the Gaussian kernel (left) and the Ricker wavelet (right). }
  \label{fig:wave}
\end{figure}

The following lemma provides a closed-form expression for the
coefficients used to build the bump and the wave. In addition, it
shows that the bump and wave always exist for the Gaussian kernel and
the Ricker wavelet under mild assumptions on the sample proximity and
the sample separation.

\begin{restatable}[Proof in
    \Cref{sec:BumpExistsProof}]{lemma}{BumpExists}
\label{lem:BumpExists}
For a fixed kernel $K$, the bump $B_{t_i}$ and wave $W_{t_i}$ exist
with coefficients given by
\begin{align}
\MAT{b_{i,1} & w_{i,1} \\ b_{i,2} & w_{i,2}} & =
\frac{1}{K(\tilde{s}_{i,2}-t_i)K^{(1)}(\tilde{s}_{i,1}-t_i)-K^{(1)}(\tilde{s}_{i,2}-t_i)K(\tilde{s}_{i,1}-t_i)}\MAT{-K^{(1)}(\tilde{s}_{i,2}-t_i)
  & -K(\tilde{s}_{i,2}-t_i)\\K^{(1)}(\tilde{s}_{i,1}-t_i) &
  K(\tilde{s}_{i,1}-t_i)} \notag
\end{align}
when the expression in the denominator is nonzero. The Gaussian kernel
has nonzero denominator for $\tilde{s}_{i,1}\neq \tilde{s}_{i,2}$.
The Ricker wavelet has nonzero denominator when $\tilde{s}_{i,1}\neq
\tilde{s}_{i,2}$ and $|\tilde{s}_{i,1}|,|\tilde{s}_{i,2}|<1$.
\end{restatable}
The condition $|\tilde{s}_{i,1}|,|\tilde{s}_{i,2}|<1$ is natural for
the Ricker wavelet, since it has roots at $\pm1$.

We would like to emphasize that the bumps and waves are just a tool
for analyzing the dual certificate obtained by solving the system of
equations~\eqref{eq:Interpolation}. Under the conditions of
\Cref{thm:Exact}, which ensure that the support
of the signal is not too clustered and that there are two samples
close to each spike, we show that each coefficient $\alpha_i$ is close
to the sign of the spike and $\beta_i$ is small, so that
\begin{align}
  Q(t) & = \sum_{t_i\in T} \alpha_i
  B_{t_i}(t,\tilde{s}_{i,1},\tilde{s}_{i,2})+\beta_i
  W_{t_i}(t,\tilde{s}_{i,1},\tilde{s}_{i,2}) \\ & \approx \sum_{t_i\in
    T} \rho_i B_{t_i}(t,\tilde{s}_{i,1},\tilde{s}_{i,2}).
\end{align}
This allows us to control $Q$ and show that it is a valid dual
combination. In contrast, characterizing the coefficients $q_1$,
\ldots, $q_{\abs{\wt{S}}}$ directly is much more challenging, as
illustrated in \Cref{fig:bwdual}.

\begin{figure}[tp]
  \centering
  \includegraphics{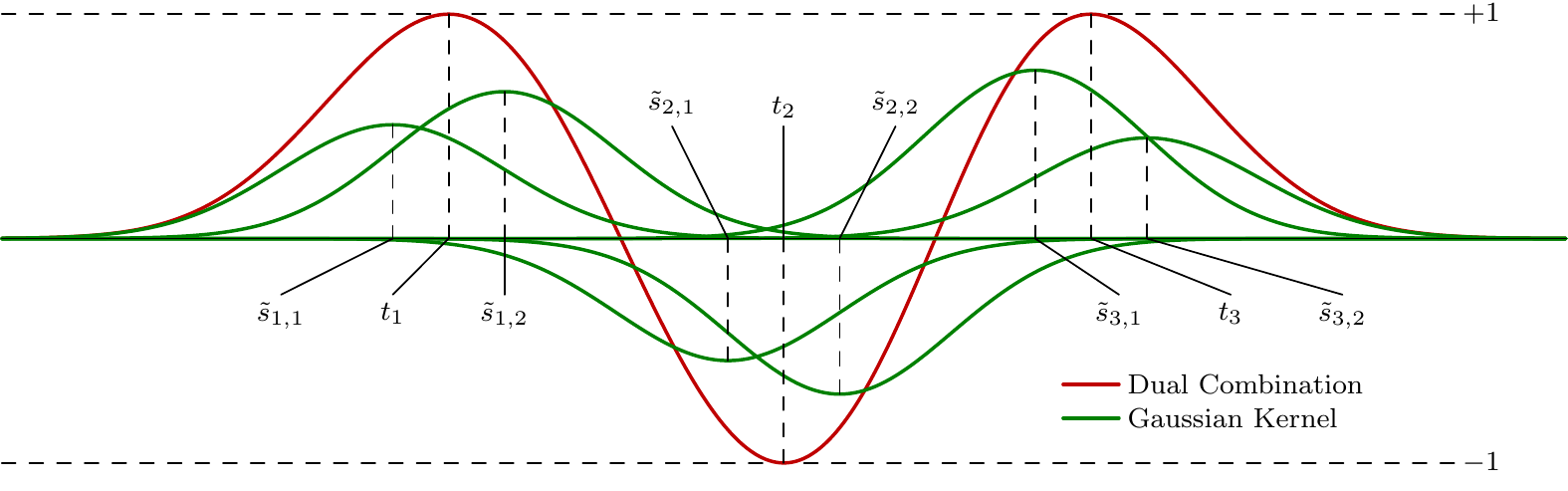}
  \includegraphics{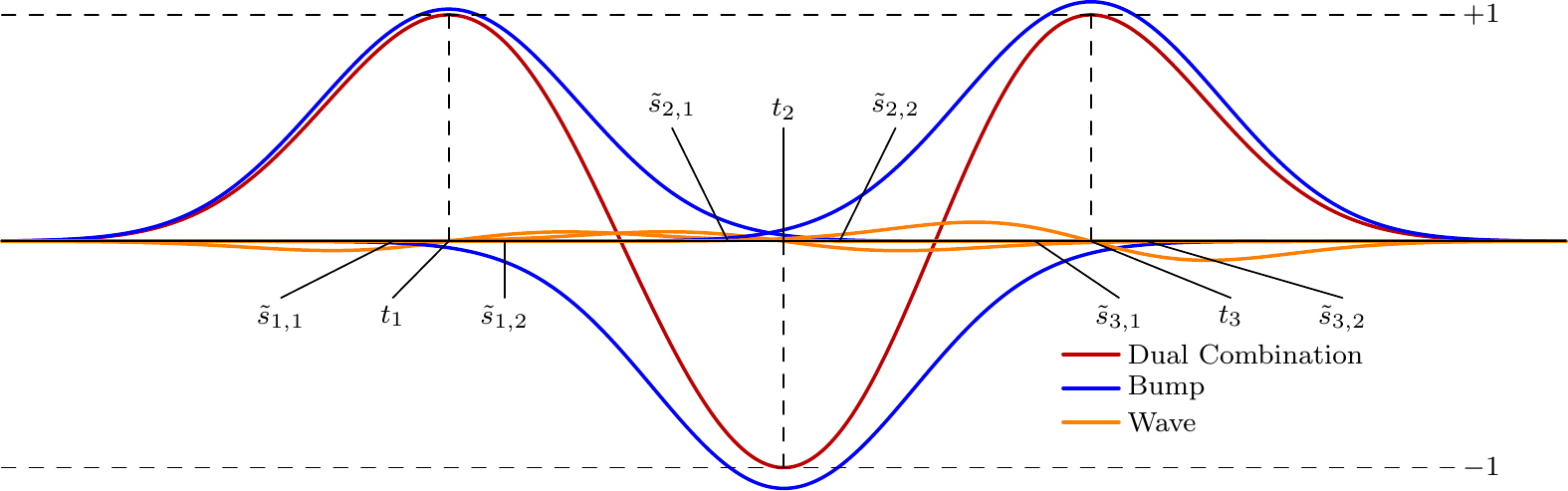}
  \includegraphics{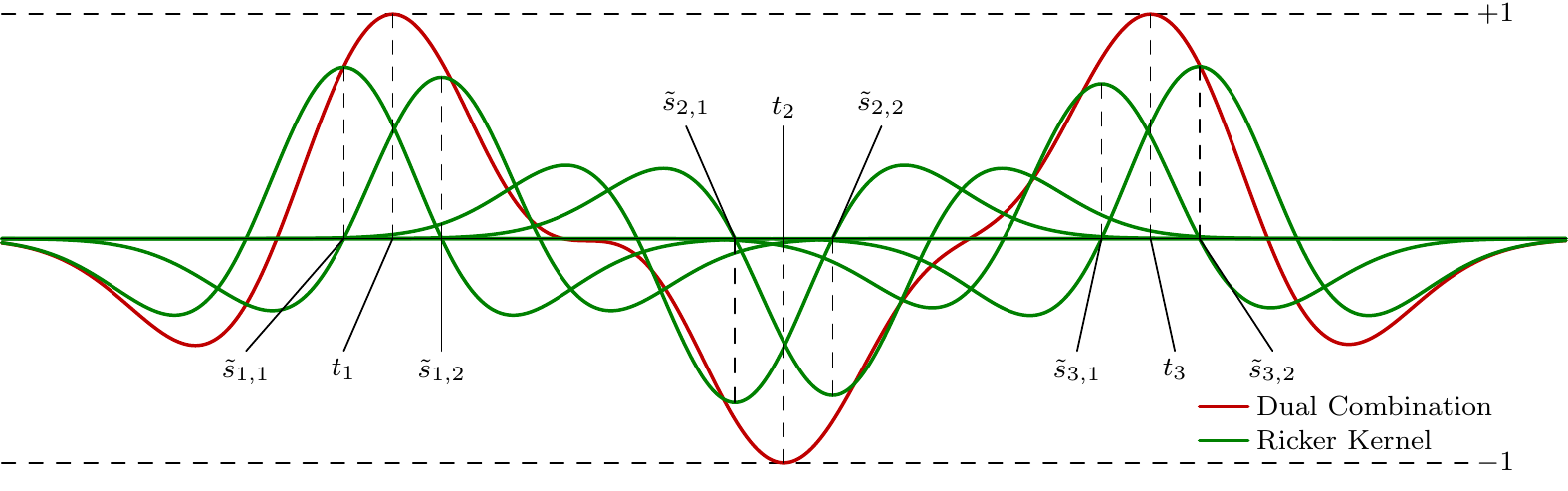}
  \includegraphics{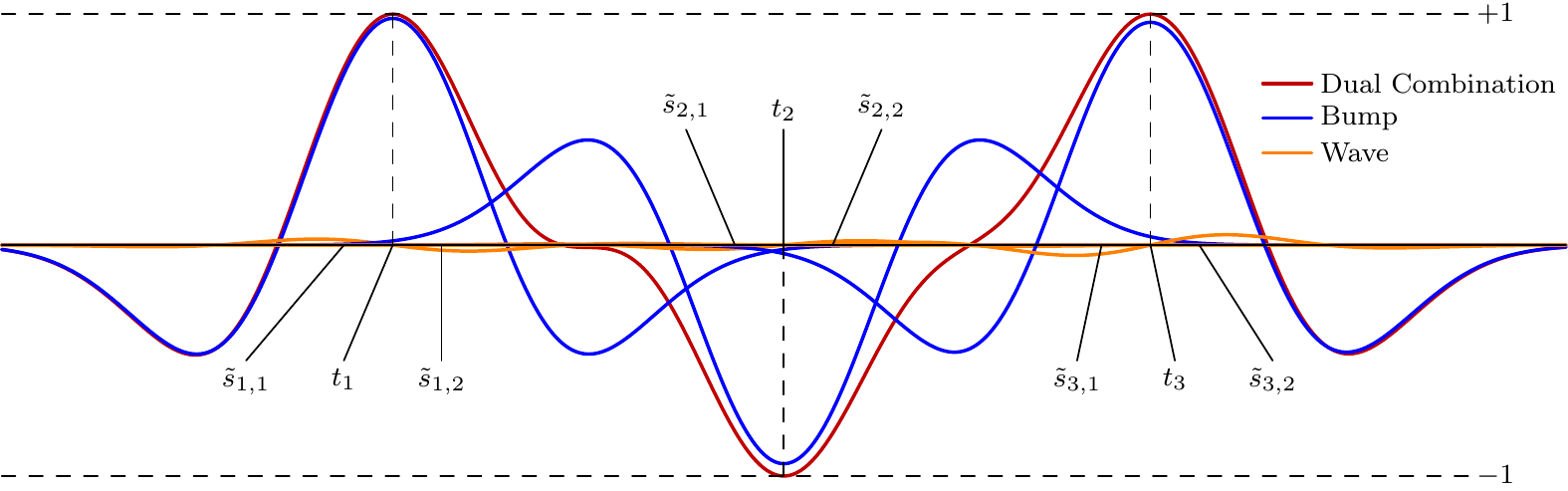}
  \caption{Effect of the proposed reparametrization for the Gaussian (top) and Ricker kernels (bottom). The dual combination $Q$ (red) is a sum of scaled shifted kernels with varying amplitudes (green). When we decompose the same function $Q$ into bumps (blue) and waves (orange), the amplitudes of the bumps all approximately equal the signs of the corresponding spikes, whereas the waves are very small.}
  \label{fig:bwdual}
\end{figure}

The following two lemmas formalize this intuition. The first
establishes that the system of equations~\eqref{eq:Interpolation} is
invertible. The second shows that the amplitude of $Q$ is bounded on
$T^c$. Together with \Cref{prop:DualCert} they yield a proof of
\Cref{thm:Exact}. Both of their proofs rely
heavily on bounds controlling the bumps and waves and their
derivatives, which are compiled in Section~\ref{sec:bwbounds}.
 
\begin{lemma}[Proof in \Cref{sec:CoeffBoundProof}]
\label{lem:CoeffBound}
 Under the assumptions of \Cref{thm:Exact}, the system
 of equations~\eqref{eq:Interpolation} has a unique solution.
\end{lemma}
\begin{lemma}[Proof in \Cref{sec:qboundproof}]
\label{lem:qbound}
  Under the assumptions of \Cref{thm:Exact}, the dual
  combination corresponding to the solution 
  system~\eqref{eq:Interpolation} satisfies $\abs{ Q(t) } <1$ for
  all $t\in T^c$.
\end{lemma}

\bdb{
We now illustrate the proof technique described in this section using
a simple example.  Consider an atomic measure $\mu$ consisting of 3
point masses:
\begin{equation}
  \mu := a_1\delta_{t_1} + a_2\delta_{t_2} + a_3\delta_{t_3}.
\end{equation}
Here $T=\{t_1,t_2,t_3\}\subset\RR$.  Let $S$ denote a set of sample
locations, and let $\tilde{S}\subseteq S$ be defined by
\begin{equation}
  \tilde{S} =
  \{\tilde{s}_{1,1},\tilde{s}_{1,2},\tilde{s}_{2,1},\tilde{s}_{2,2}
  \tilde{s}_{3,1},\tilde{s}_{3,2}\},
\end{equation}
where $\tilde{s}_{i,1}$ and $\tilde{s}_{i,2}$ are close
to $t_i$ for $i=1,2,3$.  We fix a sign
pattern $\rho=(1,-1,1)$.  Using
kernels centered at the elements of $\tilde{S}$ we construct a dual
combination $Q(t)$ of the form in \eqref{eq:DualCombination} that
satisfies the interpolation equations in \eqref{eq:Interpolation}.
The resulting dual combination $Q(t)$ can be seen in the first and third plots of
\Cref{fig:bwdual}.  The amplitudes of the kernels centered at the
samples are difficult to control. We then reparametrize $Q(t)$
into the form given in \eqref{eq:reparametrization} by writing it as
a linear combination of 3 bumps and 3 waves.  This reparametrization
is shown in the second and fourth plots of \Cref{fig:bwdual}.  As can be seen
in the figure, the bump coefficients roughly match the
values of $\rho$ and the wave coefficients are very small, making them
amenable to analysis.  
}
\subsection{Bounding Bumps and Waves}
\label{sec:bwbounds}

In this section we establish bounds on the bumps and waves and their
derivatives, which are used in Sections~\ref{sec:CoeffBoundProof}
and~\ref{sec:qboundproof} to establish exact recovery. To simplify
notation, without loss of generality we consider a bump and a wave
centered at the origin and set
\begin{align}
  B(t,\tilde{s}_{i,1},\tilde{s}_{i,2}) &:=
  B_{0}(t,\tilde{s}_{i,1},\tilde{s}_{i,2}),\\ 
  W(t,\tilde{s}_{i,1},\tilde{s}_{i,2}) &:=
  W_{0}(t,\tilde{s}_{i,1},\tilde{s}_{i,2}).
\end{align}

We begin with a simplifying observation. For even kernels like the
Gaussian and Ricker, the bumps and waves exhibit a form of symmetry.
\begin{restatable}[Proof in \Cref{sec:EvenKernelProof}]{lemma}{EvenKernel}
  \label{lem:EvenKernel}
  Let $K$ be a kernel with corresponding bump $B$ and wave $W$.  If
  the kernel satisfies $K(t)=K(-t)$ for all $t\in\RR$ then
  \begin{equation}
    B(t,s_1,s_2) = B(-t,-s_1,-s_2)\quad\text{and}\quad
    W(t,s_1,s_2)=-W(-t,-s_1,-s_2),
  \end{equation}
  for all $s_1,s_2,t\in\RR$.
\end{restatable}

The shape of each bump and wave depends on where its corresponding
samples are located.  In order to account for all possible sample
locations satisfying the sample-proximity and sample-separation
conditions in \Cref{def:ProxSep}, we define sample-independent upper
bounds of the form
\begin{align}
\sym{|B^{(i)}|}(t) & := \sup_{\substack{|s_1|,|s_2|\leq
    \gamma(S,T)\\|s_1-s_2|\geq\kappa(S)}} |B^{(i)}(t,s_1,s_2)|
,\\ \sym{|W^{(i)}|}(t) & := \sup_{\substack{|s_1|,|s_2|\leq
    \gamma(S,T)\\|s_1-s_2|\geq\kappa(S)}}
|W^{(i)}(t,s_1,s_2)|,\\ \sym{B^{(i)}}(t) & :=
\sup_{\substack{|s_1|,|s_2|\leq \gamma(S,T)\\|s_1-s_2|\geq\kappa(S)}}
B^{(i)}(t,s_1,s_2),
\end{align}
for $i=0,1,2$.  By \Cref{lem:EvenKernel}, these functions are all
even. The bounds are depicted in \Cref{fig:maximize}.

\begin{figure}[t]
  \includegraphics{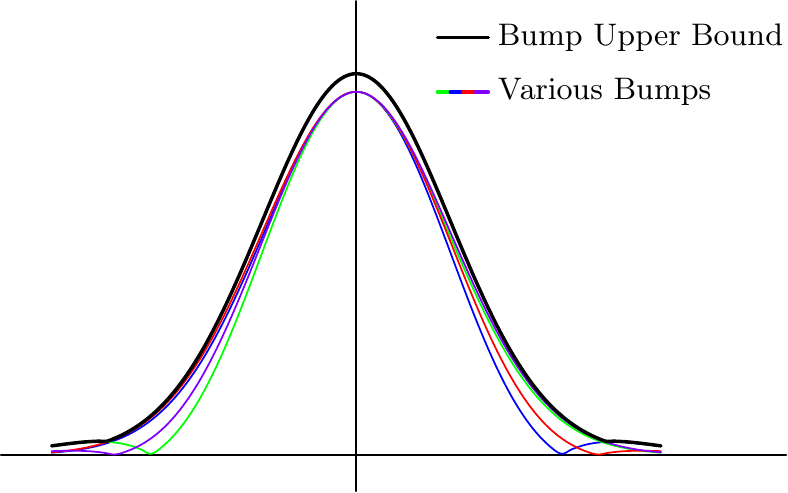}
  \includegraphics{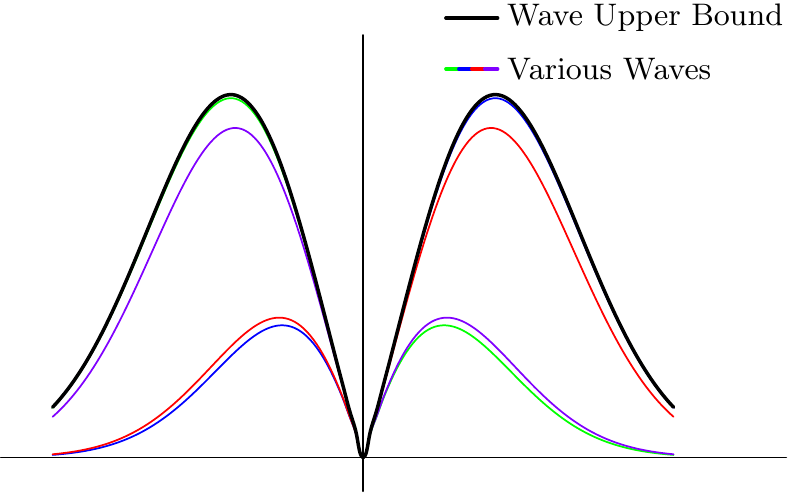}
  \caption{Sample-independent bounds on the absolute bump $\sym{|B|}(t)$ (left)
    and absolute wave $\sym{|W|}(t)$ (right) for the Gaussian kernel under a
    sample proximity $\gamma(S,T)=0.5$ and sample separation $\kappa(S)=0.05$}
  \label{fig:maximize}
\end{figure}

To simplify our analysis, we define monotonized versions of the
absolute bounds above for
$i \in \keys{0,1,2}$:
\begin{align}
\smo{|B^{(i)}|}(t) & := \sup\{|B^{(i)}|_{\infty}(u)\mid u\geq |t|\},
\\ \smo{|W^{(i)}|}(t) & := \sup\{|W^{(i)}|_{\infty}(u)\mid u\geq
|t|\}.
\end{align}
The decay properties of the Gaussian and Ricker kernels translate to
fast-decaying bumps and waves.  The following lemma formalizes this,
by controlling the tails of our monotonic bounds.
\begin{restatable}[Proof in
    \Cref{sec:TailDecayProof}]{lemma}{TailDecay}
  \label{lem:TailDecay}
  For the Gaussian and Ricker kernels we have
  \begin{align}
  \sum_{j=6}^\infty \smo{|B^{(i)}|}((j-1/2)\Delta) & \leq
  \frac{10^{-12}}{\kappa(S)}, & \sum_{j=6}^\infty
  \smo{|W^{(i)}|}((j-1/2)\Delta) & \leq
  \frac{10^{-12}}{\kappa(S)},\\ \smo{|B^{(i)}|}(t) & \leq
  \frac{10^{-12}}{\kappa(S)}, & \smo{|W^{(i)}|}(t) & \leq
  \frac{10^{-12}}{\kappa(S)},
  \end{align}
  for $i=0,1,2$, $t\geq 10$, and $\Delta\geq 2$.
\end{restatable}
When applying the lemma, $\Delta$ corresponds to the minimum
separation $\Delta(T)$ of the signal support.  There is nothing
special about the constants above (i.e.~$j=6$, $t\geq 10$, $\Delta
\geq 2$); they are simply chosen for definiteness and other values
would also work (see \Cref{sec:TailDecayProof} for more details).

\Cref{lem:TailDecay} shows that the bumps, waves, and their
derivatives are small for $t\geq 10$. The following lemma provides
piecewise constant bounds for $t\in[0,10)$.
\begin{restatable}[Proof in \Cref{sec:PiecewiseProof}]{lemma}{Piecewise}\label{lem:Piecewise}
  Fix $\gamma(S,T)$ and $\kappa(S)$ and let $N_1\in\ZZ_{>0}$. Partition the
  interval $[0,10)$ into $N_1$ intervals of the form
  \begin{align}
    \ml{U}_j & := \left[\frac{10(j-1)}{N_1},\frac{10j}{N_1}\right),
  \end{align}
 and $j=1,\ldots,N_1$.  For the Gaussian and Ricker kernels, there
 exist functions $\wt{|B^{(i)}|}$, $\wt{|W^{(i)}|}$, $\wt{B^{(i)}}$,
 such that for all $ t \in \ml{U}_j$ and $i=0,1,2$
  \begin{align}
    \sym{|B^{(i)}|}(t) & \leq \wt{|B^{(i)}|}(j) ,
    \\ \sym{|W^{(i)}|}(t) & \leq \wt{|W^{(i)}|}(j),\\ \sym{B^{(i)}}(t)
    & \leq \wt{B^{(i)}}(j).
  \end{align}
These bounds depend on $\gamma(S,T)$, $\kappa(S)$ and an additional
parameter $N_2$, satisfying $\kappa(S)>2\gamma(S,T)/N_2$ as explained in \Cref{sec:PiecewiseProof}.
\end{restatable}

\Cref{lem:TailDecay} can be used to extend the bounds of
\Cref{lem:Piecewise} to the monotonized bumps, waves, and their
derivatives. 
\begin{corollary}\label{cor:Monotonize}
  Assuming the conditions and definitions in
  \Cref{lem:TailDecay,lem:Piecewise} we have, for $i=0,1,2$ and $ t
  \in \ml{U}_j$
  \begin{align}
    \smo{|B^{(i)}|}(t) & \leq \max\left(\max_{k:j\leq k \leq
      N_1}\wt{|B^{(i)}|}(k),\epsilon\right),\\ \smo{|W^{(i)}|}(t) &
    \leq \max\left(\max_{k:j\leq k \leq
      N_1}\wt{|B^{(i)}|}(k),\epsilon\right),
  \end{align}
  where $\epsilon := 10^{-12}/\kappa(S)$.
\end{corollary}
\Cref{lem:Piecewise} and \Cref{cor:Monotonize} provide upper bounds on
the bumps, waves, and their derivatives that are symmetric,
sample-independent and, in the case of \Cref{cor:Monotonize},
monotonic.

\subsection{Proof of \Cref{lem:CoeffBound}: Invertibility of the Interpolation Equations}
\label{sec:CoeffBoundProof}

In this section we prove \Cref{lem:CoeffBound}, establishing the
invertibility of the interpolation equations~\eqref{eq:Interpolation}.
As a by-product, we also obtain bounds on the interpolation
coefficients $\alpha,\beta$ in \eqref{eq:reparametrization}.  For
simplicity we use the following abbreviated notation for the bumps and
waves:
\begin{align}
  B_i(t) &:=
  B_{t_i}(t,\tilde{s}_{i,1},\tilde{s}_{i,2}), \label{eq:shorthand_B}\\ W_i(t)
  &:=
  W_{t_i}(t,\tilde{s}_{i,1},\tilde{s}_{i,2}). \label{eq:shorthand_W}
\end{align}
To begin, we express equations~\eqref{eq:Interpolation} in terms of
bumps and waves:
\begin{align}
\sum_{j=1}^n \alpha_j B_j (t_i) + \beta_jW_j(t_i) & =
\rho_i,\\ \sum_{j=1}^n \alpha_j B_j^{(1)}(t_i) + \beta_jW_j^{(1)}(t_i)
& = 0, \quad \text{for all $t_i\in T$.}
\end{align}
Define the $n\times n$ matrices $\BB,\WW,\DBB,\DWW $ by
\begin{align}
\label{eq:block_matrices}
\begin{array}{rcl} (\BB)_{ij} & := & B_j(t_i),\\ (\WW)_{ij} & := &
  W_j(t_i),\\ (\DBB)_{ij} & := & B^{(1)}_j(t_i),\\ (\DWW )_{ij} & := &
  W^{(1)}_j(t_i).
\end{array}
\end{align}
This allows us to express the reparametrized interpolation equations in block-matrix
form:
\begin{equation}\label{eq:BlockEquation}
\MAT{\BB&\WW\\\DBB&\DWW }\MAT{\alpha\\\beta} = \MAT{\rho\\0}.
\end{equation}
If $\Delta(T)$ is sufficiently large we can exploit the decay of the
bumps and waves to prove that this matrix is close to the
identity.  This is formalized in the following linear-algebra result,
which shows how to convert norm bounds on these blocks into bounds on
$\alpha,\beta$, and gives conditions under which the system is
invertible.  Throughout, for an $n\times n$ matrix $A$, we write
$\|A\|_\infty$ to denote the matrix norm
\begin{equation}
  \|A\|_\infty = \sup_{\|x\|_\infty\leq 1 } \normInf{ Ax}.
\end{equation}
\begin{lemma}
  \label{lem:Schur} Suppose
  $\|\II-\DWW \|_\infty < 1$, and $\|\II-\CCC\|_\infty <1$ where
  \begin{equation*}
    \CCC = \BB-\WW(\DWW )^{-1}(\DBB)
  \end{equation*}
  is the Schur complement of $\DWW $, and $\II$ is the identity
  matrix.  Then
  \begin{equation}
    \MAT{\BB&\WW\\\DBB&\DWW }\MAT{\alpha\\\beta} =
    \MAT{\rho\\0}\label{eq:SchurEq}
  \end{equation}
  has a unique solution. Furthermore, we have
  \begin{align}
    \|\alpha\|_\infty & \leq
    \|\CCC^{-1}\|_\infty,\label{eq:SchurCoeff1}\\ \|\beta\|_\infty &
    \leq \|(\DWW
    )^{-1}\|_\infty\|\DBB\|_\infty\|\CCC^{-1}\|_\infty,\label{eq:SchurCoeff2}\\ \|\alpha-\rho\|_\infty
    & \leq \|\II-\CCC\|_\infty\|\CCC^{-1}\|_\infty,\label{eq:SchurLB}
  \end{align}
  where
  \begin{align}
    \|\II-\CCC\|_\infty & \leq \|\II-\BB\|_\infty +
    \|\WW\|_\infty\|(\DWW
    )^{-1}\|_\infty\|\DBB\|_\infty,\label{eq:SchurIC}\\ \|\CCC^{-1}\|_\infty
    & \leq
    \frac{1}{1-\|\II-\CCC\|_\infty},\label{eq:Neumann1}\\ \|(\DWW
    )^{-1}\|_\infty & \leq \frac{1}{1-\|\II-\DWW
      \|_\infty}.\label{eq:Neumann2}
  \end{align}
\end{lemma}
\begin{proof}
  Follows as a special case of \Cref{lem:SparseSchur}.
\end{proof}

To apply \Cref{lem:Schur} we require bounds on $\normInf{\II-\BB}$,
$\normInf{\WW}$, $\normInf{\DBB}$, and $\normInf{\II-\DWW }$.  We
compute these using the bounds in \Cref{sec:bwbounds}. We have
\begin{align}
\|\II-\BB\|_\infty & = \max_{t_i\in T} \sum_{\substack{t_j\in
    T\\t_j\neq t_i}} |B_j(t_i)| \leq 2\sum_{j=1}^\infty
\smo{|B|}(j\Delta(T)), \\ \|\WW\|_\infty &= \max_{t_i\in T} \sum_{t_j\in
  T} |W_j(t_i)| \leq 2\sum_{j=1}^\infty
\smo{|W|}(j\Delta(T)),\\ \|\DBB\|_\infty &= \max_{t_i\in T} \sum_{t_j\in
  T} |B^{(1)}_j(t_i)| \leq 2\sum_{j=1}^\infty
\smo{|B^{(1)}|}(j\Delta(T)),\\ \|\II-\DWW \|_\infty &= \max_{t_i\in T}
\sum_{\substack{t_j\in T\\t_j\neq t_i}} |W^{(1)}_j(t_i)| \leq
2\sum_{j=1}^\infty \smo{|W^{(1)}|}(j\Delta(T)),
\end{align}
where the monotonicity of the bounds allows us to assume that the
spacing between adjacent spikes equals the minimum separation. All
sums start at $j=1$ because $W_j(t_j)=B^{(1)}_j(t_j)=0$ and
$W^{(1)}_j(t_j)=B_j(t_j)=1$ by construction.

Applying \Cref{lem:TailDecay} we obtain, for $\Delta(T)\geq 2$ and
$\epsilon := 10^{-12}/\kappa(S)$,
\begin{align}
  \|\II-\BB\|_\infty & \leq 2\sum_{j=1}^5 \smo{|B|}(j\Delta(T)) +
  2\epsilon\label{eq:blocksumsB}\\ \|\WW\|_\infty & \leq 2\sum_{j=1}^5
  \smo{|W|}(j\Delta(T)) +
  2\epsilon\label{eq:blocksumsW}\\ \|\DBB\|_\infty & \leq
  2\sum_{j=1}^5 \smo{|B^{(1)}|}(j\Delta(T)) +
  2\epsilon\label{eq:blocksumsDB}\\ \|\II-\DWW \|_\infty & \leq
  2\sum_{j=1}^5 \smo{|W^{(1)}|}(j\Delta(T)) +
  2\epsilon\label{eq:blocksumsDW}.
\end{align}
If we fix values for $\Delta(T)$, $\gamma(S,T)$, $\kappa(S)$, and the
parameters $N_1$ and $N_2$ from \Cref{lem:Piecewise}, we can use
\Cref{cor:Monotonize} to compute the above bounds numerically. This
allows us to check whether the interpolation equations are invertible
for a specific triple $(\Delta(T),\gamma(S,T),\kappa(S))$ by
\Cref{lem:Schur}.  The upper bounds in
\cref{eq:blocksumsB,eq:blocksumsW,eq:blocksumsDB,eq:blocksumsDW} decrease
as $\Delta(T)$ increases due to monotonicity.  By combining this fact
with the definitions of $\gamma(S,T)$ and $\kappa(S)$, we see that
invertibility for a fixed triple $(a,b,c)$ implies invertibility for
all triples in the set
\begin{align}
\{(\Delta(T),\gamma(S,T),\kappa(S))\mid \Delta(T)\geq a,
\gamma(S,T)\leq b, \kappa(S)\geq c\}.
\end{align} 
This allows us to compute the above bounds for a finite set of triples and
obtain an infinite continuous region on which invertibility occurs.

In \Cref{fig:MatrixBounds} we compute bounds on $\normInf{\II-\DWW}$
and $\normInf{\II-\cC}$ for $\kappa(S)=0.05$.  The values of
$\gamma(S,T)$ are sampled at 200 equally-spaced points in the
interval $[0.05,1]$ for the Gaussian and $[0.05,0.7]$ for the Ricker.
The values of $\Delta(T)$ are sampled at 200 equally-spaced points in
the interval $[2,8]$ for both kernels.  For fixed
$\gamma(S,T)$ the value of $N_1$ is chosen to satisfy
$\gamma(S,T)/N_1 = 1/500$ for the Gaussian and $\gamma(S,T)/N_1 =
0.7/500$ for the Ricker.
For fixed $\Delta(T)$ the value of $N_2$ is chosen to satisfy
$\Delta(T)/N_2 = 8/700$ for both kernels.  In other words, we use a
fixed partition width in all upper bound computations (see
\Cref{lem:Piecewise} and \Cref{sec:PiecewiseProof} for the definitions
of $N_1$ and $N_2$).  
The plots show that invertibility is achieved over the required
region.

\begin{figure}[t]
\centering 
\includegraphics{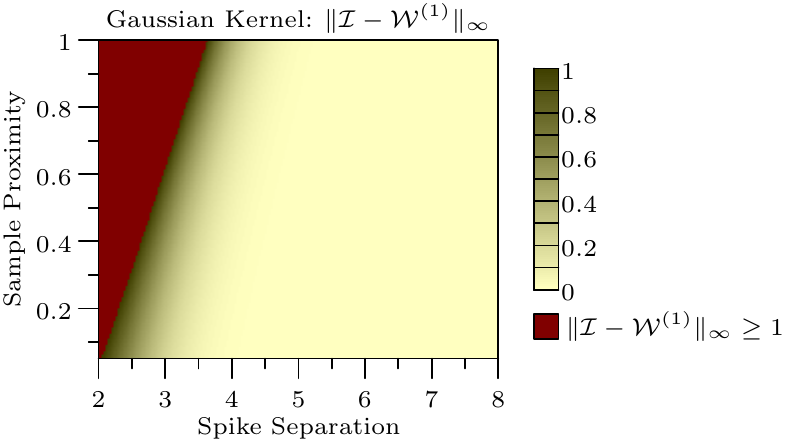}\quad
\includegraphics{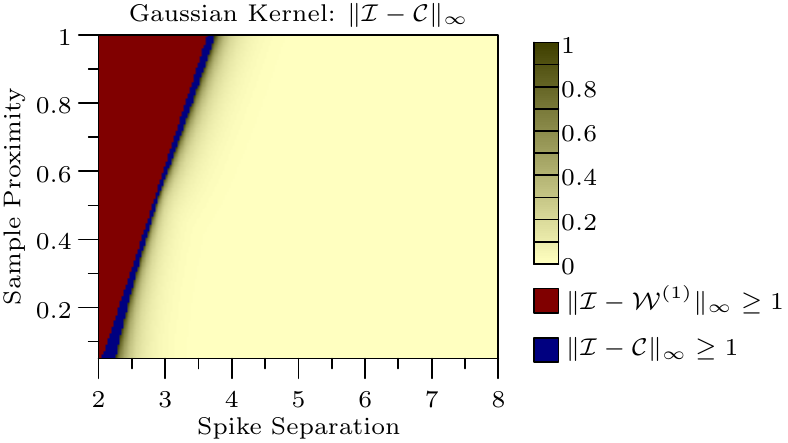}\\\vspace{.5cm}
\includegraphics{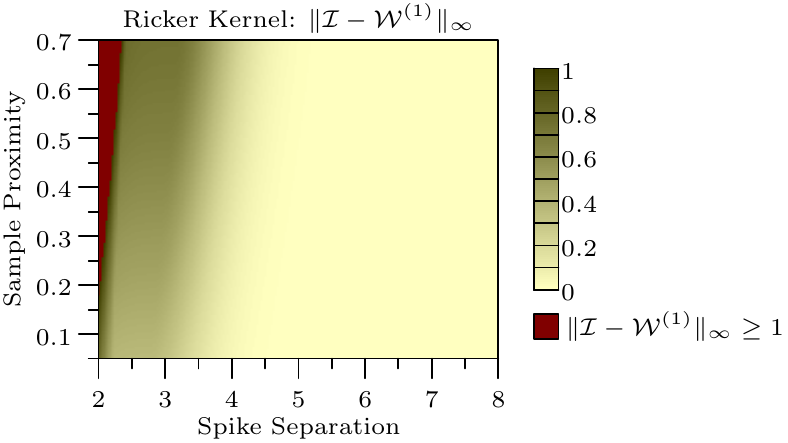}\quad
\includegraphics{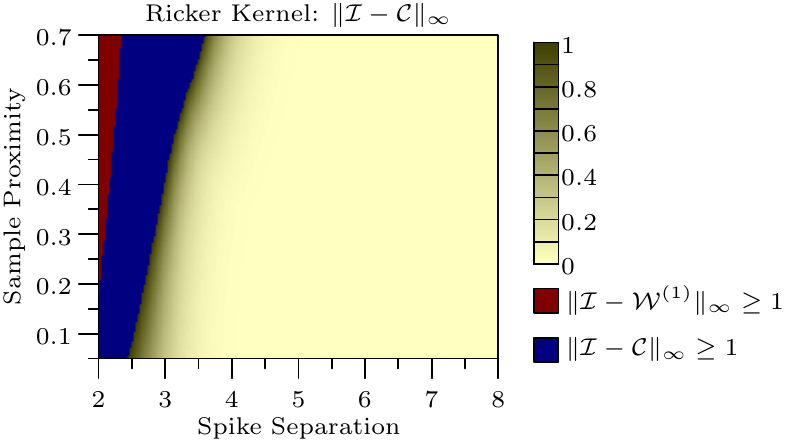}
  \caption{Upper bounds on $\normInf{\II-\DWW}$ and $\normInf{\II-\cC}$ for different values of the minimum separation and the sample proximity.}
  \label{fig:MatrixBounds}
\end{figure}

\subsection{Proof of \Cref{lem:qbound}: Bounding the Dual Combination}
\label{sec:qboundproof}
In this section, we prove that the dual combination $Q$ satisfies the
condition $|Q(t)|<1$ for $t\in T^c$. Without loss of generality, we
assume that there is a spike at the origin with sign $+1$ and we
restrict our analysis to the region between this spike and the next one. The same argument can be applied to the interval between any two other spikes. For clarity, we again use the shorthand
notation~\eqref{eq:shorthand_B} and~\eqref{eq:shorthand_W} for the
bumps and waves.

We begin by decomposing $Q$ as the sum of two terms that account for
the contribution of the bumps and waves separately:
$Q(t)=\sB(t)+\sW(t)$, where
\begin{align}
\sB(t):= \sum_{j=1}^n \alpha_j B_j(t) \quad\text{and}\quad \sW(t):=
\sum_{j=1}^n \beta_j W_j(t).
\end{align}
In the following lemma we bound these functions by applying the
triangle inequality, \Cref{lem:Schur}, and \Cref{lem:TailDecay}.
\begin{restatable}[Proof in
    \Cref{sec:GenericQBoundsProof}]{lemma}{GenericQBounds}\label{lem:GenericQBounds}
  Fix our kernel to be the Gaussian kernel or the Ricker wavelet.
  Assume the conditions of \Cref{lem:Schur,lem:TailDecay} hold, and
  that there is a spike at the origin with sign $+1$.  Let $h$ denote
  half the distance from the origin to the spike with smallest
  positive location, or $\infty$ if no such spike exists. Then, for
  $0<t\leq h$ (or $t>0$ if $h=\infty$) and
  $\epsilon:=10^{-12}/\kappa(S)$,
  \begin{align}
    |\sB^{(p)}(t)| & \leq
    \|\alpha\|_\infty\left(\smo{|B^{(p)}|}(v)+2\epsilon+ \sum_{j=1}^5
    \smo{|B^{(p)}|}(v+j\Delta(T))+\smo{|B^{(p)}|}(j\Delta(T)-v)\right)
    \label{eq:GQBoundAbsB}\\
    |\sW^{(p)}(t)| & \leq
    \|\beta\|_\infty\left(\smo{|W^{(p)}|}(v)+2\epsilon+ \sum_{j=1}^5
    \smo{|W^{(p)}|}(v+j\Delta(T))+\smo{|W^{(p)}|}(j\Delta(T)-v)\right),
    \label{eq:GQBoundAbsW}
  \end{align}
  for $p=0,1,2$ and $v=\min(t,\Delta(T)/2)$.  For $q=1,2$
  \begin{align}
    \sB^{(q)}(t) & \leq
    \alpha_{\LB}\sym{B^{(q)}}(t)+\|\alpha\|_\infty\left(
    2\epsilon+\sum_{j=1}^5 \smo{|B^{(q)}|}(v+j\Delta(T))
    +\smo{|B^{(q)}|}(j\Delta(T)-v) \right)\label{eq:GQBoundB},
  \end{align}
as long as $\sym{B^{(q)}}(t)\leq 0$, where
$\alpha_{\LB}:=1-\|\alpha-\rho\|_\infty\geq0$.
\end{restatable}
The following lemma provides three conditions under which $Q$ is
strictly bounded by one on the interval $(0,h]$. We omit the proof, which follows from basic calculus and the fact that $Q(0)=1$ and $Q^{(1)}(0)=0$ by \Cref{lem:CoeffBound}.

\begin{lemma}
\label{lem:Regions} 
Assume
  the conditions of \Cref{lem:GenericQBounds}.  Suppose there are
  $u_1,u_2\in(0,h]$ such that
  \begin{enumerate}
  \item (Neighboring) $Q^{(2)}(t)<0$ on $[0,u_1]$,
  \item (Near) $Q^{(1)}(t) < 0$ on $[u_1,u_2]$,
  \item (Far) $|Q(t)|<1$ on $[u_2,h]$,
  \end{enumerate}
  Then $|Q(t)|<1$ for $t\in(0,h]$.
\end{lemma} 
With \Cref{lem:Regions} in place, we outline a procedure for
establishing exact recovery for fixed values of $\Delta(T)$,
$\gamma(S,T)$, $\kappa(S)$ and the parameters $N_1$ and $N_2$ from
\Cref{lem:Piecewise}.
\begin{enumerate}
\item Apply \Cref{cor:Monotonize} and \Cref{lem:Schur} to check that
  the interpolation equations are invertible.
\item Using \Cref{lem:Piecewise} and \Cref{cor:Monotonize}, compute
  piecewise-constant upper bounds for $\smo{|B^{(q)}|}(t)$,
  $\smo{|W^{(q)}|}(t)$, and $\sym{B^{p}}(t)$ for $q=0,1,2$, $p=1,2$,
  and $t\in[0,\Delta/2]$.  These piecewise-constant functions are all
  defined on the same partition (see \Cref{lem:Piecewise}).  By
  \Cref{lem:GenericQBounds}, we obtain bounds on $|Q(t)|$ for
  $t\in(0,h]$, and on $Q^{(1)}(t),Q^{(2)}(t)$ for
    $t\in(0,\Delta(T)/2]$.
\item By iterating over the partition, test whether the conditions of
  \Cref{lem:Regions} are satisfied. If they are, $Q$ is a valid dual
  combination and exact recovery is proven.
\end{enumerate}
As all of the upper bounds in \Cref{lem:GenericQBounds} decrease as
$\Delta(T)$ increases, we again obtain that
recovery for a single triple $(a,b,c)$
implies recovery for all triples in
\begin{equation}
\{(\Delta(T),\gamma(S,T),\kappa(S))\mid \Delta(T)\geq a,
\gamma(S,T)\leq b, \kappa(S)\geq c\}.
\end{equation}
Using the same parameters as described in \Cref{sec:CoeffBoundProof}
to compute \Cref{fig:MatrixBounds} we verify that exact recovery holds
for the entire region specified in \Cref{thm:Exact}.
In \Cref{fig:ExactQPlots} we show an example for
the Gaussian and Ricker kernels where our bounds
on $|Q|$, $Q^{(1)}$ and $Q^{(2)}$ meet the criteria of \Cref{lem:Regions}.
In \Cref{fig:RegionPlots} we plot the region where exact recovery is
proven, and show upper bounds on the value of $u_1$.
\begin{figure}[t]
\centering 
\includegraphics{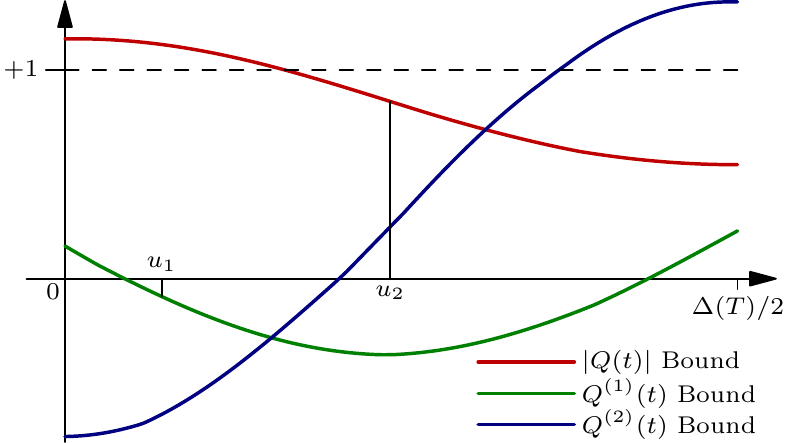}
\includegraphics{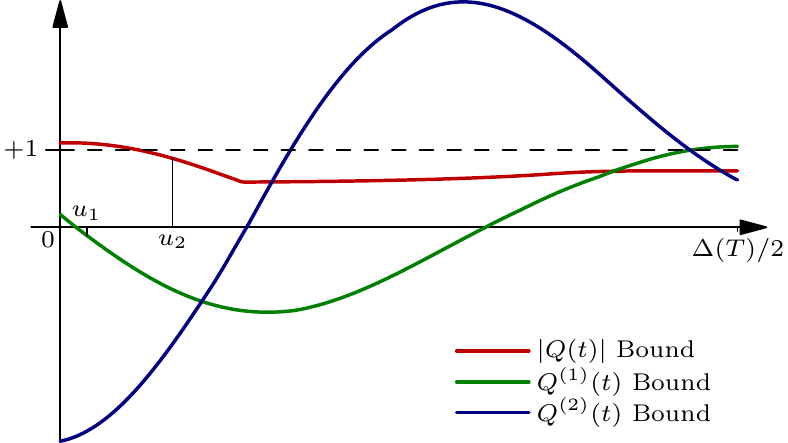}
  \caption{Bounds on the magnitude and the first and second
    derivatives of the dual combination $Q$ for a fixed value of the
    minimum separation and sample proximity. The left figure uses
    the Gaussian kernel with $\Delta(T)=3.5$.
    The right figure uses the Ricker kernel with $\Delta(T)=4.7$.
    Both figures have $\gamma(S,T)=0.3$ and $\kappa(S)=0.05$.}
  \label{fig:ExactQPlots}
\end{figure}
\begin{figure}[t]
\centering 
\includegraphics{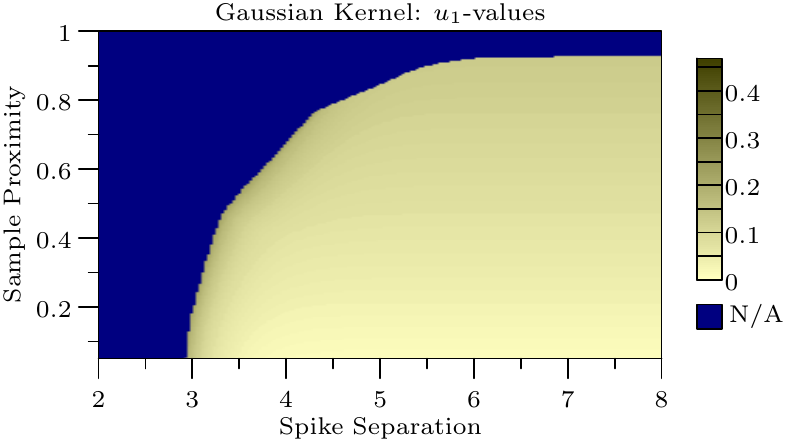}\quad
\includegraphics{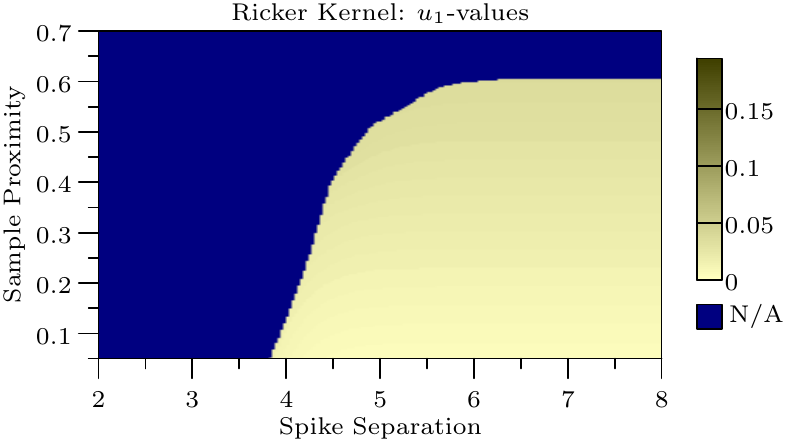}
  \caption{Upper bounds on the values of $u_1$ in the range of values of the minimum separation and the sample proximity for which we establish exact recovery.}
  \label{fig:RegionPlots}
\end{figure}

\section{Proof of Robustness to Dense Noise (\Cref{thm:RobustRecovery})}
\label{sec:RobustRecoveryProof}
The proof of \Cref{thm:RobustRecovery} is based on the proof of
support-detection guarantees for super-resolution from low-pass data
in~\cite{support_detection}. The proofs in~\cite{support_detection}
assume the existence of a dual certificate for the super-resolution problem, provided in~\cite{superres}. Here, we use similar arguments
to derive robustness guarantees for the deconvolution problem using
the dual-certificate construction presented in
\Cref{sec:ExactRecoveryProof}.  We again assume, without loss of
generality, that $\sigma=1$. Auxiliary Mathematica code to perform the computations needed for the proof is available online\footnote{\url{http://www.cims.nyu.edu/~cfgranda/scripts/deconvolution_proof.zip}}.
\subsection{Proof of \Cref{thm:RobustRecovery}, 
Inequalities \eqref{eq:RobustRecovery_2} and
\eqref{eq:RobustRecovery_3}}
\label{sec:RobustRecovery23Proof}
To prove \eqref{eq:RobustRecovery_2} and \eqref{eq:RobustRecovery_3},
we exploit some properties of the dual certificate from
\Cref{sec:ExactRecoveryProof}, which are established in the following
lemma.
\begin{restatable}[Proof in
    \Cref{sec:RobustCombinationOneProof}]{lemma}{RobustCombinationOne}
\label{lem:RobustCombinationOne}
  Under the assumptions of \Cref{thm:RobustRecovery}, there exists a
  function $Q$ of the form
  \begin{equation}
    Q(t) := \sum_{s_i\in S}q_i K(s_i-t), 
  \end{equation}
  for some $q \in \RR^n$ that satisfies:
  \begin{enumerate}
  \item $Q(t_j) = \sign(a_j)$ for all $t_j\in T$,
  \item $|Q(t)|<1$ for all $t\in T^c$,
  \item $|Q(t)| \leq 1-C_1(t-t_j)^2$ if $|t-t_j|\leq \eta$ for some
    $t_j\in T$,
  \item $|Q(t)|\leq 1- C_2$ if $|t-t_j|>\eta$ for all $t_j\in T$,
  \item $\normInf{q}\leq C_3$.
  \end{enumerate}
  The values $\eta,C_1,C_2,C_3>0$ depend only on $\Delta(T)$, $\gamma(S,T)$,
  and $\kappa(S)$. 
\end{restatable}
By combining \Cref{lem:RobustCombinationOne} with the fact that
$\normTV{\hat{\mu}}\leq\normTV{\mu}$ we obtain two useful bounds. 
Below we use $d(\hat{t}_k,T)$ to denote the distance from $\hat{t}_k$
to the set $T$:
\begin{equation}
  d(\hat{t}_k,T) := \min_{t\in T}|\hat{t}_k-t|.
\end{equation}
\begin{restatable}[Proof in
    \Cref{sec:GlobalNoiseBoundsProof}]{corollary}{GlobalNoiseBounds}
\label{lem:GlobalNoiseBounds}
  Under the assumptions of \Cref{lem:RobustCombinationOne} 
  \begin{align}
    \int Q(t)\,d\hat{\mu}(t) &
    \geq\normTV{\hat{\mu}}-2\normInf{q}\bxi\sqrt{|T|}
    ,\label{eq:NoiseQMuDiff}\\ \int Q(t)\,d\hat{\mu}(t) & \leq
    \sum_{\hat{t}_k\in\wh{T}:d(\hat{t}_k,T)\leq
      \eta}(1-C_1d(\hat{t}_k,T)^2) |\hat{a}_k| +
    \sum_{\hat{t}_k\in\wh{T}:d(\hat{t}_k,T)>
      \eta}(1-C_2)|\hat{a}_k|. \label{eq:NoiseQBound}
  \end{align}
\end{restatable}
The proof of inequality \eqref{eq:NoiseQMuDiff} differs from the analogous result
in~\cite{support_detection} because we cannot use Plancherel's
theorem. As a result, a factor of $\normInf{q}\sqrt{|T|}$ appears in
the bound (see \Cref{sec:GlobalNoiseBoundsProof} for more details).

Applying \eqref{eq:NoiseQMuDiff} and \eqref{eq:NoiseQBound} we obtain
\begin{align}
  \sum_{\hat{t}_k\in\wh{T}:d(\hat{t}_k,T)\leq
    \eta}(1-C_1d(\hat{t}_k,T)^2) |\hat{a}_k| +
  \sum_{\hat{t}_k\in\wh{T}:d(\hat{t}_k,T)> \eta}(1-C_2)|\hat{a}_k|
  &\geq \int Q(t)\,d\hat{\mu}(t) \\ &\geq
  \normTV{\hat{\mu}}-2\normInf{q}\bxi\sqrt{|T|}\\ &=
  \sum_{\hat{t}_k\in\wh{T}}
  |\hat{a}_k|-2\normInf{q}\bxi\sqrt{|T|}.
\end{align}
Rearranging terms gives
\begin{equation}
  \sum_{\hat{t}_k\in\wh{T}:d(\hat{t}_k,T)\leq \eta} C_1d(\hat{t}_k,T)^2
  |\hat{a}_k| + \sum_{\hat{t}_k\in\wh{T}:d(\hat{t}_k,T)>
    \eta}C_2|\hat{a}_k| \leq 2\normInf{q}\bxi\sqrt{|T|}.
\end{equation}
By Property 5 in~\Cref{lem:RobustCombinationOne} we obtain
\eqref{eq:RobustRecovery_2} and \eqref{eq:RobustRecovery_3}. 

\subsection{Proof of \Cref{thm:RobustRecovery}, Inequality \eqref{eq:RobustRecovery_1}}
\label{sec:RobustRecovery1Proof}
The main technical contribution of~\cite{support_detection} is a
method to isolate the error incurred by TV-norm minimization for each
particular spike by constructing a low-pass trigonometric polynomial
that is equal to one at the location of the spike and zero on the rest
of the signal support. The following lemma provides an analogous
object, which we construct using the techniques described in
\Cref{sec:ExactRecoveryProof}.
\begin{restatable}[Proof in
    \Cref{sec:RobustCombinationTwoProof}]{lemma}{RobustCombinationTwo}
\label{lem:RobustCombinationTwo}
  Assume the conditions of Lemma~\ref{lem:RobustCombinationOne}.  Then
  for any $t_j\in T$ there exists a function $Q_j$ of the form
  \begin{equation}
    Q_j(t)=\sum_{s_i\in S} \qj _{i} K(s_i-t), 
  \end{equation}
  where $\qj \in \RR^{n}$, such that:
  \begin{enumerate}
  \item $Q_j(t_j)=1$ and $Q_j(t_l)=0$ for $t_l\in T\setminus\{t_j\}$,
  \item $|Q_j(t)|<1$ for $t\in T^c$,
  \item $|1-Q_j(t)|\leq C_1'(t-t_j)^2$ for $|t-t_j|\leq \eta$,
  \item $|Q_j(t)|\leq C_1'(t-t_l)^2$ for $|t-t_l|\leq \eta$, $l\neq
    j$,
  \item $|Q_j(t)|\leq C_2'$ for $ \min_{t_l\in T} |t-t_l|>\eta$,
  \item $\normInf{\qj}\leq C_3'$.
  \end{enumerate}
  Here the values $C_1',C_2', C_3'>0$ only depend on $\Delta(T)$,
  $\gamma(S,T)$ and $\kappa(S)$.  The constant $\eta$ is the same as in \Cref{lem:RobustCombinationOne}.
\end{restatable}
\Cref{lem:RobustCombinationTwo} yields two inequalities that allow us to control the estimation error for the spike at $t_j$.
\begin{restatable}[Proof in
    \Cref{sec:GlobalNoiseBoundsTwoProof}]{corollary}{GlobalNoiseBoundsTwo}
\label{lem:GlobalNoiseBoundsTwo}
Under the conditions of \Cref{lem:RobustCombinationTwo} 
\begin{align}
  \left|\int Q_j(t)\,d(\mu-\hat{\mu})(t)\right| & \leq
  2\normInf{\qj}\bxi\sqrt{|T|},\label{eq:NoiseQjMuDiff}\\ \left|\int
  Q_j(t)\,d\hat{\mu}(t) - \sum_{\hat{t}_k\in\wh{T}:|\hat{t}_k-t_j|\leq
    \eta}\hat{a}_k\right| & \leq
  C'\bxi\sqrt{|T|},\label{eq:NoiseQjSum}
\end{align}
for some $C'>0$ that only depends on $\Delta(T)$, $\gamma(S,T)$,
and $\kappa(S)$.
\end{restatable}
Applying \Cref{lem:GlobalNoiseBoundsTwo} we have
\begin{align}
  \left|a_j - \sum_{\hat{t}_k\in\wh{T}:|\hat{t}_k-t_j|\leq
    \eta}\hat{a}_k\right| & = \left|\int Q_j(t)\,d(\mu-\hat{\mu})(t) + \int
  Q_j(t)\,d\hat{\mu}(t) - \sum_{\hat{t}_k\in\wh{T}:|\hat{t}_k-t_j|\leq
    \eta}\hat{a}_k\right|\\ & \leq 2\normInf{\qj}\bxi\sqrt{|T|} +
  C'\bxi\sqrt{|T|}\\ & \leq C\bxi\sqrt{|T|},
\end{align}
for some $C>0$ by Property 6 in \Cref{lem:RobustCombinationTwo}.  This concludes the proof.

\section{Proof of Exact Recovery with Sparse Noise (\Cref{thm:SparseNoiseRecovery,thm:SparseNoiseRecoveryRicker})}
\label{sec:SparseNoise}
We again assume, without loss of
generality, that $\sigma=1$. Auxiliary Mathematica code to perform the computations needed for the proof is available online\footnote{\url{http://www.cims.nyu.edu/~cfgranda/scripts/deconvolution_proof.zip}}.
\subsection{Dual Certificate}
The proof is based on the construction of a dual
certificate that guarantees exact recovery of the signal $\mu$ and of
the sparse noise component $z$.
\begin{restatable}[Proof in \Cref{sec:SparseDualCertProof}]{proposition}{SparseDualCert}
\label{prop:SparseDualCert}
Let $T \subseteq \RR$ be the nonzero support of $\mu$ and
$\Noi\subseteq S$ be the nonzero support of $w$.  If for any sign
patterns $\rho\in\{-1,1\}^{|T|}$ and $\rho'\in\{-1,1\}^n$ there exists a
\textit{dual combination} $Q$ of the form
\begin{align}
Q(t) := \sum_{i=1}^n q_i K(s_i-t)
\end{align}
satisfying
\begin{alignat}{2}
  Q(t_j) &= \rho_j,\qquad &&\forall t_j\in
  T \label{eq:SparseDualRho}\\ 
  |Q(t)| &<1, \qquad &&\forall t\in
  T^c,\label{eq:SparseDualQBound}\\ q_l &= \lambda\rho'_l,\qquad
  &&\forall s_l \in\Noi, \label{eq:SparseDualRhoP}\\ |q_l|&<
  \lambda,\qquad &&\forall s_l \in\Noi^c \label{eq:SparseqBound},
\end{alignat}
then $(\mu,w)$ is the unique solution to
problem~\eqref{pr:SparseNoiseRecovery}.  Here $\Noi^c:=S\setminus\Noi$.
\end{restatable}
\bdb{The form of this dual certificate comes from the Lagrange dual to
  problem \eqref{pr:SparseNoiseRecovery} given below:
  \begin{equation}
    \label{pr:SparseNoiseRecoveryDual}
    \begin{aligned}
      \underset{q}{\op{maximize}} \quad& q^Ty\\
      \text{subject to} \quad&
      \sup_{t}\left|\sum_{i=1}^n q_iK(s_i-t)\right|\leq 1,\\
      &\normInf{q}\leq \lambda.
    \end{aligned}
  \end{equation}
}

Condition~\eqref{eq:SparseDualRhoP} implies that any valid dual
combination is of the form $Q(t) = \Qaux(t) + 
R(t)$ where
\begin{align}
\Qaux(t) & := \sum_{s_i \in \Noi^c} q_i K(s_i-t), \\ 
\RQ(t) & := \lambda \sum_{s_i \in \Noi} \rho'_i K(s_i-t),
\end{align}
To obtain a valid certificate, we must construct $\Qaux$ so that it
interpolates $\rho_i - \RQ(t_i)$ for every $t_i \in T$, while
ensuring that \eqref{eq:SparseDualQBound} and
\eqref{eq:SparseqBound} hold. This is reminiscent of the certificate
constructed in~\cite{fernandez2016demixing} to provide exact recovery
guarantees for spectral super-resolution in the presence of sparse
corruptions. The crucial difference is that
in~\cite{fernandez2016demixing} $\RQ$ is a combination of sinusoids
with random signs. As a result, its magnitude is quite small and the
certificate can be constructed by choosing the coefficients of $\Qaux$
so that $\Qaux(t_i)= \rho_i - R(t_i)$ and $\Qaux(t_i)= - R^{(1)}(t_i)$
for all $t_i \in T$. In our case, such a construction satisfies
condition~\eqref{eq:SparseDualRho}, but 
violates either~\eqref{eq:SparseDualQBound}
or~\eqref{eq:SparseqBound}. 
More concretely, suppose we let $\lambda\geq
1$.  Then \eqref{eq:SparseDualQBound} does not hold at samples that are distant from any spike. If we instead choose
$\lambda<1$, it is usually not possible to satisfy \eqref{eq:SparseDualRho} and
\eqref{eq:SparseqBound} simultaneously at samples that are close to a spike.
In order to construct a valid certificate
it is necessary to cancel the effect of $\RQ$ on
$T^c$, as well as on $T$. This requires a new proof technique
described in the following section.
\subsection{Interpolation with Cancellations}
Under the assumptions of \Cref{thm:SparseNoiseRecovery}, the grid of
samples contains two disjoint subsets $\IS$ and $\CS$ which do not
overlap with $\Noi$, such that $\IS$ contains the two samples that are closest to each element
  in the support $T$ and $\CS$ contains the two samples that are immediately adjacent to
  each corrupted sample. We decompose $\Qaux$ into two
terms $\Qaux := \QI + \QC$, where
\begin{align}
\QI(t) & := \sum_{s_i \in \IS} q_i K(s_i-t), \\ \QC(t) & :=
\sum_{s_i \in \CS} q_i K(s_i-t).
\end{align}
The role of $\QC$ is to cancel out $\RQ$ so that
condition~\eqref{eq:SparseDualQBound} is not violated, whereas the
role of $\QI$ is to ensure that $\eqref{eq:SparseDualRho}$ holds after
the cancellation. We build $\QC$ as a superposition of shifted copies
of the bump defined by \cref{eq:bump}, each of which neutralize one
term of $\RQ$:
\begin{align}
\QC(t) := - \lambda \sum_{s_i\in \Noi} \rho'_i
B_{s_i}(t,s_{i-1},s_{i+1}).
\end{align}

\begin{figure}[t]
\centering \includegraphics{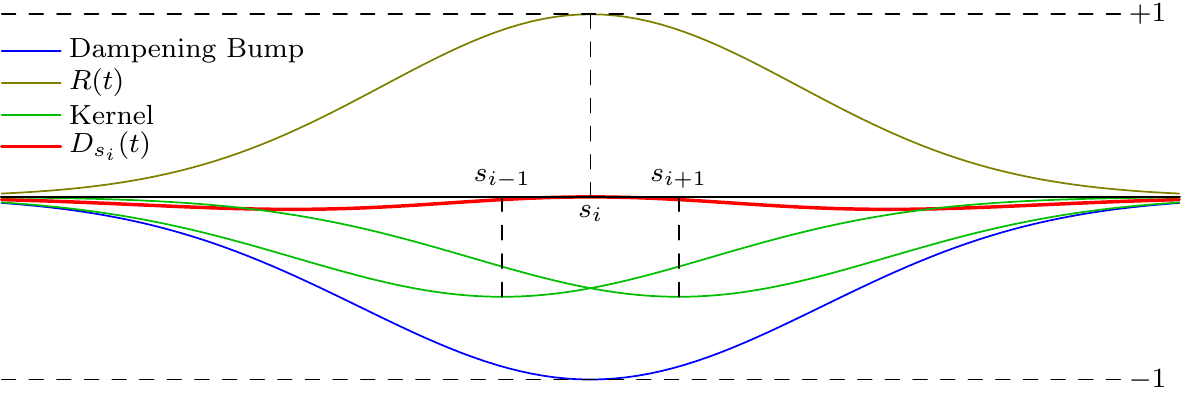}
\caption{Example of dampened kernel $D_{s_i}$ when the convolution kernel $K$ is Gaussian. A shifted copy of $K$ at $s_i$ is dampened by subtracting a bump formed using shifted copies of $K$ centered at two adjacent clean sample locations $s_{i-1}$ and $s_{i+1}$.}
\label{fig:dampen}
\end{figure}

As a result, $\QC + \RQ$ can be decomposed into a sum of
\emph{dampened kernels} defined by
\begin{equation}
  D_{s_i}(t,\lambda,\rho'_i) :=
  \lambda\rho'_i(K(s_i-t)-B_{s_i}(t,s_{i-1},s_{i+1})).
\end{equation}
$D_{s_i}$ is depicted in \Cref{fig:dampen}.
Let us define $\RC:= \QC + \RQ$, the sum of all dampened kernels,
\begin{equation}
  \RC (t) = \sum_{s_i \in \Noi} D_{s_i}(t,\lambda,\rho'_i).
\end{equation}
We choose the coefficients of $\QI$ so that $Q$ interpolates the sign
pattern $\rho$ and has zero derivative on $T$, analogously to the
system~\eqref{eq:Interpolation}:
\begin{equation}
\label{eq:Interpolation_sparse}
\begin{aligned}
\QI(t_i) & = \rho_i - \RC(t_i),\\ \QI^{(1)}(t_i) & = -\RC^{(1)}(t_i),
\quad \text{for all $t_i\in T$}.
\end{aligned}
\end{equation}

\begin{figure}[tp]
\centering
\begin{tabular}{FG}
   & Coefficients\\ \\
  \includegraphics{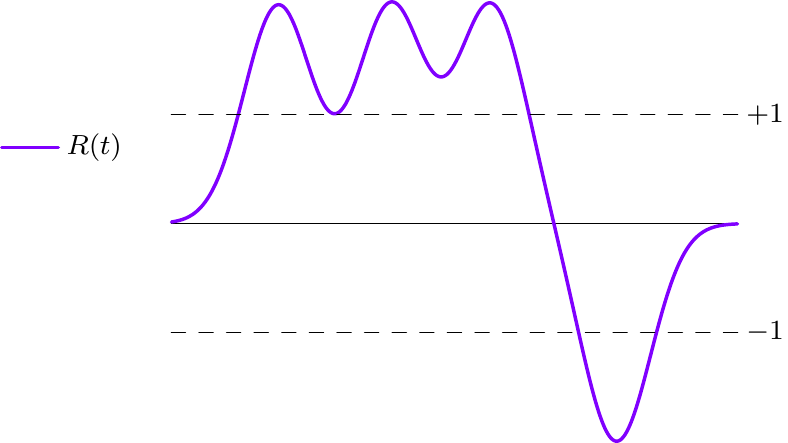} & \includegraphics{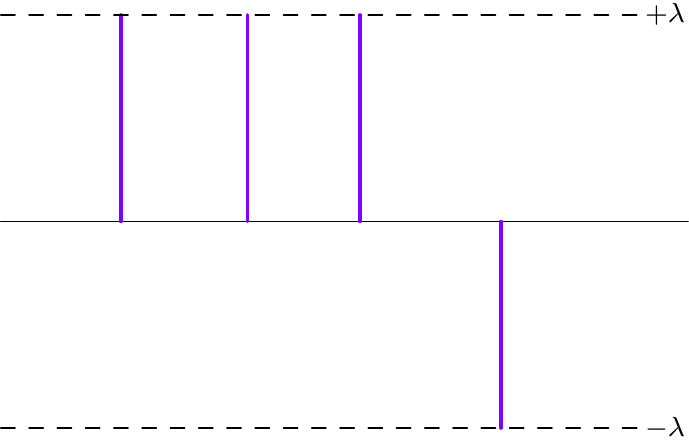}\\
  \includegraphics{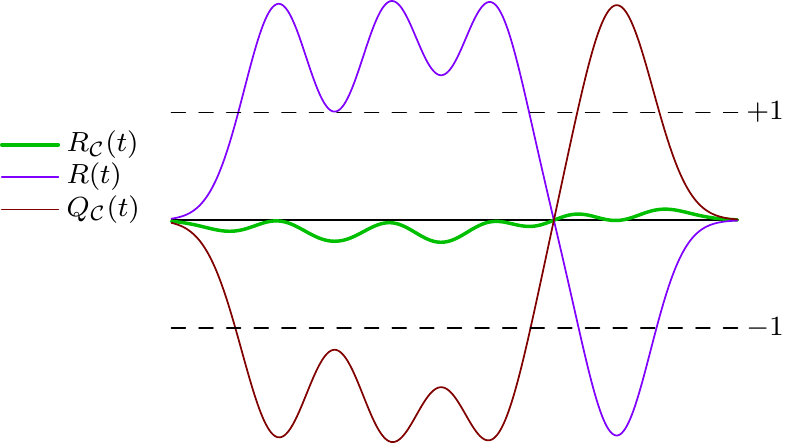}& \includegraphics{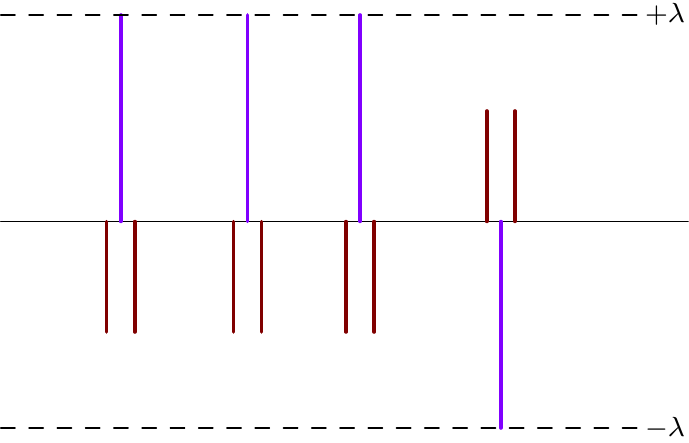}\\
  \includegraphics{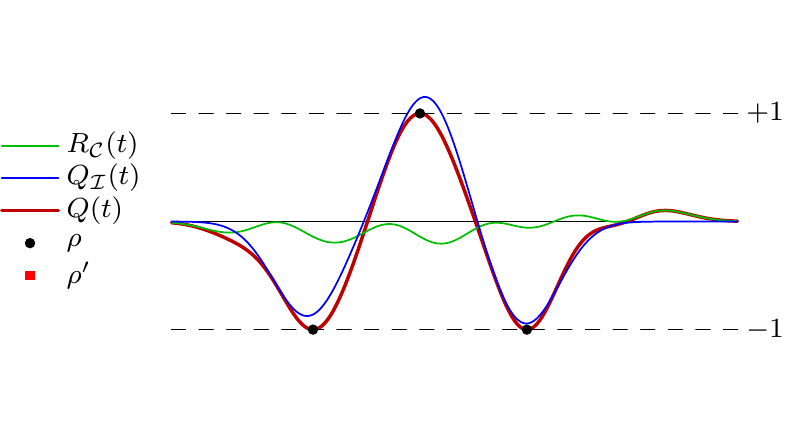} & \includegraphics{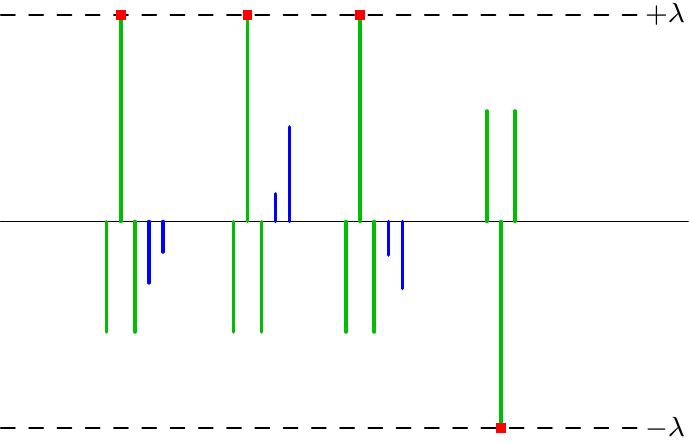}
\end{tabular}
\caption{Construction of the certificate to prove exact recovery in the presence of sparse noise. The top row illustrates the term in the dual combination that is fixed by the sparse noise. The second row shows how this term can be cancelled by forming dampened kernels using the adjacent clean samples. Finally on the third row two additional samples surrounding the spike are used to ensure that the construction interpolates the sign pattern of the spikes and has zero derivative on their support. This produces a valid certificate.}
\label{fig:sparsegaussian}
\end{figure}

This construction is depicted in \Cref{fig:sparsegaussian}. The same techniques used in \Cref{sec:CoeffBoundProof} to prove invertibility apply here, giving the following lemma.
\begin{lemma}[Proof in
  \Cref{sec:SparseCoeffBoundProof}] \label{lem:SparseCoeffBound} Under
the assumptions of \Cref{thm:SparseNoiseRecovery} for the Gaussian
kernel, and \Cref{thm:SparseNoiseRecoveryRicker} for the Ricker wavelet, the system of
equations defined in \eqref{eq:Interpolation_sparse} has a unique
solution.
\end{lemma}
Once $\QI$ is uniquely determined, the two lemmas below can be used to
show conditions \eqref{eq:SparseDualQBound} and
\eqref{eq:SparseqBound} hold.  To show $\abs{Q}$ is bounded, we write
$Q(t)=\QI(t)+\RC(t)$ and combine bounds on $\RC$ with the techniques
used in the proof of noiseless exact recovery to obtain
\Cref{lem:SparseQBound}.  Having proven
$|Q(t)|<1$ for $t\in T^c$, we then use this fact to prove a bound on $q_i$ for
$s_i\in\II$.  We defer the details to the appendix.
\begin{lemma}[Proof in \Cref{sec:SparseQBoundProof}]
\label{lem:SparseQBound}
Under the assumptions of \Cref{thm:SparseNoiseRecovery} for the Gaussian
kernel, and \Cref{thm:SparseNoiseRecoveryRicker} for the Ricker
wavelet, the dual combination corresponding to the solution
of \eqref{eq:Interpolation_sparse} satisfies $\abs{ Q(t) } <1$ for
all $t\in T^c$.
\end{lemma}
\begin{lemma}[Proof in \Cref{sec:SparseqBoundProof}]
\label{lem:SparseqBound}
Under the assumptions of \Cref{thm:SparseNoiseRecovery} for the Gaussian
kernel, and \Cref{thm:SparseNoiseRecoveryRicker} for the Ricker
wavelet, the dual combination corresponding to the solution
of \eqref{eq:Interpolation_sparse} satisfies $\abs{ q_l } <\lambda$
for all $l\in \Noi^c$.
\end{lemma}

\section{Numerical experiments}
\label{sec:numerical}
\subsection{Conditioning of Convolution Measurements with a Fixed Support}
\label{sec:conditioning}

\Cref{fig:minsep_illposed} shows an example of two signals with
substantially different supports that produce almost the same
measurements when convolved with a Gaussian or Ricker
kernel. In this section we investigate this phenomenon. We fix a set of points $T_D :=
\keys{t_1,t_2, \ldots,t_m} \subseteq \R$ with minimum separation
$\Delta_0/2$. This represents the support of the difference between
two signals with minimum separation $\Delta_0$, the first
corresponding to the odd locations $t_1, t_3, t_5, \ldots$ and the
second to the even locations $t_2,t_4,t_6,\ldots$ as in the example in \Cref{fig:minsep_illposed}. We denote the amplitudes of the difference
by $\tilde{a}\in\RR^m$, which is normalized so that $\normTwo{\tilde{a}}=1$.
The difference between the
convolutions of the two signals at a fixed sample $s_i$ is given by $ \sum_{t_j\in T_D}\wt{a}_j K(s_i-t_j)$.

\begin{figure}[tp]
\centering
\begin{tabular}{ECC}
  \vspace{.25cm}& Gaussian & Ricker\\ Smallest &
  \includegraphics{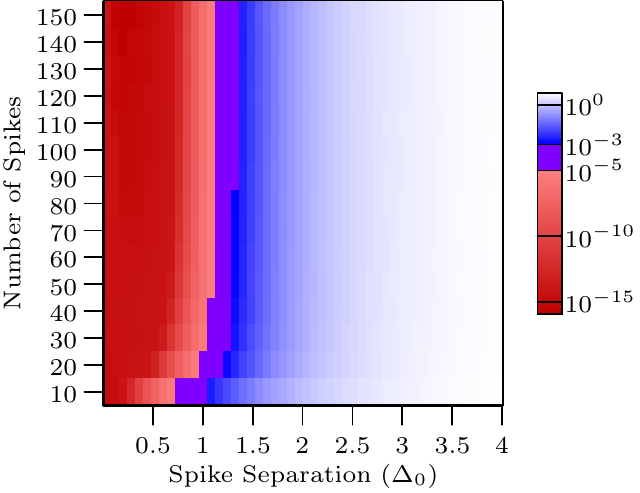}&
  \includegraphics{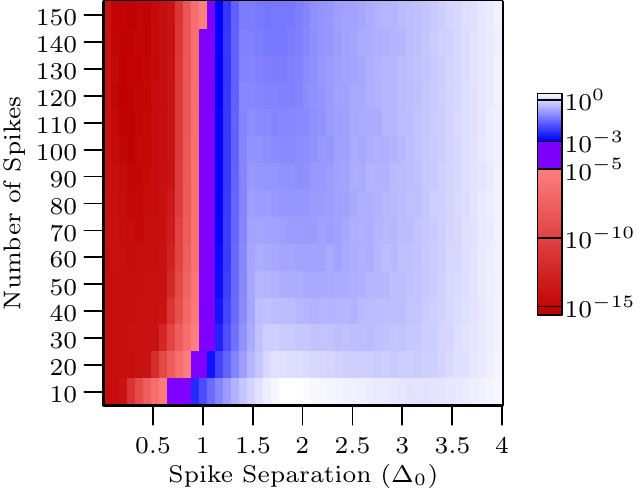}\\ Middle &
  \includegraphics{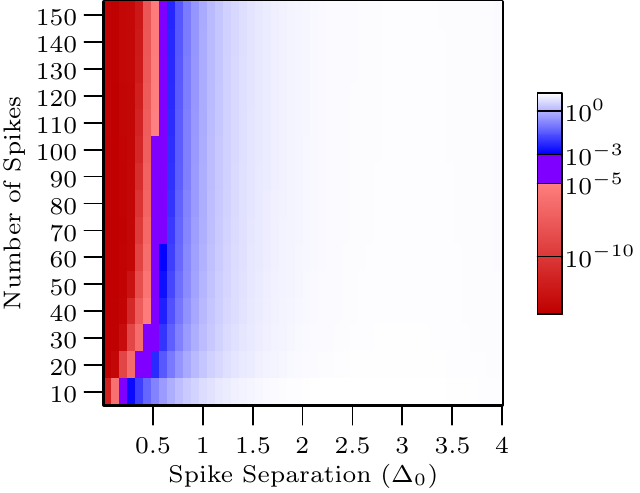}&
  \includegraphics{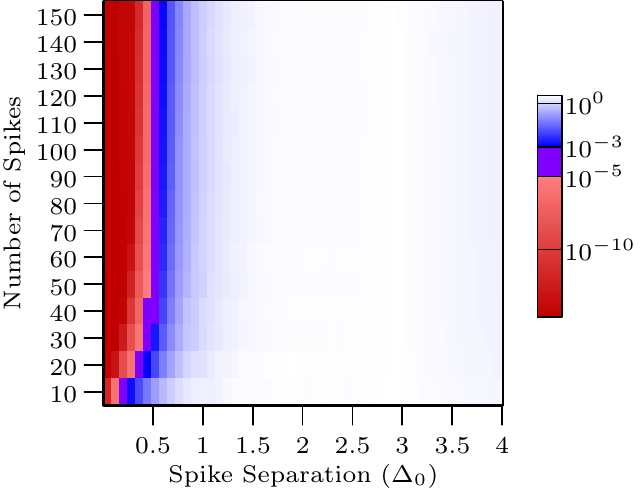}
\end{tabular}
\caption{The smallest and middle singular values of the measurement
  operator as a function of spike separation and number of
  spikes. Each plotted value is an average over 5 runs. The spike
  locations are perturbed randomly so that their separation varies up
  to 2.5\% of the given $\Delta_0$ value. The spike separation is
  expressed in units of~$\sigma$. }
\label{fig:SingNSpikeSep}
\end{figure}

\begin{figure}[tp]
\centering
\begin{tabular}{C@{\qquad\qquad}C}
  Gaussian & Ricker\\ \includegraphics{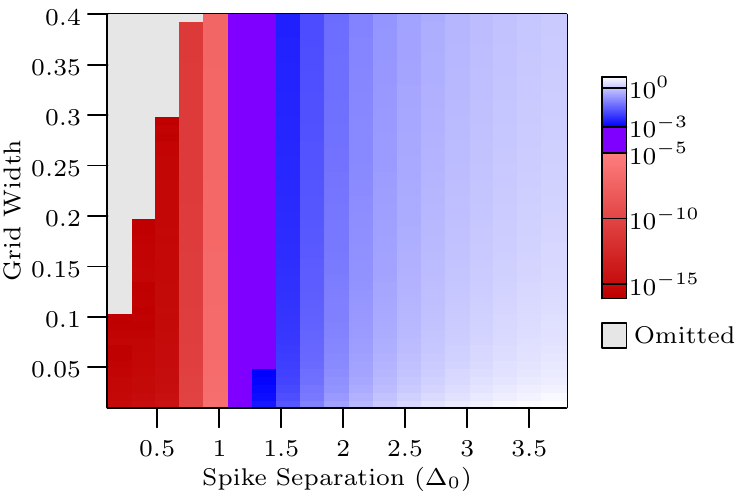} &
  \includegraphics{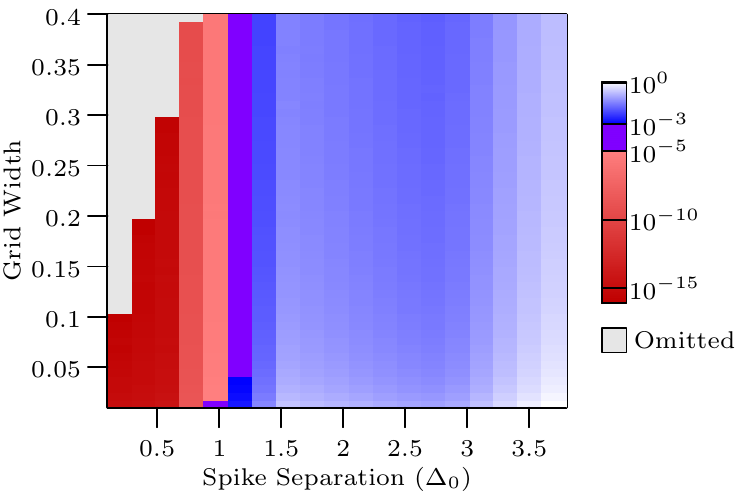}
\end{tabular}
\caption{The smallest singular value of the measurement operator as a
  function of spike separation and sampling grid width. The number of
  spikes is fixed at 50. Grid widths and spike separations are in
  units of~$\sigma$. Points with grid width larger than $\Delta_0/2$
  are omitted.}
  \label{fig:SingSepSep}
\end{figure}

Our aim is to determine for what values of $\Delta_0$ these
samples have negligible energy in the worst case and on average for a fixed set of sample locations
$s_1, s_2, \ldots, s_n$. This
requires choosing a worst-case and average-case value for the
amplitudes, which can be achieved by computing the singular-value decomposition
of a matrix with entries given by $K(s_i-t_j)$ for $t_j\in T_D$ and
$1\leq i \leq n$. The smallest singular value corresponds to the
smallest possible $\ell_2$ norm of the samples, whereas the middle
singular value can be interpreted as an average value of their $\ell_2$ norm. In
\Cref{fig:SingNSpikeSep} we plot the smallest and the middle singular
value for a range of values of the number of spikes $m$ and the
minimum separation $\Delta_0$ using a uniform grid of $n:=10m$
samples. For both convolution kernels of interest there is a phase
transition at a value of the minimum separation that does not seem to
depend on the number of spikes, as long as they are not very few. This
phase transition occurs around $\Delta_0 = \sigma$ (smallest singular
value) and $\Delta_0 = 0.5 \sigma$ (middle singular value) for both
kernels\footnote{The experiment is carried out for $\sigma=1$, but the
  results apply to any $\sigma$ by a simple change of
  variable.}. These results are not affected by the choice of step
size for the sampling grid, as shown in \Cref{fig:SingSepSep}, where
we fix the number of spikes to 50 and repeat the experiment for
different step sizes.

\subsection{Exact Recovery}
\label{sec:numexactrecovery}

In this section we evaluate our optimization-based deconvolution
method for different values of the minimum separation and the sample
separation. In our experiments, we restrict the signal to lie on a
uniform grid with $5\cdot 10^4$ points and solve
problem~\eqref{pr:TVnormDisc} using CVX~\cite{cvx}. We choose the
sample locations so that there are just two samples per spike which
are as far as possible from the spike location, with some random
jitter ($\pm 1\%$ of the sample proximity). In our experience, this
yields a \emph{worst-case} nonuniform sampling pattern for a fixed
sample proximity. The spike locations are also worst case in the sense
that the spikes are separated as much as possible, with some random
jitter ($\pm1\%$ of the minimum separation). We consider that the
method achieves exact recovery if the relative error of the recovered
amplitudes is below $10^{-4}$. We observe that when the minimum
separation is small, the chances of achieving exact recovery may be
higher for larger sample proximities. In those cases, there are four
samples close to any given spike (its own two and two more from the
neighboring spikes), which may be better than having two samples that
are closer. In order to correct for this effect in
\Cref{fig:ProofVsCVX} we \emph{monotonize} the results to show when we
observe exact recovery for all larger minimum separations and smaller
sample proximities.

\begin{figure}[tp]
\centering
\begin{tabular}{ECC}
  \vspace{.25cm}& Gaussian&Ricker\\ 
  10 Spikes&
  \includegraphics{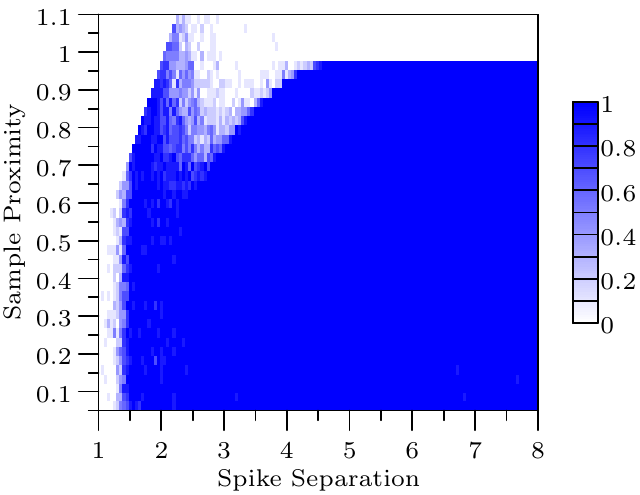}&
  \includegraphics{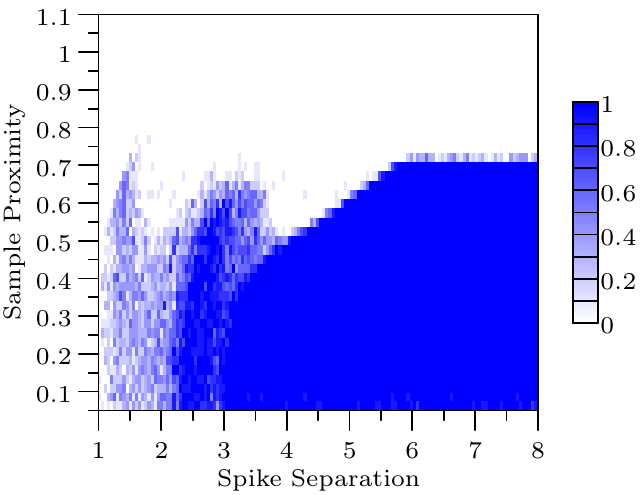}\\ 
  30 Spikes&
  \includegraphics{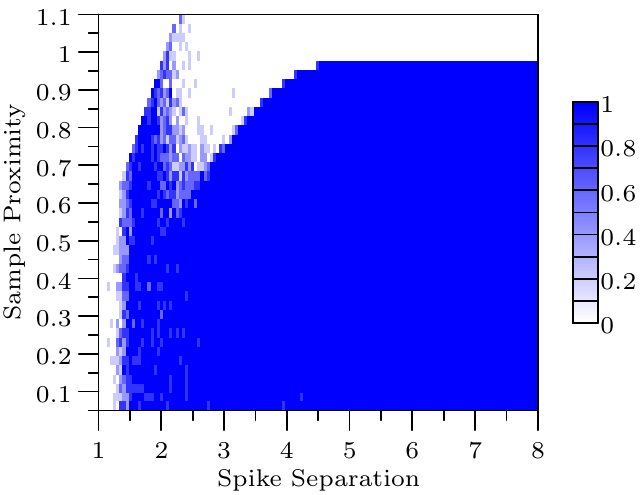}&
  \includegraphics{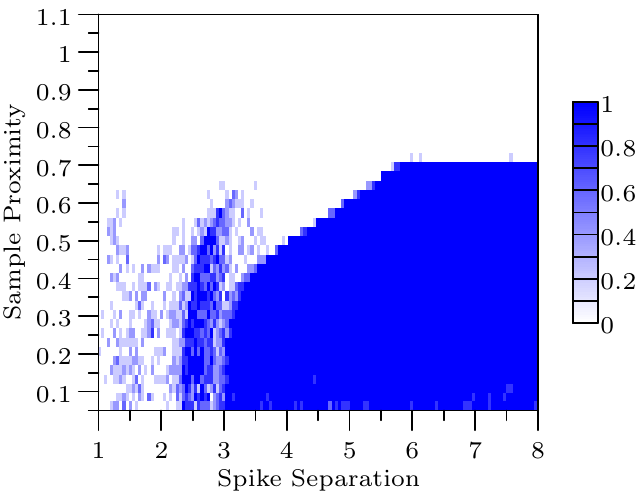}\\
  60 Spikes&
  \includegraphics{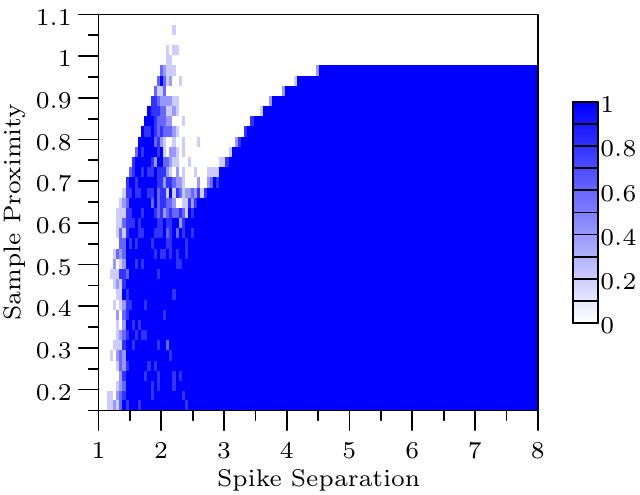}&
  \includegraphics{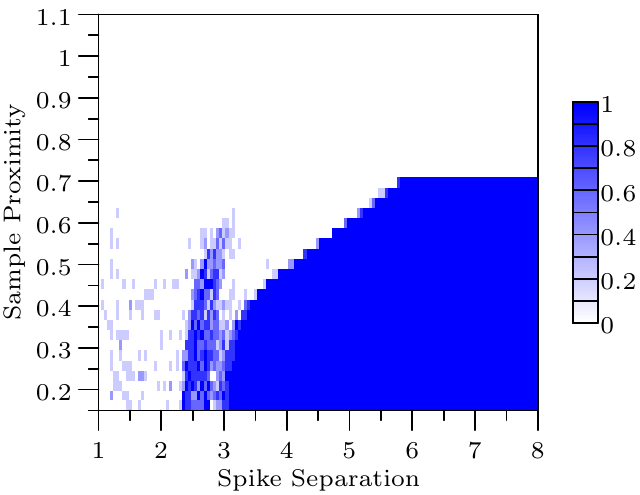}
\end{tabular}
\caption{Fraction of 5 runs yielding exact recovery for varying values
  of the spike minimum separation, sample proximity, and numbers of
  spikes in the signal. Spike separations and sample proximities are
  given in units of~$\sigma$ (the experiments are carried out for
  $\sigma:=0.003$). }
\label{fig:CVXExactRecovery}
\end{figure}

\begin{figure}[tp]
\centering \includegraphics{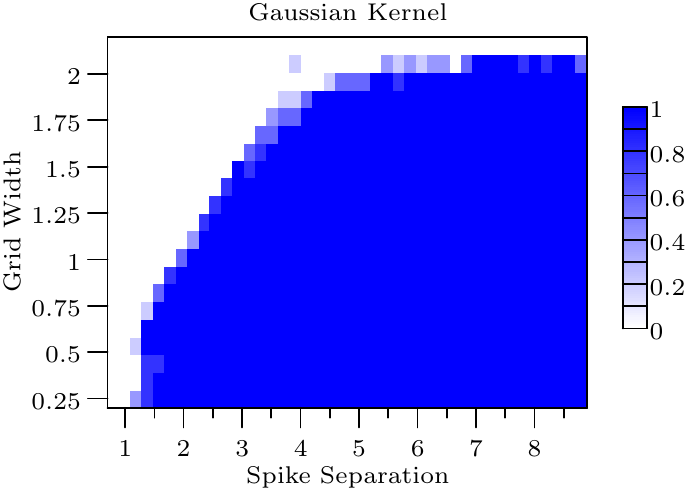}\qquad
\includegraphics{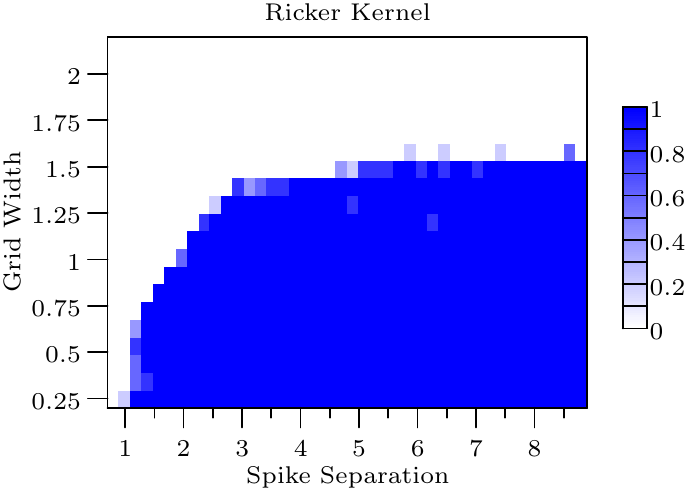}
\caption{Fraction of 5 runs yielding exact recovery for varying spike
  separations and grid widths when the sampling pattern is uniform. Spike separations
  and sampling grid widths are given in units of~$\sigma$. }
\label{fig:CVXExactRecoveryUnif}
\end{figure}

In addition, we evaluate the method for a uniform sampling pattern, varying the grid step size and the minimum separation between spikes. \Cref{fig:CVXExactRecoveryUnif} shows the results. The method recovers the signal perfectly up until a grid step size that is approximately double the largest sample proximity at which exact recovery occurs in \Cref{fig:CVXExactRecovery}. The reason may be that there are spikes that are right in between two samples, so we can interpret the grid as imposing an effective sample separation that is half the grid step size (note however that this is not completely accurate because every spike has a neighboring sample that is closer and another that is farther than this separation). 


\begin{figure}[tp]
\centering \includegraphics{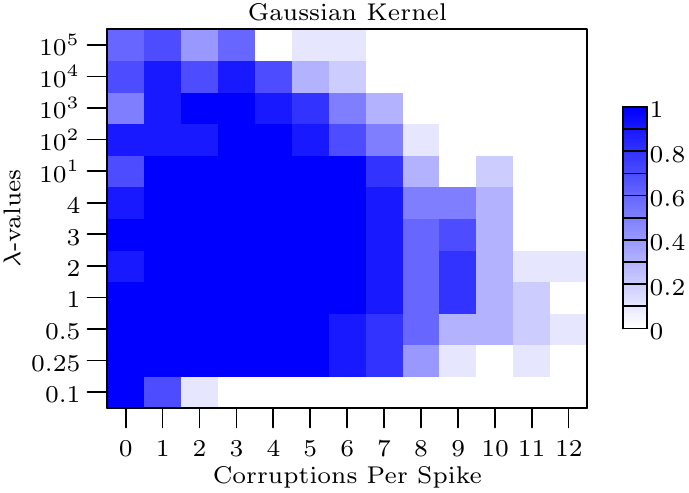}\qquad
\includegraphics{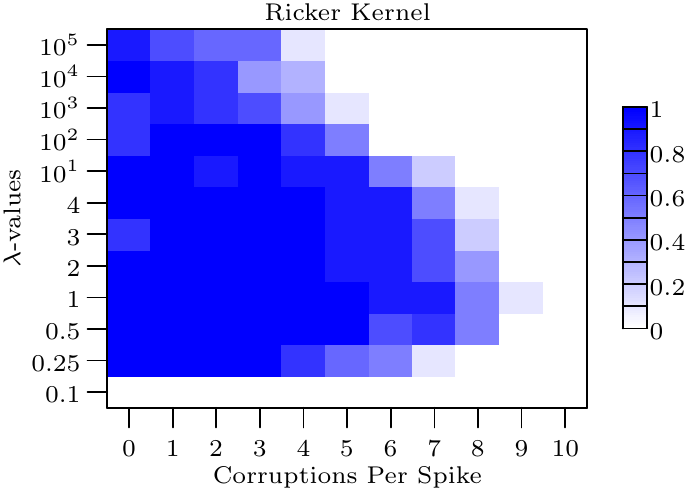}
\caption{Fraction of 10 instances yielding exact recovery for varying
  values of $\lambda$ and number of corruptions per spike. We fix
$\sigma=0.02$, 10 spikes and a spike separation of $4.5\sigma$.}
\label{fig:LambdaVsNoise}
\end{figure}

\begin{figure}[tp]
\centering 
\includegraphics{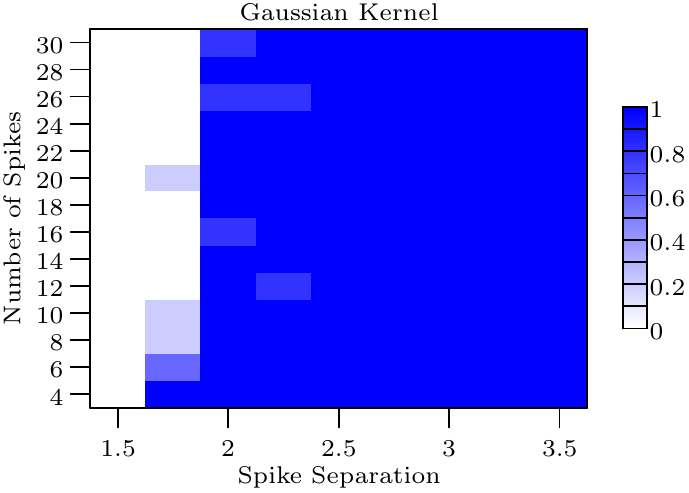}\qquad
\includegraphics{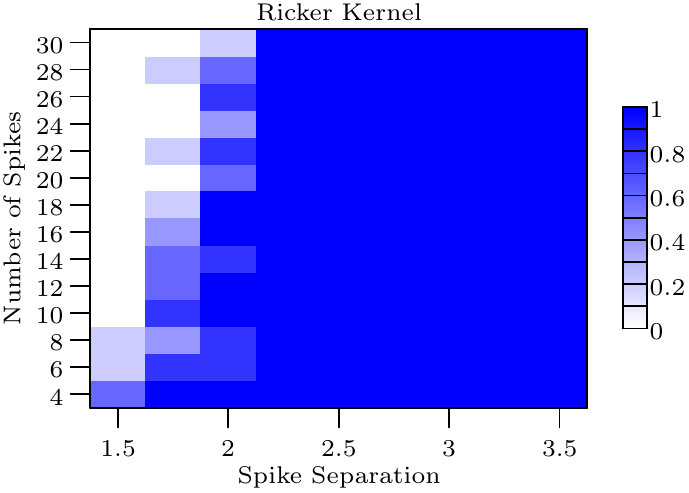}
\caption{Fraction of 5 instances yielding exact recovery with 2 corruptions
  per spike for varying numbers of spikes, and values of the spike
  separation. We fix $\sigma=0.004$, $\lambda=2$, and a sampling grid width of
$0.2\sigma$. Spike separations are given in units of~$\sigma$.}
\label{fig:SepVsNSpikes}
\end{figure}


\begin{figure}[tp]
\centering \includegraphics{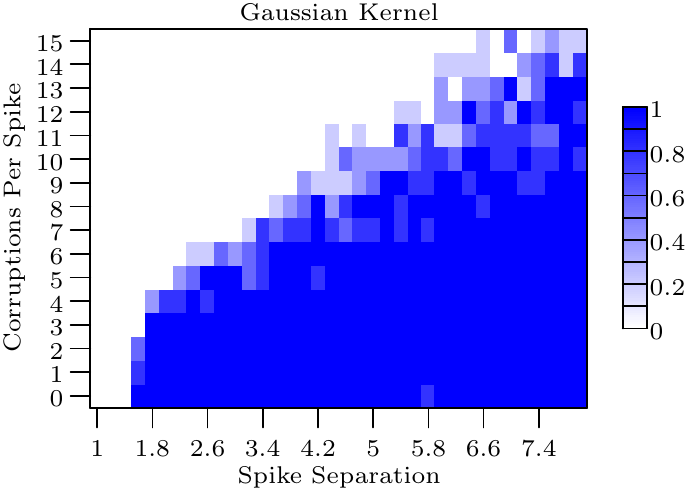}\qquad
\includegraphics{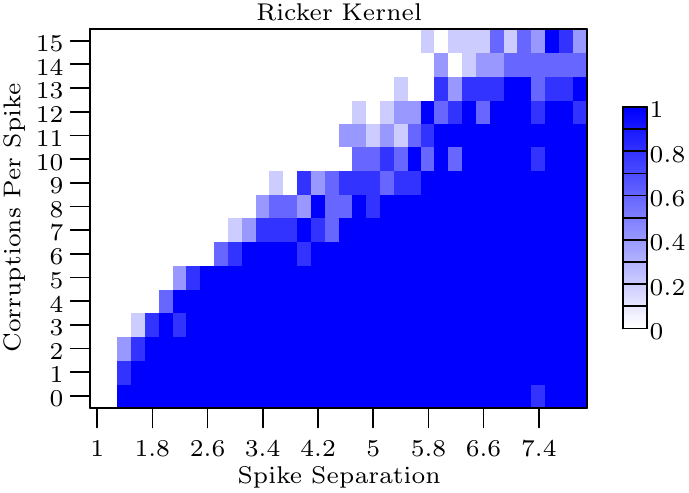}
\caption{Fraction of 5 runs yielding exact recovery with 10 spikes for
  varying spike separation and corruptions per spike. We fix $\sigma=0.01$ and $\lambda=2$. Spike
  separations are given in units of~$\sigma$.}
\label{fig:SpikeSepVsNoise}
\end{figure}

\subsection{Robustness to Sparse Noise}
\label{sec:sparse_recovery}
In this section, we investigate the performance of the optimization-based approach for deconvolution in the presence of outliers described in \Cref{sec:sparse_noise_results}. As in \Cref{sec:numexactrecovery}, we discretize Problem~\eqref{pr:SparseNoiseRecovery} by restricting the spikes in the measure to lie on a uniform grid with $5\cdot 10^4$ points and solve the resulting convex program using CVX~\cite{cvx}. We consider that the method achieves exact recovery if the relative error is $10^{-3}$ (we use a larger tolerance because we found the solutions to be less precise than in \Cref{sec:numexactrecovery}). Our experiments evaluate the effect of the following parameters:
\begin{itemize}
\setlength\itemsep{0em}
\item \emph{Number of spikes} and \emph{minimum separation}: Adjacent spike locations have separation $\Delta$ with a random jitter of $\pm 1\%$. The amplitudes are drawn iid from a standard normal distribution. 
\item \emph{Number of corruptions}: We divide the unit interval into
  segments of length $\Delta$ and corrupt $c$ samples per segment 
selected at random by drawing their amplitudes iid from a standard normal distribution. This allows $c$ to be interpreted as the number of corruptions per spike, and prevents the corruptions
from becoming too clustered.  
\item The \emph{regularization parameter} $\lambda$.
\end{itemize}
For all experiment the samples are uniformly distributed in the unit
interval on a grid of width $0.2\sigma$. \Cref{fig:LambdaVsNoise}
shows that the performance of the method is quite robust to $\lambda$,
as suggested by \Cref{lem:SparseTrim}. \Cref{fig:SepVsNSpikes} shows
that for a fixed number of corruptions (two per spike in this case)
the method succeeds as long as the minimum separation is large enough,
and that this minimum separation does not vary much for different
numbers of spikes (especially in the case of the Gaussian kernel). \Cref{fig:SpikeSepVsNoise} shows that a larger
minimum separation between spikes allows the method to achieve exact recovery in the presence of more corruptions. 

\section{Conclusions and Directions for Future Research}
\label{sec:conclusions}
In this work we establish that minimizing a continuous counterpart of the $\ell_1$ norm achieves exact recovery of signals satisfying a minimum-separation condition from samples of their convolution with a Gaussian or Ricker kernel, under the assumption that there are at least two samples that are close to each spike. Our theoretical analysis is based on a dual-certificate construction that is not specific to these convolution kernels and can be adapted to study other deconvolution problems. In addition, we show that the optimization-based deconvolution approach is robust to dense and sparse noise. Some interesting research directions are:
\begin{itemize}
\setlength\itemsep{0em}
\item Investigating the effect of additional samples (beyond two per spike) on the minimum separation at which exact recovery is achieved, as well as on robustness to noise.
\item Quantifying the discretization error incurred by solving $\ell_1$-norm minimization on a fine grid instead of solving the continuous TV-norm minimization problem, in the spirit of~\cite{duval2017sparse}.
\item Developing algorithms for TV-norm minimization on a continuous domain, see~\cite{boyd2017alternating,bredies2013inverse,rao2015forward} for some recent work in this direction.
\item Extending our proof techniques to analyze deconvolution in multiple dimensions.
\item Characterizing robustness to sparse noise under weaker conditions.
\item Extending the method to handle unknown noise levels as in~\cite{superres_unknown_noise}, which studies super-resolution of point sources.
\item Studying blind deconvolution, i.e. joint estimation of the convolution kernel and the signal, for deterministic kernels. Recent works in this direction focus on randomized measurements~\cite{ahmed2014blind,ling2017blind} and alternating optimization~\cite{repetti2015euclid}.
\end{itemize}

\subsection*{Acknowledgements}
C.F. is generously supported by NSF award DMS-1616340. C.F. thanks Waheed Bajwa for useful comments and Stan Osher for introducing him to the problem of heat-source localization via deconvolution. 

\begin{small}
\bibliographystyle{abbrv}
\bibliography{refs}
\end{small}

\appendix
\renewcommand{\appendixpagename}{Appendix}
\appendixpage
Throughout the appendix, we assume there is some compact interval
$I\subset\RR$ containing the support of the true measure $\mu$ (given in
\eqref{eq:signal}).  In problems \eqref{pr:TVnorm},
\eqref{pr:NoiseRecovery}, and \eqref{pr:SparseNoiseRecovery} 
the variable $\tilde{\mu}$ takes values in the set of finite signed Borel
measures supported on $I$.
\section{Proof of \Cref{prop:DualCert}: Certifying Exact Recovery}
\label{sec:DualCertProof}
The proof of \Cref{prop:DualCert} follows standard arguments and is included for completeness (see Section A in~\cite{superres} and also~\cite{de2012exact}).
\DualCert*
\begin{proof}
  Let $\nu$ be feasible for Problem~\eqref{pr:TVnorm}
  and define $h=\nu-\mu$.
  By taking the Lebesgue decomposition of $h$ with respect to
  $|\mu|$ we can write
  \begin{equation}
    h = h_{T}+h_{T^c},
  \end{equation}
  where $h_{T}$ is absolutely continuous with respect to $|\mu|$,
  and $h_{T^c}$ is mutually orthogonal to $|\mu|$. In other
  words, the support of $h_T$ is contained in~$T$, and $h_{T^c}(T)=0$.
  This allows us to write
  \begin{equation}
    h_{T} = \sum_{t_j \in T} b_j\delta_{t_j},
  \end{equation}
  for some $b\in\RR^{|T|}$. Set $\rho :=\sign(b)$, where we
  arbitrarily choose $\rho_j=\pm1$ if $b_j=0$. 
  By assumption there exists a corresponding
  $Q$ interpolating $\rho$ on $T$.  Since $\mu$ and $\nu$ are feasible
  for Problem~\eqref{pr:TVnorm} we have
  $(K*h)(s_i)=0$ for $i=1,\ldots,n$.  This implies
  \begin{equation}
    \label{eq:nutvbound}
    \begin{aligned}
      0 &= \sum_{i=1}^n q_i(K*h)(s_i) = \sum_{i=1}^n q_i\int K(s_i-t)\,dh(t)\\
      &= \int Q(t)\,dh(t)= \normTV{h_{T}} + \int Q(t)\,dh_{T^c}(t).
    \end{aligned}
  \end{equation}
  Applying the triangle inequality, we obtain
  \begin{align}
    \normTV{\nu} & =
    \normTV{\mu+h_{T}}+\normTV{h_{T^c}} &&\text{(Mutual Orthogonality)}\\
    & \geq  \normTV{\mu} + \normTV{h_{T^c}}-\normTV{h_T}&& \text{(Triangle Inequality)}\\
    & = \normTV{\mu} + \normTV{h_{T^c}}+\int
      Q(t)\,dh_{T^c}(t)&& \text{(Equation
        \ref{eq:nutvbound})}\\
    & \geq \normTV{\mu} && \text{($|Q(t)|\leq1$)},
  \end{align}
  where the last inequality is strict if $\normTV{h_{T^c}}> 0$ since
  $|Q(t)|<1$ for $t\in T^c$.
  This establishes that $\mu$ is optimal for Problem~\eqref{pr:TVnorm},
  and that any other optimal solution must be supported on $T$.
  Equation~\eqref{eq:nutvbound} implies that any feasible solution
  supported on $T$ must be equal to $\mu$ (since $\normTV{h_T}=0$), completing the proof of uniqueness.
\end{proof}
\section{Bumps and Waves}
\subsection{Auxiliary Results}
In this section we compile several results that will be useful in the following sections.
\begin{lemma}\label{lem:GaussDerivs}
  For $|s|\leq 1$ we have $|(\KG)^{(i)}(s)|\leq 3$ for $i=0,\ldots,4$.  If we also have
  $t\geq 10$ then 
  \begin{align}
    |(\KG)^{(i)}(s-t)| &\leq
    2t^4\exp\brac{-\frac{t^2}{2}+t-\frac{s^2}{2}},  \\
    |(\KR)^{(i)}(s-t)| &\leq 2t^4\exp\brac{-\frac{t^2}{2}+t-\frac{s^2}{2}},  
  \end{align}
  for $i=0,\ldots,2$.
\end{lemma}
\begin{proof}
  The first 4 derivatives of the Gaussian kernel are given by
  \begin{align}
    (\KG)^{(1)}(t) & = -t\KG(t), \qquad & (\KG)^{(2)}(t) & = (t^2-1)\KG(t),  \\
    (\KR)^{(1)}(t) & = t(t^2-3)\KG(t), \qquad & (\KR)^{(2)}(t) & = - (t^4-6t^2+3)\KG(t).
  \end{align}
  If $|s|\leq 1$ then $\KG(s)\leq 1$, so we just need to bound the polynomial factors.  We have
  \begin{align}
    |s| & \leq  1, \\
    |s^2-1| & \leq 1,\\
    |s(s^2-3)| & \leq |s^2-3| \leq 3.
  \end{align}
  For the remaining factor, note that $p(s)=s^4-6s^2+3$ has negative second derivative,
  \begin{align}
    p^{(2)}(s) = 12s^2-12 \leq 0,
  \end{align}
  so $p$ is even and concave when $|s|\leq1$.  This implies $|p(s)|\leq p(0)=3$.

  Next assume $t\geq 10$.  Since
  \begin{align}
    \KG(s-t)=\exp\brac{-\frac{(s-t)^2}{2}} =
    \exp\brac{-\frac{t^2}{2}+st-\frac{s^2}{2}}
    \leq \exp\brac{-\frac{t^2}{2}+t-\frac{s^2}{2}},
  \end{align}
we again only need to bound the polynomial factors. The following bounds complete the proof: 
  \begin{align}
    |s-t| & \leq 2t \leq 2t^4\\
    |(s-t)^2-1| & = t^2-2st+s^2-1 \leq t^2+2t \leq 2t^2\leq2t^4\\
    |(s-t)((s-t)^2-3)| & = t^3-3st^2+3s^2t-3t-s^3+3s \leq
    t^3+3t^2+6t+4 \leq 4t^3\leq 2t^4\\
    |(s-t)^4-6(s-t)^2+3| & \leq 
    s^4-4 s^3 t+6 s^2 t^2-6 s^2-4 s t^3+12 s t+t^4-6 t^2+3\\
    & \leq t^4+4t^3+12t^2+16t+10 \leq t^4 + 4t^3 + 2t^3 + t^3 + t^3
    \leq 2t^4.
  \end{align}
\end{proof}
The next lemma gives explicit formulas for the denominators of the
bump and wave coefficients in \Cref{lem:BumpExists}. The result is obtained by a simple expansion of the expressions, so we omit the proof. 
\begin{lemma}\label{lem:DenomFormula}
  Define
  \begin{align}
    \DG(s_1,s_2) &:= \KG(s_2)(\KG)^{(1)}(s_1)-(\KG)^{(1)}(s_2)\KG(s_1),\\
    \DR(s_1,s_2) &:= \KR(s_2)(\KR)^{(1)}(s_1)-(\KR)^{(1)}(s_2)\KR(s_1).
  \end{align}
  Then 
  \begin{align}
    \DG(s_1,s_2) &= (s_2-s_1)\exp\brac{-\frac{s_1^2+s_2^2}{2}},\\
    \DR(s_1,s_2) &=(s_2-s_1)(3-(s_1-s_2)^2+s_1^2s_2^2)\exp\brac{-\frac{s_1^2+s_2^2}{2}}.
  \end{align}
\end{lemma}
The following lemma allows us to control $\DR$.
\begin{restatable}{lemma}{RickerF}
\label{lem:RickerF}
  Let $f:\RR^2\to\RR$ be defined by $f(x,y)=3-(x-y)^2+x^2y^2$. 
  If $\max(|x|,|y|) \leq a < 1$ then $f(x,y)\geq 3-4a^2+a^4 > 0$.
\end{restatable}
\begin{proof} Given in \Cref{sec:RickerFProof}. \end{proof}
\subsection{Proof of \Cref{lem:BumpExists}: Existence of the
  Bumps and Waves}
\label{sec:BumpExistsProof}
\BumpExists*
\begin{proof}
  In matrix form, the bump and wave coefficients satisfy
  \begin{align}
    \MAT{K(\tilde{s}_{i,1}-t_i) & K(\tilde{s}_{i,2}-t_i)\\
    -K^{(1)}(\tilde{s}_{i,1}-t_i) & -K^{(1)}(\tilde{s}_{i,2}-t_i)}
 \MAT{b_{i,1} & w_{i,1} \\ b_{i,2} & w_{i,2}} 
  &= \MAT{1&0\\0&1}.
  \end{align}
  We obtain the formulas in the statement by solving the system.  Without loss of generality, we set
  $t_i=0$.  The conditions for the denominators to be non-zero then follow directly from \Cref{lem:DenomFormula,lem:RickerF}.
\end{proof}
\subsection{Proof of \Cref{lem:EvenKernel}: Symmetry of the Bumps
  and Waves}
\label{sec:EvenKernelProof}
\EvenKernel* 
\begin{proof}
  By \Cref{lem:BumpExists} we have
  \begin{align}
    B(t,s_1,s_2)
    & = \frac{-K^{(1)}(s_2)K(s_1-t) + K^{(1)}(s_1)K(s_2-t)}
    {K(s_2)K^{(1)}(s_1)-K^{(1)}(s_2)K(s_1)} \\
    & = \frac{K^{(1)}(-s_2)K(-s_1+t) - K^{(1)}(-s_1)K(-s_2+t)}
    {-K(-s_2)K^{(1)}(-s_1)+K^{(1)}(-s_2)K(-s_1)} \\
    & = B(-t,-s_2,-s_1)\\
    & = B(-t,-s_1,-s_2)\\
  \end{align}
  since $K$ is even and $K^{(1)}$ is odd.  Similarly we have
  \begin{align}
    W(t,s_1,s_2)
    & = \frac{-K(s_2)K(s_1-t) + K(s_1)K(s_2-t)}
    {K(s_2)K^{(1)}(s_1)-K^{(1)}(s_2)K(s_1)} \\
    & = \frac{-K(-s_2)K(-s_1+t) + K(-s_1)K(-s_2+t)}
    {-K(-s_2)K^{(1)}(-s_1)+K^{(1)}(-s_2)K(-s_1)} \\
    & = -W(-t,-s_1,-s_2).
  \end{align}
\end{proof}

\subsection{Proof of \Cref{lem:TailDecay}: Decay of the Bumps and
  Waves}
\label{sec:TailDecayProof}
\TailDecay*
For ease of notation we denote the bump and wave coefficients by $b_1$, $b_2$, $w_1$ and $w_2$. By the triangle inequality we have
\begin{align}
    \left|\frac{\partial^i}{\partial t^i}B(t,s_1,s_2)\right| & \leq
    |b_1K^{(i)}(s_1-t)|+|b_2K^{(i)}(s_2-t)|,\label{eq:bumpabs}\\
    \left|\frac{\partial^i}{\partial t^i}W(t,s_1,s_2)\right| & \leq
    |w_1K^{(i)}(s_1-t)|+|w_2K^{(i)}(s_2-t)|,\label{eq:waveabs}
\end{align}
for $i=0,1,2$.  The following lemmas bound these quantities for the Gaussian and Ricker kernels.
\begin{lemma}\label{lem:GaussDecayBound}
  Fix a sample proximity $\gamma(S,T)\leq 1$, a sample separation
  $\kappa(S)$, and $t\geq10$.  Then for the Gaussian kernel and $i=0,1,2$ we have
  \begin{align}
    \left|\frac{\partial^i}{\partial t^i}B(t,s_1,s_2)\right| & \leq
    \frac{4t^4}{\kappa(S)}\exp\brac{-\frac{t^2}{2}+t},\\
    \left|\frac{\partial^i}{\partial t^i}W(t,s_1,s_2)\right| & \leq \frac{4t^4}{\kappa(S)}\exp\brac{-\frac{t^2}{2}+t}.
  \end{align}
\end{lemma}
\begin{proof}
  For the Gaussian kernel,
  \begin{align}
    |b_1(\KG)^{(i)}(s_1-t)|
    & = \frac{|(\KG)^{(i)}(s_1-t)(\KG)^{(1)}(s_2)|}{|\DG(s_1,s_2)|}
    &&\text{(\Cref{lem:BumpExists})}\\
    & \leq
    \frac{2t^4\exp\brac{-\frac{t^2}{2}+t-\frac{s_1^2}{2}}|s_2|\exp\brac{-\frac{s_2^2}{2}}}
    {|s_2-s_1|\exp\brac{-\frac{s_1^2+s_2^2}{2}}}&&\text{(\Cref{lem:DenomFormula,lem:GaussDerivs})}\\
    & \leq 
    \frac{2t^4\exp\brac{-\frac{t^2}{2}+t}}{\kappa(S)}.
  \end{align} 
  The bound also holds for $|b_2\KG(s_2-t)|$ by the same argument. For the wave, 
  \begin{align}
    |w_1(\KG)^{(i)}(s_1-t)|
    & = \frac{|(\KG)^{(i)}(s_1-t)\KG(s_2)|}{|\DG(s_1,s_2)|}
    &&\text{(\Cref{lem:BumpExists})}\\
    & \leq
    \frac{2t^4\exp\brac{-\frac{t^2}{2}+t-\frac{s_1^2}{2}}\exp\brac{-\frac{s_2^2}{2}}}
    {|s_2-s_1|\exp\brac{-\frac{s_1^2+s_2^2}{2}}}&&\text{(\Cref{lem:DenomFormula,lem:GaussDerivs})}\\
    & \leq 
    \frac{2t^4\exp\brac{-\frac{t^2}{2}+t}}{\kappa(S)}.
  \end{align}
  The same argument can be applied to $w_2(\KG)^{(i)}(s_2-t)$. 
\end{proof}
\begin{lemma}\label{lem:RickerDecayBound}
  Fix a sample proximity $\gamma(S,T)\leq 0.8$, a sample separation
  $\kappa(S)$, and $t\geq 10$.  Then for the Ricker kernel we have
  \begin{align}
    \left|\frac{\partial^i}{\partial t^i}B(t,s_1,s_2)\right| & \leq
    \frac{11.3t^4}{\kappa(S)}\exp\brac{-\frac{t^2}{2}+t},\\
    \left|\frac{\partial^i}{\partial t^i}W(t,s_1,s_2)\right| & 
    \leq \frac{4.8t^4}{\kappa(S)}\exp\brac{-\frac{t^2}{2}+t},
  \end{align}
  for $i=0,1,2$.
\end{lemma}
\begin{proof}
  We have
  \begin{align}
    |b_1(\KR)^{(i)}(s_1-t)|
    & = \frac{|(\KR)^{(i)}(s_1-t)(\KR)^{(1)}(s_2)|}{|\DR(s_1,s_2)|}
    &&\text{(\Cref{lem:BumpExists})} \notag \\
    & \leq
    \frac{2t^4\exp\brac{-\frac{t^2}{2}+t-\frac{s_1^2}{2}}|s_2(s_2^2-3)|\exp\brac{-\frac{s_2^2}{2}}}
    {|s_2-s_1|(3-(s_1-s_2)^2+s_1^2s_2^2)\exp\brac{-\frac{s_1^2+s_2^2}{2}}}
    &&\text{(\Cref{lem:DenomFormula,lem:GaussDerivs})} \notag\\
    & \leq 
    \frac{4.8t^4\exp\brac{-\frac{t^2}{2}+t}}{\kappa(S)(3-4(0.8)^2+0.8^4)}
    &&\text{(\Cref{lem:RickerF})}\\
    & \leq 
    \frac{5.65t^4\exp\brac{-\frac{t^2}{2}+t}}{\kappa(S)}.
  \end{align}
  The same argument applies to $|b_2\KR(s_2-t)|$.  For the wave,
    \begin{align}
    |w_1(\KR)^{(i)}(s_1-t)|
    & = \frac{|(\KR)^{(i)}(s_1-t)\KR(s_2)|}{|\DR(s_1,s_2)|}
    &&\text{(\Cref{lem:BumpExists})} \notag \\
    & \leq
    \frac{2t^4\exp\brac{-\frac{t^2}{2}+t-\frac{s_1^2}{2}}|s_2^2-1|\exp\brac{-\frac{s_2^2}{2}}}
    {|s_2-s_1|(3-(s_1-s_2)^2+s_1^2s_2^2)\exp\brac{-\frac{s_1^2+s_2^2}{2}}}
    &&\text{(\Cref{lem:DenomFormula,lem:GaussDerivs})} \notag \\
    & \leq 
    \frac{2t^4\exp\brac{-\frac{t^2}{2}+t}}{\kappa(S)(3-4(0.8)^2+0.8^4)}
    &&\text{(\Cref{lem:RickerF})}\\
    & \leq 
    \frac{2.4t^4\exp\brac{-\frac{t^2}{2}+t}}{\kappa(S)}.
  \end{align}
  The same argument applies to $|w_2(\KR)^{(i)}(s_2-t)|$.
\end{proof}
The proof is completed by combining \Cref{lem:GaussDecayBound,lem:RickerDecayBound} with the following result, which can be applied to both the Gaussian kernel and the Ricker wavelet.
\begin{lemma}\label{lem:SumTailBound}
  Suppose $f(t)\leq 20t^4\exp(-t^2/2+t)$ for $t\geq10$.  Then
  $$\sum_{k=j}^\infty
  f(c+kd) \leq 10^{-12},$$
  for any $c,d\in\RR$ and $j\in\ZZ_{>0}$ with $d\geq1$ and $c+jd\geq10$.
\end{lemma}
\begin{proof}
  Let $p(t):=20t^4\exp(-t^2/2+t)$.  
For $t\geq 10$
  \begin{align}
    p^{(1)}(t) & = 
    80t^3e^{-t^2/2+t}+20t^4e^{-t^2/2+t}(1-t)\\
    & \leq 8t^4e^{-t^2/2+t}+20t^4e^{-t^2/2+t}(1-t)\\
    & = (28-20t)t^4e^{-t^2/2+t}\\
    & < 0,
  \end{align}
  and
  \begin{align}
    \frac{p(t)}{p(t+1)} & = 
    \frac{20t^4\exp(-t^2/2+t)}{20(t+1)^4\exp(-(t+1)^2/2+(t+1))}\\
    & = \frac{t^4}{(t+1)^4}\exp(((t+1)^2-t^2)/2-1)\\
    & = \frac{t^4}{t^4+4t^3+6t^2+4t+1}\exp(t-1/2)\\
    & \geq \frac{1}{2e^{1/2}}e^t\\
    & > 2.
  \end{align}  
  Combining these facts, we obtain a geometric-series bound, which we sum to complete the proof:
  \begin{align}
    \sum_{k=j}^\infty f(c+kd) &\leq
    \sum_{k=j}^\infty p(c+kd) 
    \leq \sum_{i=10}^\infty p(i)\\
    & \leq p(10)\sum_{i=0}^\infty 2^{-i} = 2p(10)\leq 10^{-12}.
  \end{align}
\end{proof}

\subsection{Proof of \Cref{lem:Piecewise}: Piecewise-Constant Bounds}
\label{sec:PiecewiseProof}
\Piecewise*
In this section we define piecewise-constant upper bounds for
$\sym{|B^{(i)}|}$, $\sym{|W^{(i)}|}$, $\sym{B^{(i)}}$ where $i=0,1,2$
and $t\in[0,10)$.  We begin by partitioning the interval $[0,10)$ into $N_1$ segments of the form
\begin{align}
  \ml{U}_j &:= \left[\frac{10 \brac{j-1}}{N_1},\frac{10 j}{N_1}\right), \quad 1 \leq j \leq N_1,
\end{align}
and by dividing the interval $[-\gamma(S,T),\gamma(S,T)]$ into $2N_2$ almost
disjoint segments of the form
\begin{align}
  \ml{V}_k &:= \left[\frac{\brac{k-1} \gamma(S,T)}{N_2},\frac{ k \gamma(S,T) }{N_2}\right],  \quad -N_2+1 \leq k \leq N_2.
\end{align}
Next we define bounds on our functions of interest when $t \in \ml{U}_j$, $s_1 \in \ml{V}_k$ and $s_2 \in \ml{V}_l$
\begin{align}
  |b^{(i)}|_{j,k,l} &  \geq \sup |B^{(i)}(\ml{U}_j,\ml{V}_k,\ml{V}_l)|,\\
  |w^{(i)}|_{j,k,l} &  \geq \sup |W^{(i)}(\ml{U}_j,\ml{V}_k,\ml{V}_l)|,\\
  b^{(i)}_{j,k,l} &  \geq \sup B^{(i)}(\ml{U}_j,\ml{V}_k,\ml{V}_l),
\end{align}
where $i=0,1,2$, $1 \leq j \leq N_1$ and $-N_2+1 \leq k,l \leq N_2$. To be clear, $|B^{(i)}(\ml{U}_j,\ml{V}_k,\ml{V}_l)|$ is the image of the set $\ml{U}_j \times \ml{V}_k \times \ml{V}_l$ with respect to the function $|B^{(i)}|$. These bounds can be computed by interval arithmetic using the fact that the functions $B^{(i)}$ and $W^{(i)}$ can be expressed in terms of exponentials and polynomials. We provide a small description of interval arithmetic in \Cref{sec:IntervalArithmetic}. To compute the bounds in practice, we apply the interval-arithmetic library in Mathematica, which we have found to be sufficient for our purposes. The code is available online\footnote{\url{http://www.cims.nyu.edu/~cfgranda/scripts/deconvolution_proof.zip}}.

To define piecewise bounds that are valid on each $\ml{U}_j$, we maximize over the sets in the partition of $[-\gamma(S,T),\gamma(S,T)]^2$ which contain points that satisfy the sample-separation condition in \Cref{def:ProxSep}. Consider the distance between $\ml{V}_k$ and $\ml{V}_l$
\begin{align} 
  d(\ml{V}_k,\ml{V}_l) := \inf_{a\in \ml{V}_k ,b\in \ml{V}_l } |a-b|.
\end{align}
The following lemma implies that it is sufficient to consider the pairs $\ml{V}_k$, $\ml{V}_l$ such that $d(\ml{V}_k,\ml{V}_l)\geq c:=\kappa(S)-2\gamma(S,T)/N_2$. 
\begin{lemma} 
 If $d(\ml{V}_k,\ml{V}_l) < c:=\kappa(S)-2\gamma(S,T)/N_2$ then no pair of points $s_1 \in \ml{V}_k$, $s_2 \in \ml{V}_l$ satisfy the sample-separation condition $\abs{s_1 - s_2} \geq \kappa(S)$.
\end{lemma}    
\begin{proof}
Since $\ml{V}_k \times \ml{V}_l$ is compact there exist two points $\tilde{a}  \in \ml{V}_k$ and $\tilde{b} \in \ml{V}_l$ such that $|\tilde{a} - \tilde{b}| = d(\ml{V}_k,\ml{V}_l)$. By the triangle inequality, for any pair of points $s_1 \in \ml{V}_k$, $s_2 \in \ml{V}_l$
\begin{align}
\abs{s_1 - s_2} & \leq \abs{s_1 - \tilde{a}} + \abs{\tilde{a} - \tilde{b}} + \abs{\tilde{b} - s_2} \\
& <  \frac{2\gamma(S,T)}{N_2} + c = \kappa(S).
\end{align}
\end{proof}
Finally, we obtain the following bounds on $\sym{|B^{(i)}|}\brac{t}$, $\sym{|W^{(i)}|}(t)$ and $\sym{B^{(i)}}\brac{t}$ for $t \in \ml{U}_j$ and $i=0,1,2$
 \begin{align}
  \wt{|B^{(i)}|}(j) &:= \max_{d(\ml{V}_k,\ml{V}_l)\geq c}
  |b^{(i)}|_{j,k,l},\label{eq:AbsBumpApprox}\\
  \wt{|W^{(i)}|}(j) &:= \max_{d(\ml{V}_k,\ml{V}_l)\geq c}
  |w^{(i)}|_{j,k,l},\label{eq:AbsWaveApprox}\\
  \wt{B^{(i)}}(j) &:= \max_{d(\ml{V}_k,\ml{V}_l)\geq c}
  b^{(i)}_{j,k,l}.\label{eq:BumpApprox}
\end{align}
These bounds are parametrized by a finite number of points and can therefore be computed explicitly as long as we fix $N_1$ and $N_2$.  
\Cref{fig:MaxRegion} shows, for a fixed $\ml{U}_j$, the subset of
$[-\gamma(S,T),\gamma(S,T)]^2$ used to compute $\wt{|B|^{(i)}}(j)$,
$\wt{|W|^{(i)}}(j)$, and $\wt{B^{(i)}}(j)$.

\begin{figure}[t]
\centering
\includegraphics{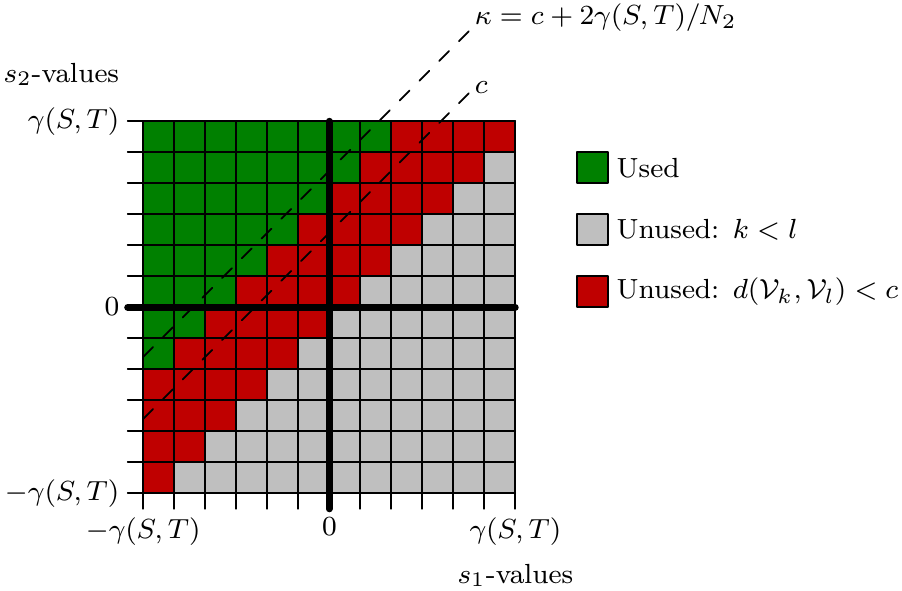}
\caption{Squares $\ml{V}_k\times \ml{V}_l$ used to compute $\wt{|B|^{(i)}}(j)$,
    $\wt{|W|^{(i)}}(j)$, and $\wt{B^{(i)}}(j)$ for a fixed $\ml{U}_j$.}
\label{fig:MaxRegion}
\end{figure}

\subsection{Interval Arithmetic}
\label{sec:IntervalArithmetic}
This section briefly describes interval arithmetic, a technique developed to handle rounding and measurement errors in numerical analysis. In this work, we apply it to derive the piecewise-constant bounds in \Cref{sec:PiecewiseProof}. Interval arithmetic allows us to compute an interval which contains the possible values of a function when its arguments belong to given intervals. To be precise, let $f:\RR^n\to\RR$ be a function that can be expressed by composing predefined functions (such as monomials or exponentials) with standard arithmetic operations. Given $n$ intervals of the form $\sqbr{a_i,b_i}\subseteq\RR$, $1\leq i \leq n$, the goal is to obtain an interval $[c,d]\subseteq\RR$ such that  
\begin{equation*}
  f([a_1,b_1]\times\cdots\times[a_n,b_n])\subseteq[c,d].
\end{equation*}
For a set $\ml{A}$, $f\brac{\ml{A}}$  denotes the image of $\ml{A}$. Interval-arithmetic methods usually compute $[c,d]$ recursively by applying simple identities to obtain intervals containing the output of the functions that are composed to obtain $f$. The following lemma lists some of these expressions. 
\begin{lemma}[Proof in \cite{hickey2001interval}]\label{lem:IntervalArithmetic}
  Let $f_1,f_2:\RR^n\to\RR$, $\ml{A}\subseteq\RR$, 
  $f_1(\ml{A})\subseteq[c_1,d_1]$ and $f_2(\ml{A})\subseteq [c_2,d_2]$ for some
  $c_1,d_1,c_2,d_2\in\RR$.  Then the following hold.
  \begin{enumerate}
  \item $|f_1(\ml{A})| \subseteq [\max(0,c_1),\max(|c_1|,|d_1|)]$.
  \item $f_1(\ml{A})+f_2(\ml{A})\subseteq [c_1+c_2,d_1+d_2]$.
  \item $-f_1(\ml{A})\subseteq [-d_1,-c_1]$.
  \item $f_1(\ml{A})f_2(\ml{A})\subseteq
    [\min(c_1d_1,c_1d_2,c_2d_1,c_2d_2),\max(c_1d_1,c_1d_2,c_2d_1,c_2d_2)]$.
  \item Let $k\in\ZZ_{>0}$. If $0$ is not in $ [c_1,d_1]$ then
    $f_1(\ml{A})^{2k} \subseteq [\min(c_1^{2k},d_1^{2k}),\max(c_1^{2k},d_1^{2k})]$. If $0$ is in $ [c_1,d_1]$ then $f_1(\ml{A})^{2k}\subseteq[0,\max(c_1^{2k},d_1^{2k})]$.
  \item If $0\notin [c_1,d_1]$, $1/f_1(\ml{A})\subseteq[1/d_1,1/c_1]$.
  \item If $g:\RR\to\RR$ is increasing $g(f_1(\ml{A}))\subseteq[g(c_1),g(d_1)]$.
  \end{enumerate}
\end{lemma}
To conclude we illustrate how interval arithmetic works with a simple example.
\begin{example}
  Let
  \begin{align}
    f(x,y) := \frac{(x-y)^2xy}{x}
  \end{align}
  where $x\in [1,2]$ and $y\in[-3,3]$.  Then by \Cref{lem:IntervalArithmetic}
  \begin{align}
    -y &\in [-3,3],&x-y&\in [-2,5], & (x-y)^2&\in [0,25],\\
    xy &\in [-6,6],& 1/x &\in [1/2,1],&(x-y)^2xy &\in [-150,150],
  \end{align}
  and $f(x,y):= \frac{(x-y)^2xy}{x}\subseteq [-150,150]$ for any $x,y$ in the predefined intervals.
  Note that if we perform the algebraic simplification $f(x,y)=(x-y)^2y$, the result is improved to $(x-y)^2y\in[-75,75]$, a tighter interval.  This shows that interval arithmetic bounds are sensitive to how we express the function.
\end{example}

\subsection{Proof of Lemma~\ref{lem:RickerF}}
\label{sec:RickerFProof}
\RickerF*
\begin{proof}
For fixed $a$, the continuous function $f$ is restricted to a
  compact set, and thus must attain its minimum.  We determine the
  critical points of $f$, and then show the minimum must occur on the
  boundary. Note that the gradient of $f$ is given by
  \begin{align}
    \nabla f(x,y) = (-2(x-y)+2xy^2,2(x-y)+2x^2y).
  \end{align}
  Assume $\nabla f(x,y)=0$ and sum the components giving
  \begin{align}
    0 = 2xy^2 + 2x^2y = 2xy(x+y).
  \end{align}
  This implies $x=0$, $y=0$ or $x=-y$.  If $x=0$ then 
  \begin{align}
    0 = \nabla f(0,y) = (2y,-2y)
  \end{align}
  shows $y=0$.  By symmetry, $y=0$ implies $x=0$.  Finally, if
  $x=-y$ then
  \begin{align}
    0 = \nabla f(x,-x) = (-4x+2x^3,4x-2x^3) = 2x(x^2-2,2-x^2).
  \end{align}
  Thus the critical points of $f$ are $(0,0)$, $(\sqrt{2},-\sqrt{2})$,
  $(-\sqrt{2},\sqrt{2})$.  Restricting to $|x|,|y|\leq a<1$, only the origin is a
  critical point with value $f(0,0)=3$.  Fixing $x=a$,
    \begin{align}
    f(a,y)=3-(a-y)^2+a^2y^2 = 3-a^2+(a^2-1)y^2+2ay
      \end{align}
  is a concave quadratic in $y$ minimized at $y=-a$.  Similarly, if
  $x=-a$ 
    \begin{align}
    f(-a,y)=3-(-a-y)^2+a^2y^2 = 3-a^2+(a^2-1)y^2-2ay
    \end{align}
  is minimized at $y=a$.  As $f$ is symmetric, we have accounted for all
  points on the boundary, and have shown that the minimum occurs at
  \begin{align}
    f(a,-a)=3-4a^2+a^4 = (a^2-3)(a^2-1).
  \end{align}
  This is positive for $|a|<1$ completing the proof.
\end{proof}
\section{Proof of \Cref{lem:GenericQBounds}: Bounding the $Q$ Function}
\label{sec:GenericQBoundsProof}
\GenericQBounds*
\begin{proof}
Suppose the spikes are numbered so that $t_1=0$.
We have, for $p=0,1,2$,
\begin{align}
   |\sB^{(p)}(t)| 
   &= \left|\alpha_1B_1^{(p)}(t)+\sum_{t_j>0}\alpha_jB_j^{(p)}(t)
   +\sum_{t_j<0}\alpha_jB_j^{(p)}(t)\right|\\
   &\leq \normInf{\alpha}\left(|B^{(p)}(t,\tilde{s}_{1,1},\tilde{s}_{1,2})|
   + \sum_{t_j>0}|B^{(p)}(t-t_j,\tilde{s}_{j,1}-t_j,\tilde{s}_{j,2}-t_j)|\right.\\
   &\left.\quad+ \sum_{t_j<0}|B^{(p)}(t-t_j,\tilde{s}_{j,1}-t_j,\tilde{s}_{j,2}-t_j)|\right)\\
   & \leq \normInf{\alpha}\left(\smo{|B^{(p)}|}(t)
      + \sum_{t_j>0}\smo{|B^{(p)}|}(t_j-t)+ \sum_{t_j<0}\smo{|B^{(p)}|}(t-t_j)\right)
   \\ 
   &\leq \normInf{\alpha}\left(\smo{|B^{(p)}|}(v)
      + \sum_{j=1}^\infty\smo{|B^{(p)}|}(j\Delta(T)-v)+\smo{|B^{(p)}|}(v+j\Delta(T))\right)
   && \text{(Monotonicity)}\notag \\
   &\leq \normInf{\alpha}\left(\smo{|B^{(p)}|}(v)+2\epsilon
      + \sum_{j=1}^5\smo{|B^{(p)}|}(j\Delta(T)-v)+\smo{|B^{(p)}|}(v+j\Delta(T))\right)
   && \text{(\Cref{lem:TailDecay})}  . \notag 
\end{align}
This establishes \eqref{eq:GQBoundAbsB}. The same argument applied to $\smo{|W^{(p)}|}$ in place of
$\smo{|B^{(p)}|}$ yields \eqref{eq:GQBoundAbsW}.  For $q=1,2$, if in addition we 
assume $\sym{B^{(q)}}(t)\leq 0$, then
\begin{align}
   \sB^{(q)}(t) 
   &= \alpha_1B_1^{(q)}(t)+\sum_{t_j>0}\alpha_jB_j^{(q)}(t)
   +\sum_{t_j<0}\alpha_jB_j^{(q)}(t)\\
   &\leq \alpha_{\LB}B^{(q)}(t,\tilde{s}_{1,1},\tilde{s}_{1,2})+ \normInf{\alpha}\left(
      \sum_{t_j>0}|B^{(q)}(t-t_j,\tilde{s}_{j,1}-t_j,\tilde{s}_{j,2}-t_j)|\right.\\
   &\left.\quad+ \sum_{t_j<0}|B^{(q)}(t-t_j,\tilde{s}_{j,1}-t_j,\tilde{s}_{j,2}-t_j)|\right)\\
   &\leq \alpha_{\LB}\sym{B^{(q)}}(t)+
      \normInf{\alpha}\left(
      \sum_{t_j>0}\smo{|B^{(q)}|}(t_j-t)+ \sum_{t_j<0}\smo{|B^{(q)}|}(t-t_j)\right)
   \\ 
   &\leq \alpha_{\LB}\sym{B^{(q)}}(t)+
      \normInf{\alpha}\left(
      \sum_{j=1}^\infty\smo{|B^{(q)}|}(j\Delta(T)-v)+\smo{|B^{(q)}|}(v+j\Delta(T))\right)
   && \text{(Monotonicity)} \notag  \\
   &\leq \alpha_{\LB}\sym{B^{(q)}}(t)+
   \normInf{\alpha}\left(2\epsilon
      + \sum_{j=1}^5\smo{|B^{(q)}|}(j\Delta(T)-v)+\smo{|B^{(q)}|}(v+j\Delta(T))\right)
   && \text{(\Cref{lem:TailDecay})}.  \notag   
\end{align}
The above argument requires $\alpha_{\LB}\geq 0$.  In
\Cref{fig:alphaLBValues} we apply \Cref{lem:Schur} to compute the
region where $\alpha_{\LB}\geq0$ for the Gaussian and Ricker kernels.
The parameters values we use in the computation are the same as used to compute
\Cref{fig:MatrixBounds} at the end of \Cref{sec:CoeffBoundProof}.  The
regions where $\alpha_{\LB}\geq0$ for the two kernels are 
strictly contained within the exact recovery regions stated in
\Cref{thm:Exact}, as required.
\begin{figure}[t]
\centering 
\includegraphics{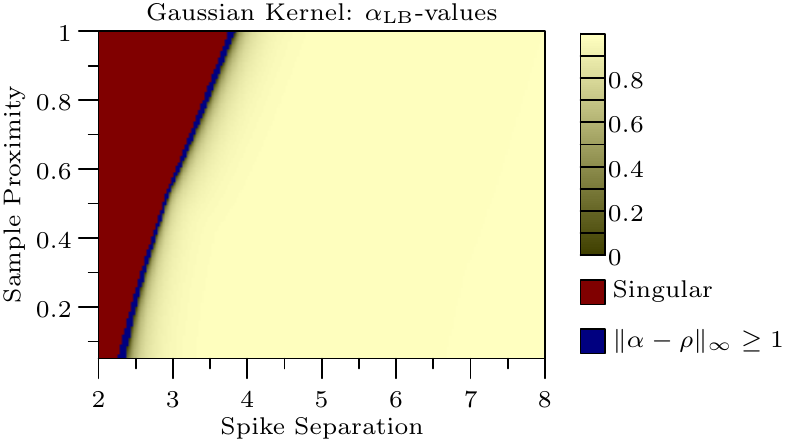}\quad
\includegraphics{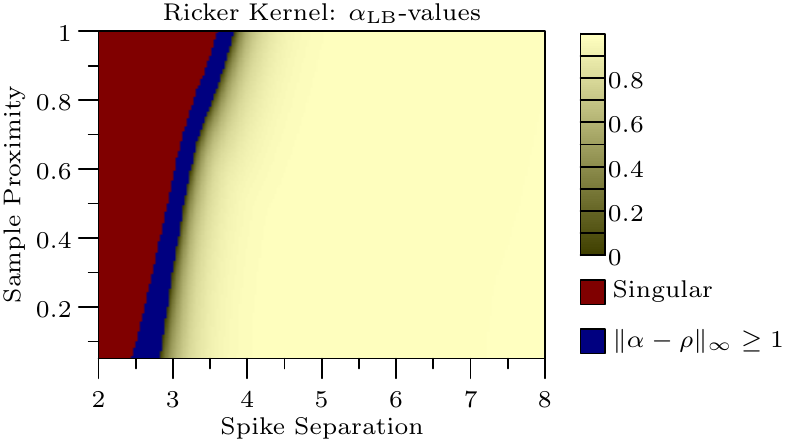}
  \caption{Lower bounds on $\alpha_{\LB}$ for different values of the
    minimum separation and the sample proximity.  The singular region
    denotes where we are unable to prove the system of interpolation
    equations is invertible.}
  \label{fig:alphaLBValues}
\end{figure}

\end{proof}


\section{Duality for Problem~\eqref{pr:NoiseRecovery}}
Recall that $I$ is a compact interval containing the support of
$\mu$, and that $\tilde{\mu}$ in problem~\eqref{pr:NoiseRecovery}
takes values in the set of finite signed Borel measures supported on $I$.
\begin{lemma}
  \label{lem:NoiseRecoveryDual}
  Problem~\eqref{pr:NoiseRecovery} is the Lagrange dual of
  \begin{equation}
    \label{pr:NoiseRecoveryDual}
    \begin{aligned}
      \text{maximize} \quad& c^Ty - \bxi\normTwo{c}\\
      \text{subject to} \quad&
      \sup_{t\in I}\left|\sum_{i=1}^n c_iK(s_i-t)\right|\leq 1,
    \end{aligned}
  \end{equation}
  and strong duality holds.  Furthermore, optimal solutions exist for
  both the primal and dual problems.
\end{lemma}
\begin{proof}
  For convenience, we minimize
  $-c^Ty+\bxi\normTwo{c}$ instead.
  The Lagrangian is given by
  \begin{align}
    \cL(c,\nu,\lambda) 
    & = -c^Ty + \bxi\normTwo{c} + \int
    \brac{ C(t)-1} \,\diff{\nu}(t) +\int \brac{-C(t)-1} \,d\lambda(t) 
    \label{eq:NoiseDualLagrange}\\
    & = -c^Ty + \bxi\normTwo{c} + \int C(t)\,d(\nu-\lambda)(t)  -
    \normTV{\nu+\lambda} \\
    & = \sum_{i=1}^n c_i( (K*(\nu-\lambda))(s_i) - y_i) 
    + \bxi\normTwo{c}-\normTV{\nu+\lambda},
  \end{align}
  where $C(t):=\sum_{i=1}^n c_iK(s_i-t)$ and $\nu,\lambda$ are finite
  non-negative Borel measures on $I$.  Here the measures $\nu,\lambda$
  act as the Lagrange multipliers (see \cite{shapiro2009semi}).
  To compute the dual function 
  $g$, we minimize the Lagrangian over the
  primal variable, $g(\nu,\lambda):=\inf_c\cL(c,\nu,\lambda)$.
  Define $v\in \RR^n$ by $v_i := (K*(\nu-\lambda))(s_i) - y_i$. 
  Since $c$ occurs in an inner product with $v$ and in an
  $\ell_2$-norm, 
  we can set $c :=-\alpha v/\normTwo{v}$ and minimize over $\alpha\geq0$, so that
  \begin{equation}
    \cL(-\alpha v/\normTwo{v},\nu,\lambda) = 
    -\alpha\normTwo{v} +
    \alpha\bxi-\normTV{\nu+\lambda}
    = \alpha(\bxi-\normTwo{v})-\normTV{\nu+\lambda}.
  \end{equation}
  Unless $\normTwo{v}\leq\bxi$, the minimum is $-\infty$.  Under this constraint, the minimum occurs
  at $\alpha:=0$, which yields the dual problem
  \begin{align}
    \text{maximize} \quad& -\normTV{\nu+\lambda}\\
    \text{subject to} \quad& \sum_{i=1}^n \brac{
      (K*(\nu-\lambda))(s_i) - y_i}^2 \leq \bxi.
  \end{align}
  Fix $\zeta :=\nu-\lambda$ and let $\zeta = \zeta_+-\zeta_-$ be the
  Jordan decomposition, where $\zeta_+,\zeta_-$ are the positive and
  negative parts of $\zeta$, respectively.  Then for any feasible
  $\zeta$ we see that
  $\normTV{\nu+\lambda}$ is minimized when $\nu=\zeta_+$ and
  $\lambda=\zeta_-$.  Up to the negation we introduced for convenience,
  this is exactly Problem~\eqref{pr:NoiseRecovery}.

  To prove strong duality we apply Theorem~2.3 of
  \cite{shapiro2009semi}, noting that the objective is finite by weak
  duality, and that $c=0$ satisfies Slater's condition.  The same
  theorem also shows that the Problem~\eqref{pr:NoiseRecovery} has a minimizing
  solution $\mu^*$. To prove Problem~\eqref{pr:NoiseRecoveryDual} has a
  maximizing solution $c^*$, we will apply Theorem~2.2 of
  \cite{shapiro2009semi} which states the following sufficient
  condition: there is a neighborhood $\cN$ of $0\in\RR^n$ such that
  for every $w\in\cN$ there are finite Borel measures $\nu,\lambda$ on
  $I$ satisfying
  \begin{equation}
    \label{eq:ShapTest}
    \inf_{c\in\RR^n} \cL(c,\nu,\lambda) -w^Tc = -c^T(y+w) + \bxi\normTwo{c} + \int
    C(t)-1\,d\nu(t) +\int -C(t)-1\,d\lambda(t)>-\infty.
  \end{equation}
  This is the Lagrangian for (the minimizing version of) 
  Problem~\eqref{pr:NoiseRecoveryDual} where $y$ is replaced by
  $y+w$.  Our analysis above shows that \eqref{eq:ShapTest} holds for $w$
  whenever there are $\nu,\lambda$ satisfying
  \begin{equation}
    \sum_{i=1}^n ((K*(\nu-\lambda))(s_i)-(y_i+w_i))^2 \leq \bxi^2.\label{eq:SuffBound}
  \end{equation}
  By assumption, our noisy measurements satisfy
  \begin{equation}
    \sum_{i=1}^n ((K*\mu)(s_i)-y_i)^2 =:\beta^2 < \bxi^2.
  \end{equation}
  Letting $\cN=\{w:\|w\|_2< \bxi-\beta\}$ and applying the triangle
  inequality to \eqref{eq:SuffBound} (after taking square roots) completes the proof.
\end{proof}

\section{Proof of \Cref{lem:NoiseAtomic}}
\label{sec:NoiseAtomicProof}
\NoiseAtomic*
\begin{proof}
  By \Cref{lem:NoiseRecoveryDual}, there is a minimizer $\mu^*$ for
  Problem~\eqref{pr:NoiseRecovery} and a maximizer $c^*$ for 
  Problem~\eqref{pr:NoiseRecoveryDual}.  Then the
  Lagrangian \eqref{eq:NoiseDualLagrange} takes the form
  \begin{equation}
    \cL(c^*,\mu_+^*,\mu_-^*) = -(c^*)^Ty + \bxi\normTwo{c^*} +
    \int C^*(t)-1\,d\mu_+^*(t) +\int -C^*(t)-1\,d\mu_-^*(t),
    \label{eq:OptimalLag}
  \end{equation}
  where $C^*(t):=\sum_{i=1}^m c_i^*K(s_i-t)$.  Define $A\subseteq\RR$ by
  \begin{equation}
    A = \{ t\in\RR: |C^*(t)|=1\}.
  \end{equation}
  We first show that $A\cap I$ is finite.  If not, $A$ has a limit point.
  As $C^*$ is analytic, this implies $C^*$ is constant giving $A=\RR$. 
  But $\lim_{t\to-\infty}C^*(t)=0$ by assumption, a contradiction.  This
  proves $A\cap I$ is finite.

  Applying complementary slackness, both integrals in
  equation~\eqref{eq:OptimalLag} are zero.  For $t\in I\setminus A$,
  we have
  $|C^*(t)|-1<0$ by the feasibility of $c^*$ and the definition of
  $A$. Thus $\mu_+^*(I\setminus A)=\mu_-^*(I\setminus A)=0$.
  This shows any primal optimal solution $\mu^*$ is atomic,
  as required.
\end{proof}

\section{Proofs of Auxiliary Results in \Cref{sec:RobustRecoveryProof}}

\subsection{Proof of \Cref{lem:RobustCombinationOne}}
\label{sec:RobustCombinationOneProof}
\RobustCombinationOne*
\begin{proof}
We consider the dual combination $Q$ defined by
equation \eqref{eq:DualCombination} with coefficients adjusted to satisfy the
system of equations~\eqref{eq:Interpolation} for $\rho :=\sign(a)$.
By \Cref{lem:CoeffBound,lem:qbound}, $Q$ satisfies properties 1 and
2. To establish property 3, without loss of generality we set $t_j:=0$
and $a_j:=1$. The application of \Cref{lem:Regions} in the proof of
\Cref{lem:qbound} shows that for each
$(\Delta(T),\gamma(S,T))$ pair in \Cref{fig:ExactRecoveryRegion}
there is a $u_1>0$ such that $Q^{(2)}(t)<0$ if
$t\in [-u_1,u_1]$. By Taylor's theorem
  \begin{align}
    Q(t) &\leq Q(0) + tQ'(0) + \frac{t^2}{2}\sup_{u\in[-u_1,u_1]}
    Q^{(2)}(u)\\ & = 1 + \frac{ct^2}{2},
  \end{align}
  for $t\in [-u_1,u_1]$ where $c:=\sup_{u\in[-u_1,u_1]} Q^{(2)}(u)$
  satisfies $c<0$. This shows that property 3 holds.
  Similarly, the application of \Cref{lem:Regions} in the 
  proof of \Cref{lem:qbound} also shows that for each
  $(\Delta(T),\gamma(S,T))$ pair in \Cref{fig:ExactRecoveryRegion} we have
  $|Q(t)|\leq c' < 1$ if $t\in[u_2,h]$ for a certain constant $c'$.
  Letting $C_2 = \min(1-c',-c\eta^2/2)$ proves property 4.  

  For the Gaussian, by \Cref{lem:BumpExists,lem:DenomFormula}, 
  \begin{align}
    |b_{i,1}| &=
    \frac{|\tilde{s}_{i,2}|\exp\brac{-\frac{\tilde{s}_{i,2}^2}{2}}}{|\tilde{s}_{i,2}-\tilde{s}_{i,1}|\exp\brac{-\frac{\tilde{s}_{i,1}^2+\tilde{s}_{i,2}^2}{2}}}
    \leq \frac{\gamma(S,T)e^{\gamma(S,T)^2/2}}{\kappa(S)},\\
    |w_{i,1}| &=
    \frac{\exp\brac{-\frac{\tilde{s}_{i,2}^2}{2}}}{|\tilde{s}_{i,2}-\tilde{s}_{i,1}|\exp\brac{-\frac{\tilde{s}_{i,1}^2+\tilde{s}_{i,2}^2}{2}}}
    \leq \frac{e^{\gamma(S,T)^2/2}}{\kappa(S)},
  \end{align}
  where we simplified notation by using $t_i=0$. Similarly, for the Ricker, 
  \begin{align}
    |b_{i,1}| &=
    \frac{|\tilde{s}_{i,2}(\tilde{s}_{i,2}^2-3)|\exp\brac{-\frac{\tilde{s}_{i,2}^2}{2}}}{|\tilde{s}_{i,2}-\tilde{s}_{i,1}||3-(\tilde{s}_{i,2}-\tilde{s}_{i,1})^2+\tilde{s}_{i,2}^2\tilde{s}_{i,1}^2|
      \exp\brac{-\frac{\tilde{s}_{i,1}^2+\tilde{s}_{i,2}^2}{2}}}
    \leq \frac{D\gamma(S,T)e^{\gamma(S,T)^2/2}}{\kappa(S)},\\
    |w_{i,1}| &=
    \frac{|\tilde{s}_{i,2}^2-1|\exp\brac{-\frac{\tilde{s}_{i,2}^2}{2}}}{|\tilde{s}_{i,2}-\tilde{s}_{i,1}||3-(\tilde{s}_{i,2}-\tilde{s}_{i,1})^2+\tilde{s}_{i,2}^2\tilde{s}_{i,1}^2|
      \exp\brac{-\frac{\tilde{s}_{i,1}^2+\tilde{s}_{i,2}^2}{2}}}
    \leq \frac{D'e^{\gamma(S,T)^2/2}}{\kappa(S)},
  \end{align}
  where $D,D'>0$ come from the bounds on $\gamma(S,T)$ and \Cref{lem:RickerF}.
  The same bounds hold for $|b_{i,2}|$ and $|w_{i,2}|$.
  Combining the two bounds we obtain
  \begin{equation}
    \normInf{q} \leq
    \frac{\gamma(S,T)\normInf{\alpha}+\normInf{\beta}}
         {\kappa(S)}e^{\gamma(S,T)^2/2}\max(D,D'), 
  \end{equation}
  which establishes propety 5. Note that bounds on 
  $\normInf{\alpha}$ and $\normInf{\beta}$ computed from $\Delta(T)$,
  $\gamma(S,T)$, and $\kappa(S)$ are given in 
  \Cref{sec:CoeffBoundProof} by applying \Cref{lem:Schur}.
\end{proof}

\subsection{Proof of \Cref{lem:GlobalNoiseBounds}}
\label{sec:GlobalNoiseBoundsProof}
\GlobalNoiseBounds*
\begin{proof}
  We have 
   \begin{equation}
    \normTV{\hat{\mu}}-\int Q(t)\,d\hat{\mu}(t)
    \leq \normTV{\mu}-\int Q(t)\,d\hat{\mu}(t) = \int Q(t)\,d(\mu-\hat{\mu})(t),
  \end{equation}
  since $\mu$ is feasible for Problem~\eqref{pr:NoiseRecovery}. 
  To prove \eqref{eq:NoiseQMuDiff}, we bound the right hand side,
  \begin{align}
    \int Q(t)\,d(\mu-\hat{\mu})(t)
    &= \sum_{s_i\in S} q_i \int
    K(s_i-t)\,d(\mu-\hat{\mu})(t)\\
    &= \sum_{s_i\in S} q_i (K*(\mu-\hat{\mu}))(s_i)\\
    &\leq \normTwo{q} \sqrt{ \sum_{s_i\in S} (K*(\mu-\hat{\mu}))(s_i)^2}\quad \text{(Cauchy-Schwarz)}\\
    &\leq 2\bxi\normTwo{q}\label{eq:TwoEps}\\
    &\leq 2\bxi\sqrt{|T|}\normInf{q},
  \end{align}
where~\cref{eq:TwoEps} follows from the assumption that $\hat{\mu},\mu$ are
  both feasible for Problem~\eqref{pr:NoiseRecovery} and the triangle inequality.  
  
  To prove \eqref{eq:NoiseQBound}, we express the quantity of interest in terms of the coefficients of $\hat{\mu}$ at the different locations 
  \begin{align}
    \int Q(t)\,d\hat{\mu}(t)
    & = \sum_{\hat{t}_k\in\wh{T}} Q(\hat{t}_k)\hat{a}_k\\
    &\leq \sum_{\hat{t}_k\in\wh{T}:d(\hat{t}_k,T)\leq \eta}
    |Q(\hat{t}_k)||\hat{a}_k|
    +\sum_{\hat{t}_k\in\wh{T}:d(\hat{t}_k,T)> \eta}
    |Q(\hat{t}_k)||\hat{a}_k|.
  \end{align}
  Applying properties 3 and 4 in \Cref{lem:RobustCombinationOne} and
  noting that $\eta < \Delta(T)/2$ completes the proof.
\end{proof}

\subsection{Proof of \Cref{lem:RobustCombinationTwo}}
\label{sec:RobustCombinationTwoProof}
\RobustCombinationTwo*
The construction is very similar to the one presented in \Cref{sec:ExactRecoveryProof}. We define $Q_j$ as
\begin{equation}
Q_j(t) := \sum_{\tilde{s}_k \in \wt{S}} \qj_k K( \tilde{s}_k - t),
\end{equation}
where the coefficients $\qj \in \RR$ are chosen so that
\begin{equation}
  \label{eq:QjInterpolation}
  \begin{aligned}
    Q_j(t_j) &= 1, \\
    Q_j(t_l) &= 0\quad\text{for $t_l\in T\setminus\{t_j\}$,}\\
    Q_j'(t_i) &= 0\quad\text{for $t_i\in T$.}
  \end{aligned}
\end{equation}
By \Cref{lem:Schur} the system is invertible and if we reparametrize $Q$ in terms of the bumps and waves defined in \Cref{sec:bumps_waves},
\begin{align}
Q_j(t) = \sum_{t_i\in T}\alpha(j)_i B_{t_i}(t,\tilde{s}_{i,1},\tilde{s}_{i,2})+\beta(j)_i W_{t_i}(t,\tilde{s}_{i,1},\tilde{s}_{i,2}),
\end{align}
the bounds in \cref{eq:SchurCoeff1,eq:SchurCoeff2,eq:SchurLB} hold for $\normInf{\alpha(j)}$, $\normInf{\alpha(j)-\rho(j)}$
and $\normInf{\beta(j)}$, where $\alpha(j),\beta(j),\rho(j) \in \RR^{\abs{T}}$ and the entries of $\rho(j)$ are all zero except for $\rho(j)_j=\sign{a_j}$. To control the magnitude of $Q(j)$, we decompose it as $Q_j(t)=\sB_j(t)+\sW_j(t)$, where
\begin{equation}
\sB_j(t):= \sum_{k=1}^n \alpha(j)_k B_k(t) \quad\text{and}\quad \sW_j(t):= \sum_{k=1}^n \beta_k W_k(t).
\end{equation}
These quantities can be bounded by a modified version of \Cref{lem:GenericQBounds}, which has the same proof.
\begin{lemma}
  \label{lem:GenericQjBounds}
  Assume the
  conditions of \Cref{lem:Schur,lem:TailDecay} hold, and that $0\in T$
  has corresponding $\rho$-value $r\in\{0,1\}$.  Let $h$ denote half the distance from
  the origin to the element of $T$ with smallest positive location, or $\infty$
  if no such element exists. Then, for $0<t\leq h$ (or $t>0$ if
  $h=\infty$) and $\epsilon:=10^{-12}/\kappa(S)$,
  \begin{align}
    |\sB_j^{(p)}(t)| & \leq
    \|\alpha(j)\|_\infty\left(\smo{|B^{(p)}|}(v)+2\epsilon+ \sum_{j=1}^5
    \smo{|B^{(p)}|}(v+j\Delta)+\smo{|B^{(p)}|}(j\Delta-v)\right)
    \notag \\
    |\sW_j^{(p)}(t)| & \leq
    \|\beta(j)\|_\infty\left(\smo{|W^{(p)}|}(v)+2\epsilon+ \sum_{j=1}^5
    \smo{|W^{(p)}|}(v+j\Delta)+\smo{|W^{(p)}|}(j\Delta-v)\right),
    \label{eq:GQjBoundAbsBW}
  \end{align}
  for $p=0,1,2$ and
  $v=\min(t,\Delta(T)/2)$.  Under the same conditions, if $r=0$ we
  have
  \begin{equation*}
    |\sB_j^{(p)}(t)| \leq
    \|\alpha(j)-\rho(j)\|_\infty\smo{|B^{(p)}|}(v)+\|\alpha(j)\|_\infty\left(2\epsilon+ \sum_{j=1}^5
    \smo{|B^{(p)}|}(v+j\Delta)+\smo{|B^{(p)}|}(j\Delta-v)\right).
  \end{equation*}
  For $q=1,2$ 
  \begin{align}
    \sB_j^{(q)}(t) & \leq
    \alpha_{\LB}(j)\sym{B^{(q)}}(t)+\|\alpha(j)\|_\infty\left(
    2\epsilon+\sum_{j=1}^5 \smo{|B^{(q)}|}(v+j\Delta)
    +\smo{|B^{(q)}|}(j\Delta-v) \right)\label{eq:GQjBoundB},
  \end{align}
  as long as $r=1$, and $\sym{B^{(q)}}(t)\leq 0$, where 
  $\alpha_{\LB}(j):=1-\|\alpha(j)-\rho(j)\|_\infty\geq0$. 
\end{lemma}
To complete the proof, we give a method for testing whether properties
2-5 are satisfied for fixed values of  
$\Delta(T)$, $\gamma(S,T)$, $\kappa(S)$ 
and the parameters $N_1$ and $N_2$ in \Cref{sec:bwbounds}. Without
loss of generality we set $t_j:=0$ and restrict the analysis to
$\brac{0,h}$ for the $h$ defined in \Cref{lem:GenericQjBounds}.  
\begin{enumerate}
\item Case $r=1$:  
  Test whether the conditions of
  \Cref{lem:Regions} hold using the same process as performed for 
  exact recovery at the end of \Cref{sec:qboundproof} with the bounds
  from \Cref{lem:GenericQjBounds}.
  Also verify that $|Q^{(2)}(t)|$ is bounded above on $[0,\eta]$.
\item Case $r=0$: 
  Verify that $|Q(t)|<1$ for $t\in(0,h]$ by using the bounds on
    $|\sB_j(t)|$ and $|\sW_j(t)|$ from \Cref{lem:GenericQjBounds} 
    directly.  Also verify that
    $|Q^{(2)}(t)|$ is bounded on $[0,\eta]$.
\end{enumerate}
Case $r=1$ above must always succeed for the region stated in
\Cref{thm:RobustRecovery}, as it is the exact same calculation as
performed in \Cref{sec:qboundproof}, and our bound on $|Q^{(2)}(t)|$
is always finite.
In \Cref{fig:C2PBound}, we compute bounds on $\sup_{t\in(0,h]}|Q(t)|$ for case
$r=0$ above.  The parameter values we use in the 
computation are the same as used to compute
\Cref{fig:MatrixBounds} at the end of \Cref{sec:CoeffBoundProof}.
This shows that the steps above succeed for the entire region stated
in \Cref{thm:RobustRecovery}.  
\begin{figure}[t]
\centering
\includegraphics{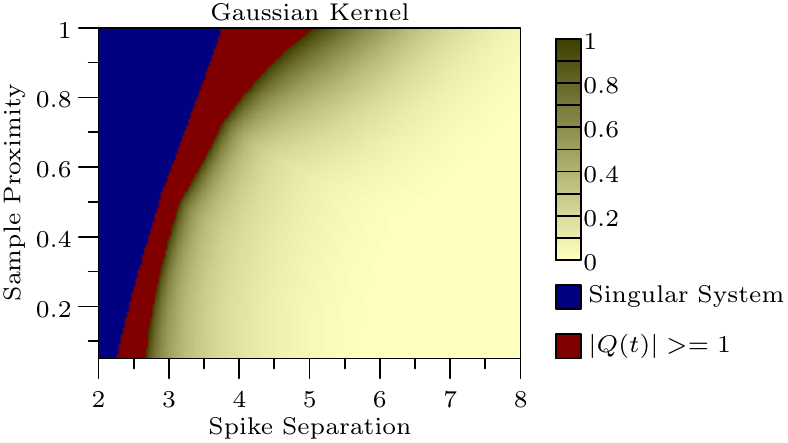}
\includegraphics{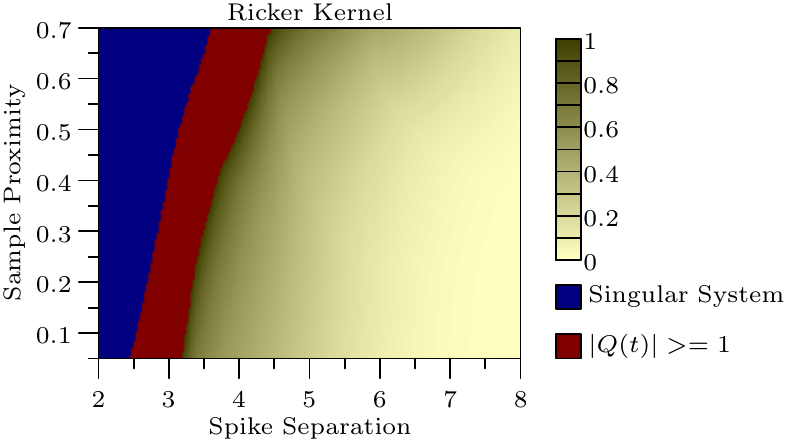}
\caption{Bounds on $\sup_{t\in(0,h]}|Q(t)|$ when $r=0$
for different values of the minimum separation and sample proximity.}
\label{fig:C2PBound}
\end{figure}

When cases $r=0$ and $r=1$ above are satisfied, they yield property 2.  
Property 4 follows from
Taylor's theorem since $|Q^{(2)}(t)|$ is bounded on $[0,\eta]$ in case
$r=0$ above.  The same type of bound was used in the proof of
\Cref{lem:RobustCombinationOne} property 3.
Property 5 follows by \Cref{lem:Regions} for the case $r=1$, and from
\Cref{fig:C2PBound} for the case $r=0$.
To establish property 3, we apply Taylor's theorem to obtain 
\begin{align}
  Q_j(t) &\geq Q_j(0) + tQ_j^{(1)}(0) +
  \frac{t^2}{2}\inf_{u\in[0,\eta]}Q_j^{(2)}(u)\\
  & \geq 1 - ct^2/2,
\end{align}
where $c$ is the upper bound on $|Q^{(2)}(t)|$ computed in case $r=1$ above.
By taking maxes, we can uses the same coefficient $C_1'$ in both properties 3 and~4.

The proof of the bound on $\normInf{\qj}$ is also identical to that 
given in \Cref{lem:RobustCombinationOne}.  This completes the proof of 
\Cref{lem:RobustCombinationTwo}.
\subsection{Proof of \Cref{lem:GlobalNoiseBoundsTwo}}
\label{sec:GlobalNoiseBoundsTwoProof}
\GlobalNoiseBoundsTwo*
\begin{proof}
  To prove \eqref{eq:NoiseQjMuDiff} we follow the same argument as the proof of \eqref{eq:NoiseQMuDiff},
  \begin{align}
    \left|\int Q_j(t)\,d(\mu-\hat{\mu})(t)\right|
    &= \left|\sum_{s_i\in S} \qj_i \int
    K(s_i-t)\,d(\mu-\hat{\mu})(t)\right|\\
    &= \left|\sum_{s_i\in S} \qj_i (K*(\mu-\hat{\mu}))(s_i)\right|\\
    &\leq 2\bxi\normTwo{\qj} &&\text{(Cauchy-Schwarz)}\\
    &\leq 2\normInf{\qj}\bxi\sqrt{|T|}.
  \end{align}
 To prove \eqref{eq:NoiseQjSum}, note that
  \begin{align}
    \alignmulticolumn{\left|\int Q_j(t)\,d\hat{\mu}(t)-\sum_{\hat{t}_k\in\wh{T}:|\hat{t}_k-t_j|\leq
      \eta}\hat{a}_k\right|}\\
    & = 
  \left|\sum_{\hat{t}_k\in\wh{T}:|\hat{t}_k-t_j|>\eta}\hat{a}_kQ_j(\hat{t}_k)
    + \sum_{\hat{t}_k\in\wh{T}:|\hat{t}_k-t_j|\leq
      \eta}\hat{a}_k(Q_j(\hat{t}_k)-1)\right|\\ 
    & = 
  \left|\sum_{\hat{t}_k\in\wh{T}:d(\hat{t}_k,T)>\eta}\hat{a}_kQ_j(\hat{t}_k)
  + \sum_{\hat{t}_k\in\wh{T}:d(\hat{t}_k,T\setminus\{t_j\})\leq \eta}\hat{a}_kQ_j(\hat{t}_k)
    + \sum_{\hat{t}_k\in\wh{T}:|\hat{t}_k-t_j|\leq
      \eta}\hat{a}_k(Q_j(\hat{t}_k)-1)\right|\\ 
  &\leq
  \sum_{\hat{t}_k\in\wh{T}:d(\hat{t}_k,T)>\eta}|\hat{a}_k|C_2'
  + \sum_{\hat{t}_k\in\wh{T}:d(\hat{t}_k,T\setminus\{t_j\})\leq
    \eta}|\hat{a}_k|C_1'd(\hat{t}_k,T)^2
  + \sum_{\hat{t}_k\in\wh{T}:|\hat{t}_k-t_j|\leq \eta}|\hat{a}_k|C_1'd(\hat{t}_k,T)^2,
\end{align}
  by \Cref{lem:RobustCombinationTwo}.  Applying
  \eqref{eq:RobustRecovery_2} and \eqref{eq:RobustRecovery_3}
  completes the proof.
\end{proof}

\section{Proof of \Cref{prop:SparseDualCert}}
\label{sec:SparseDualCertProof}
Our proof of \Cref{prop:SparseDualCert} is based on the proof of Lemma 2.3 in
\cite{fernandez2016demixing}. 
\SparseDualCert*
\begin{proof}
  Let $(\mu',w')$ denote a feasible solution to Problem~\eqref{pr:SparseNoiseRecovery}.
  Applying the Lebesgue decomposition to $\mu'$ we can write
  $\mu'=\mu'_T+\mu'_{T^c}$ where $\mu'_T$ is absolutely continuous
  with respect to $|\mu|$, and $\mu'_{T^c}$ is mutually orthogonal to
  $|\mu|$.  Analogously, we can write $w'=w'_{\Noi}+w'_{\Noi^c}$,
  where $(w'_\Noi)_i=0$ for $s_i\notin\Noi$ and
  $(w'_{\Noi^c})_i=0$ for $s_i\notin\Noi^c$.  

  We first prove that if
  $\mu'_{T^c}=0$ and $w'_{\Noi^c}=0$ then $\mu=\mu'$ and $w=w'$.
  Let $h := \mu'-\mu$ and $g:=w'-w$.  This allows us to write
  \begin{equation}
    h = \sum_{t_j\in T}b_j\delta_{t_j}
  \end{equation}
  where $b\in\RR^{|T|}$.
  Set $\rho = \sign(b)$ and $\rho'=\sign(g)$ where we choose an
  arbitrary value in $\{-1,1\}$ for $\sign(0)$.  By assumption there
  exists a corresponding $Q$ satisfying the hypotheses of
  \Cref{prop:SparseDualCert} with respect to $\rho,\rho'$.  
  As $(\mu,w)$ and $(\mu',w')$ are
  feasible, we have $(K*h)(s_i)+g_i=0$ for $i=1,\ldots,n$.  This
  gives
  \begin{align}
    0 & = \sum_{i=1}^n q_i((K*h)(s_i)+g_i) = \sum_{i=1}^n q_i\int
    K(s_i-t)\,dh(t) + q^Tg\\
    &= \int Q(t)\,dh(t)+q^Tg = \normTV{h}+\lambda\normOne{g},
  \end{align}
  implying $h=0$ and $g=0$ and proving that $\mu=\mu'$ and $w=w'$.

  Next assume that $\mu'_{T^c}\neq 0$ or
  $w'_{\Noi^c}\neq 0$.  To finish the proof, 
  we will show that $(\mu',w')$ is not a minimizer for
  Problem~\eqref{pr:SparseNoiseRecovery}.  Write $\mu$ as
  \begin{equation}
    \mu = \sum_{t_j\in T}a_j\delta_{t_j},
  \end{equation}
  where $a\in\RR^{|T|}$.  Set $\rho=\sign(a)$ and $\rho'=\sign(w)$, where
  again we choose an arbitrary value in $\{-1,1\}$ for $\sign(0)$.
  By assumption there
  exists a corresponding $Q$ satisfying the hypotheses of
  \Cref{prop:SparseDualCert} with respect to $\rho,\rho'$.
  Then we have
  \begin{align}
    \normTV{\mu'} + \lambda\normOne{w'}
    &> \int Q(t)\,d\mu'(t) + q^Tw' \label{eq:QZAssump}\\
    &= \sum_{i=1}^n q_i((K*\mu')(s_i)+w'_i)\\
    &= \sum_{i=1}^n q_i((K*\mu)(s_i)+w_i)\label{eq:QZFeas}\\
    &= \int Q(t)\,d\mu(t) + q^Tw\\
    &= \normTV{\mu} + \lambda\normOne{w},
  \end{align}
  where \eqref{eq:QZAssump} follows from $|Q(t)|<1$ for $t\in T^c$
  and $|q_i|<\lambda$ for $s_i\in\Noi^c$.  Equation~\eqref{eq:QZFeas}
  holds because $(\mu,w)$ and $(\mu',w')$ are both feasible by assumption.
\end{proof}

\section{Proof of \Cref{lem:SparseTrim}}
\label{sec:SparseTrimProof}
\SparseTrim*
\begin{proof}
  For any vector $u$ and any atomic measure $\nu$, we denote $u_S$ and
  $\nu_S$ the restriction of $u$ and $\nu$ to the subset of their
  support index by a set~$S$.  Let $(\hat{\mu},\hat{w})$ be any
  solution to Problem~\eqref{pr:SparseNoiseRecovery} applied to $y'$.
  The pair $(\hat{\mu}+\mu_{T\setminus
    T'},\hat{w}+w_{\wid\setminus\wid'})$ is feasible for 
  Problem~\eqref{pr:SparseNoiseRecovery} applied to $y$ since
  \begin{align}
    (K*\hat{\mu})(s_i) + (K*\mu_{T\setminus T'})(s_i) +
    \hat{w}_i+(w_{\wid\setminus\wid'})_i
    & = y'_i + (K*\mu_{T\setminus T'})(s_i) +
    (w_{\wid\setminus\wid'})_i\\
    & = (K*\mu')(s_i) + w'_i +(K*\mu_{T\setminus T'})(s_i) +
    (w_{\wid\setminus\wid'})_i\notag\\
    & = (K*\mu)(s_i) + w_i\\
    & = y_i,
  \end{align}
  for $i=1,\ldots,n$. By the triangle inequality and the assumption
  that $(\mu,w)$ is the unique solution to
  Problem~\eqref{pr:SparseNoiseRecovery} applied to $y$, this implies
  \begin{align}
    \normTV{\mu} + \lambda\normOne{w}
    & < \normTV{\hat{\mu}+\mu_{T\setminus T'}} +
    \lambda\normOne{\hat{w}+w_{\wid\setminus\wid'}}\\
    & \leq \normTV{\hat{\mu}}+\normTV{\mu_{T\setminus T'}} +
    \lambda\normOne{\hat{w}}+\lambda\normOne{w_{\wid\setminus\wid'}}
  \end{align}
  unless $\hat{\mu}+\mu_{T\setminus T'}=\mu$ and
  $\hat{w}+w_{\wid\setminus\wid'}=w$.  This implies
  \begin{align}
    \normTV{\mu'} + \lambda\normOne{w'}
    & = \normTV{\mu}-\normTV{\mu_{T\setminus T'}} +
    \lambda\normOne{\hat{w}}-\lambda\normOne{w_{\wid\setminus\wid'}}\\
    & < \normTV{\hat{\mu}}+\lambda\normOne{\hat{w}},
  \end{align}
  unless $\hat{\mu}=\mu'$ and $\hat{w}=w'$.  We conclude that
  $(\mu',w')$ must be the unique solution to 
  Problem~\eqref{pr:SparseNoiseRecovery} applied to $y'$.
\end{proof}

\section{Proofs of Auxiliary Results in \Cref{sec:SparseNoise}}
In this section we prove the lemmas that allow us to establish exact recovery in the presence of sparse noise in \Cref{sec:SparseNoise}. We begin by proving some bounds on the dampened noise function $D$ which will be needed in the following sections. For notational convenience we consider $\wid\in[\wid_1,\wid_2]$
where this interval is $[\botwid,\topwid]$ for the Gaussian, and
$[\botwidR,\topwidR]$ for the Ricker.

\subsection{Bounds on the Dampened Noise Function}
For convenience, define
\begin{equation}
  D(t) := D_0(t,1,1),
\end{equation}
and note that $D_0(t,\lambda,\rho'_i)=\lambda\rho'_iD(t)$.  
As in \Cref{sec:bwbounds} where we bounded the bumps and waves, we need 
monotonic bounds on $D$ and its derivatives, which we define below:
\begin{equation}
  \smo{|D^{(j)}|}(t) := \sup_{\substack{u\geq |t|\\\wid\in[\wid_1,\wid_2]}}|D^{(j)}|(u),
\end{equation}
for $j=0,1,2$.  By \Cref{lem:EvenKernel}, these functions are all
even.  In \Cref{fig:Vfunction} we show the derivatives of $D$, and their
monotonized versions. The following lemmas derive bounds on $D$ using the same techniques used to bound the bump
and wave functions in \Cref{sec:TailDecayProof}. 
\begin{figure}[t]
\centering
\includegraphics{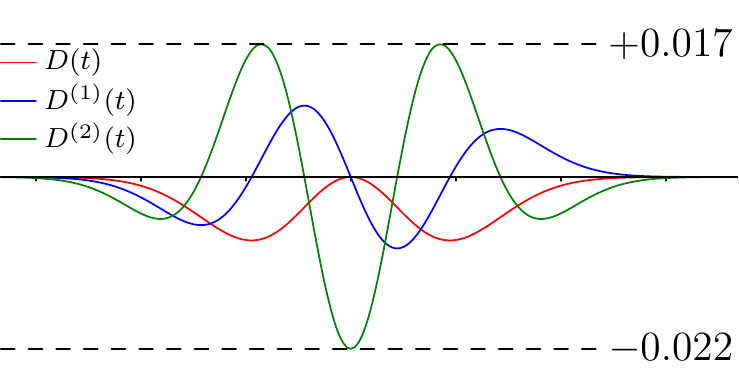} $ \qquad $ 
\includegraphics{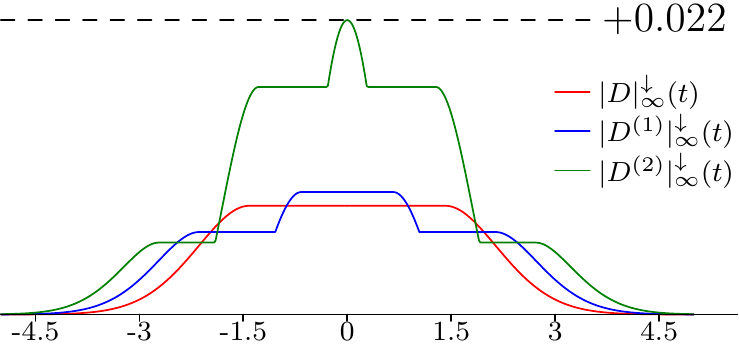}
\caption{The left image shows the dampened kernel $D$ for the Gaussian
  kernel, along with its derivatives, for $\wid=.15$. On the right, we can see the corresponding monotonized bounds.}
\label{fig:Vfunction}
\end{figure}

\begin{lemma}
  \label{lem:VTailDecay}
  For the Gaussian and Ricker kernels with $\wid\in[\wid_1,\wid_2]$ we have
  \begin{align}
  \sum_{j=6}^\infty \smo{|D^{(i)}|}(j\Delta) & \leq
  \frac{10^{-12}}{\wid_1},\label{eq:VTailSumBound}\\
  \smo{|D^{(i)}|}(t) & \leq \frac{10^{-12}}{\wid_1},\label{eq:VTailBound}
  \end{align}
  for $i=0,1,2$, $t\geq 10$, and $\Delta\geq 2$.
\end{lemma}
\begin{proof}
  By \Cref{lem:GaussDecayBound,lem:RickerDecayBound} we have
  \begin{equation}
    \left|\frac{\partial^i}{\partial t^i}B(t,s_1,s_2)\right| \leq
    \frac{12t^4}{\wid_1}\exp\brac{-\frac{t^2}{2}+t},
  \end{equation}
  where the bound holds for both kernels.  Note that the grid width
  $\tau$ can be used in place of $\kappa(S)$.
  Thus we have, for $t\geq 10$,
  \begin{align}
    \left|\KG(t)-\frac{\partial^i}{\partial t^i}B(t,s_1,s_2)\right| &\leq
    \frac{12t^4}{\wid_1}\exp\brac{-\frac{t^2}{2}+t} +
    \exp\brac{-\frac{t^2}{2}}\\
    & \leq \frac{13t^4}{\wid_1}\exp\brac{-\frac{t^2}{2}+t},
  \end{align}
  for the Gaussian and
  \begin{align}
    \left|\KR(t)-\frac{\partial^i}{\partial t^i}B(t,s_1,s_2)\right| &\leq
    \frac{12t^4}{\wid_1}\exp\brac{-\frac{t^2}{2}+t} +
    (t^2-1)\exp\brac{-\frac{t^2}{2}}\\
    & \leq \frac{12t^4}{\wid_1}\exp\brac{-\frac{t^2}{2}+t} 
    + t^4\exp\brac{-\frac{t^2}{2}}\\
    & \leq \frac{13t^4}{\wid_1}\exp\brac{-\frac{t^2}{2}+t},
  \end{align}
  for the Ricker.  Applying \Cref{lem:SumTailBound} gives
  \eqref{eq:VTailSumBound} and \eqref{eq:VTailBound}.
\end{proof}
\begin{lemma}
\label{lem:VPiecewise}
  Fix $\wid\in[\wid_1,\wid_2]$ and let $N_1\in\ZZ_{>0}$. Partition the
  interval $[0,10)$ into $N_1$ intervals of the form
  \begin{equation}
    \ml{U}_j  := \sqbr{ \frac{10(j-1)}{N_1},\frac{10j}{N_1}},
  \end{equation}
 and $j=1,\ldots,N_1$.  For the Gaussian and Ricker kernels, there
 exists a function $\wt{|D^{(i)}|}$
 such that for all $ t \in \ml{U}_j$ and $i=0,1,2$
 \begin{equation}
    \smo{|D^{(i)}|}(t) \leq \wt{|D^{(i)}|}(j).
 \end{equation}
\end{lemma}
\begin{proof}
  The construction 
  here is a slight modification to the argument given in \Cref{sec:PiecewiseProof}.
  Partition $[0,10)$ into $N_1$ segments of the form
    \begin{align}
      \ml{U}_j &:= \left[\frac{10 \brac{j-1}}{N_1},\frac{10
          j}{N_1}\right), \quad 1 \leq j \leq N_1,
    \end{align}
    and define
    \begin{equation}
      |\wt{D^{(i)}}(j)| \geq \sup_{\substack{t\in
          \ml{U}_j\\\wid\in[\wid_1,\wid_2]}} |D^{(i)}|(t)
      = \sup_{\substack{t\in
          \ml{U}_j\\\wid\in[\wid_1,\wid_2]}} |K^{(i)}(t)-B^{(i)}(t,-\wid,\wid)|,
    \end{equation}
    for $i=0,1,2$.  This can be computed by interval arithmetic (see
    \Cref{sec:IntervalArithmetic}).
   
\end{proof}
\begin{corollary}\label{cor:VMonotonize}
  Assuming the conditions and definitions in
  \Cref{lem:VTailDecay,lem:VPiecewise} we have, for $i=0,1,2$ and $t
  \in \ml{U}_j$
  \begin{equation}
    \smo{|D^{(i)}|}(t) \leq \max\left(\max_{k:j\leq k \leq
      N_1}\wt{|D^{(i)}|}(k),\epsilon\right),
  \end{equation}
  where $\epsilon := 10^{-12}/\wid_1$.
\end{corollary}
\subsection{Proof of \Cref{lem:SparseCoeffBound}: Invertibility of the
  Interpolation Equations}
\label{sec:SparseCoeffBoundProof}
In this section, we use a variant of \Cref{lem:Schur} to prove the
interpolation equations \eqref{eq:Interpolation_sparse} are
invertible.  We begin by reparametrizing into bumps and waves:
\begin{equation}
\label{eq:Reparam_sparse}
\begin{aligned}
\sum_{t_j\in T} \alpha_jB_j(t_i) + \beta_jW_j(t_i) & = \psi_i, \\
\sum_{t_j\in T} \alpha_jB_j^{(1)}(t_i) + \beta_jW_j^{(1)}(t_i) & =
\zeta_i, 
\end{aligned}
\end{equation}
for all $t_i\in T$ where
\begin{align}
  \psi_i &= \rho_i - \RC(t_i),\\
  \zeta_i &= -\RC^{(1)}(t_i).
\end{align}
The block-matrix representation of \eqref{eq:Reparam_sparse} is:
\begin{equation}\label{eq:SparseBlockEquation}
\MAT{\BB&\WW\\\DBB&\DWW }\MAT{\alpha\\\beta} = \MAT{\psi\\\zeta},
\end{equation}
where the different matrices are defined in~\eqref{eq:block_matrices}. 
\begin{lemma}
  \label{lem:SparseSchur}
 Suppose
  $\|\II-\DWW \|_\infty < 1$, and $\|\II-\CCC\|_\infty <1$ where
  \begin{equation*}
    \CCC = \BB-\WW(\DWW )^{-1}(\DBB)
  \end{equation*}
  is the Schur complement of $\DWW$, and $\II$ is the identity
  matrix.  Then
  \begin{equation}
    \MAT{\BB&\WW\\\DBB&\DWW }\MAT{\alpha\\\beta} =
    \MAT{\psi\\\zeta}\label{eq:SparseSchurEq}
  \end{equation}
  has a unique solution. Furthermore, we have
  \begin{align}
    \normInf{\alpha} & \leq \normInf{\CCC^{-1}}\brac{\normInf{\psi}
    +\normInf{\WW}\normInf{(\DWW)^{-1}}\normInf{\zeta}}, \label{eq:SparseSchurCoeff1}\\ 
    \normInf{\beta} & \leq \normInf{(\DWW)^{-1}}\brac{\normInf{\zeta}
    +\normInf{\DBB}\normInf{\alpha}},\label{eq:SparseSchurCoeff2}\\
    \normInf{\alpha-\psi}
    & \leq \normInf{\II-\CCC}\normInf{\alpha}
    +\normInf{\WW}\normInf{(\DWW)^{-1}}\normInf{\zeta},\label{eq:SparseSchurCoeff3}
  \end{align}
  where
  \begin{align}
    \|\II-\CCC\|_\infty & \leq \|\II-\BB\|_\infty +
    \|\WW\|_\infty\|(\DWW
    )^{-1}\|_\infty\|\DBB\|_\infty,\label{eq:SparseSchurIC}\\ \|\CCC^{-1}\|_\infty
    & \leq
    \frac{1}{1-\|\II-\CCC\|_\infty},\label{eq:SparseNeumann1}\\ \|(\DWW
    )^{-1}\|_\infty & \leq \frac{1}{1-\|\II-\DWW
      \|_\infty}.\label{eq:SparseNeumann2}
  \end{align}
\end{lemma}
\begin{proof}
  For any matrix $A\in\RR^{n\times n}$ such that $\normInf{ A} < 1$ the Neumann
  series $\sum_{j=0}^\infty A^j$ converges to $(\II-A)^{-1}$.  By the
  triangle inequality and the submultiplicativity of the $\infty$-norm, this gives
  \begin{align}
  \|(\II-A)^{-1}\|_\infty \leq \sum_{j=0}^\infty \|A\|_\infty^j =
  \frac{1}{1-\|A\|_\infty}.
  \end{align}
  Setting $A:=\II-\DWW$ proves $\DWW$ is invertible, $\cC$ exists, and
  gives inequality \eqref{eq:SparseNeumann2}. Setting
  $A:=\II-\cC$ proves that $\cC$ is invertible and gives
  inequality \eqref{eq:SparseNeumann1}.  If both $\DWW$ and its Schur complement
  $\CCC$ are invertible, then the matrix in \cref{eq:SparseSchurEq} is
  invertible and the system has a unique solution.
  Applying blockwise Gaussian elimination to the augmented
  system we have
  \begin{equation}
    \left[\begin{array}{cc|c}
        \BB&\WW&\psi\\\DBB&\DWW&\zeta\end{array}\right]
    \implies 
    \left[\begin{array}{cc|c}
        \BB-\WW(\DWW)^{-1}\DBB&0&\psi-\WW(\DWW)^{-1}\zeta\\
        \DBB&\DWW&\zeta\end{array}\right].
  \end{equation}
 As a result
  \begin{equation}
    \alpha = \cC^{-1}(\psi-\WW(\DWW)^{-1}\zeta),
  \end{equation}
  giving \eqref{eq:SparseSchurCoeff1}. Since $\DBB\alpha+\DWW\beta=\zeta$,
  \begin{equation}
    \beta = (\DWW)^{-1}(\zeta-\DBB\alpha),
  \end{equation}
  giving \eqref{eq:SparseSchurCoeff2}.
  Finally, note that
  \begin{equation}
    (\II-\cC)\alpha = \alpha - \cC\cC^{-1}(\psi-\WW(\DWW)^{-1}\zeta).
  \end{equation}
  Rearranging gives \eqref{eq:SparseSchurCoeff3}.
\end{proof}
Using $\gamma(S,T)=\kappa(S)=\wid$, 
the techniques of \Cref{sec:CoeffBoundProof} apply directly to bound
$\normInf{\BB}$, $\normInf{\DBB}$, $\normInf{\WW}$, and
$\normInf{\DWW}$.   This proves invertibility for $(\Delta(T),\wid)$
in the required region, as it is a subset of the
$(\Delta(T),\gamma(S,T))$ region used in \Cref{thm:Exact}.
To bound $\normInf{\psi}$ and $\normInf{\zeta}$ we
apply the triangle inequality to obtain
\begin{align}
  |\psi_j| &\leq 1 + \sum_{i\in\Noi}|D_{s_i}(t_j,\lambda,\rho'_i)|\\
  |\zeta_j| &\leq \sum_{i\in\Noi}\abs{\frac{\partial}{\partial t}
    D_{s_i}(t_j,\lambda,\rho'_i)},
\end{align}
for all $t_j\in T$.  By the separation requirement on the sparse noise
locations, and monotonicity we obtain the bounds
\begin{align}
  |\psi_j| &\leq 1 + 2\lambda\sum_{i=0}^\infty\smo{|D|}(i\Delta(T)),\\
  |\zeta_j| &\leq 2\lambda\sum_{i=0}^\infty\smo{|D^{(1)}|}(i\Delta(T)),
\end{align}
for all $t_j\in T$.  By \Cref{lem:VTailDecay},
\begin{align}
  \normInf{\psi} &\leq 1 + 2\lambda\brac{\sum_{i=0}^5\smo{|D|}(i\Delta(T))+\epsilon} ,\\
  \normInf{\zeta} &\leq 2\lambda\brac{\sum_{i=0}^5\smo{|D^{(1)}|}(i\Delta(T))+\epsilon},\\
\end{align}
where $\epsilon=10^{-12}/\wid_1$.  For the Gaussian kernel with
$\Delta(T)\geq \DeltaG$ and $\wid\in[\wid_1,\wid_2]$  
\begin{equation}
  \normInf{\psi} \leq \psiG \quad\text{and}\quad\normInf{\zeta}\leq \zetaG,
\end{equation}
whereas for the Ricker wavelet with $\Delta(T)\geq \DeltaR$
\begin{equation}
  \normInf{\psi} \leq \psiR \quad\text{and}\quad\normInf{\zeta}\leq \zetaR,
\end{equation}
by \Cref{cor:VMonotonize}.  Applying
\Cref{lem:SparseSchur}, for the Gaussian kernel we have
\begin{equation}
  \normInf{\alpha} \leq \alphaG \quad\text{and}\quad\normInf{\beta}\leq \betaG,
\end{equation} 
and for the Ricker kernel
\begin{equation}
  \normInf{\alpha} \leq \alphaR \quad\text{and}\quad\normInf{\beta}\leq \betaR.
\end{equation}
\subsection{Proof of \Cref{lem:SparseQBound}: 
  Satisfying Condition \eqref{eq:SparseDualQBound}}
\label{sec:SparseQBoundProof}
In this section, we prove that the dual combination $Q$ satisfies the
condition $|Q(t)|<1$ for $t\in T^c$.
Without loss of generality, we restrict our attention to the
region between a spike at the origin with sign $+1$ and an adjacent spike.  

We begin by decomposing $Q$ as the sum of three terms that account for
the bumps, waves, and noise separately: $Q(t)=\sB(t)+\sW(t)+\sD(t)$,
where
\begin{align}
  \sB(t) &:= \sum_{j=1}^m \alpha_jB_j(t),\\
  \sW(t) &:= \sum_{j=1}^m \beta_jW_j(t),\\
  \sD(t) &:= \sum_{s_i\in\Noi} D_{s_i}(t,\lambda,\rho'_i).
\end{align}
\Cref{lem:GenericQBounds} provides bounds for $|\sB^{(p)}(t)|$,
$|\sW^{(p)}(t)|$, and $\sB^{(q)}(t)$ for $p=0,1,2$ and $q=1,2$.
For the noise component, we have the following bounds which follow
from monotonicity, \Cref{lem:VTailDecay}, and the noise separation condition:
\begin{lemma}
  \label{lem:VBounds}
  Fix our kernel to be the Gaussian kernel or the Ricker wavelet.
  Assume the conditions of \Cref{lem:VTailDecay} hold, and
  that there is a spike at the origin with sign $+1$.  Let $h$ denote
  half the distance from the origin to the spike with smallest
  positive location, or $\infty$ if no such spike exists. Then, for
  $0<t\leq h$ (or $t>0$ if $h=\infty$) and
  $\epsilon:=10^{-12}/\wid_1$,
  \begin{equation}
    |\sD^{(p)}(t)|  \leq
    2\lambda\brac{\sum_{i=0}^5\smo{|D^{(p)}|}(i\Delta(T))+\epsilon},
  \end{equation}
  for $p=0,1,2$.
\end{lemma}
\begin{proof}
  By the triangle inequality and monotonicity 
  \begin{align}
    |\sD{(p)}(t)|
    &\leq \sum_{i\in\Noi} |D^{(p)}_{s_i}(t,\lambda,\rho'_i)|\\
    &= \lambda\sum_{i\in\Noi} |D^{(p)}(t-s_i)|\\
    &= \lambda\sum_{\substack{i\in\Noi\\t\geq s_i}}
    |D^{(p)}(t-s_i)|+
    \lambda\sum_{\substack{i\in\Noi\\t< s_i}} |D^{(p)}(s_i-t)|\\
    &\leq \lambda\sum_{i=0}^\infty \smo{|D^{(p)}|}(i\Delta(T))+
    \lambda\sum_{i=0}^\infty \smo{|D^{(p)}|}(i\Delta(T))\label{eq:sVMonotoneBound}\\
    &\leq 2\lambda\brac{\sum_{i=0}^5
      \smo{|D^{(p)}|}(i\Delta(T))+\epsilon}\label{eq:sVTailBound},
  \end{align}
  where \eqref{eq:sVMonotoneBound} holds due to monotonicity, and
  \eqref{eq:sVTailBound} is due to \Cref{lem:VTailDecay}.
\end{proof}

We use the procedure at the end of
\Cref{sec:qboundproof} with $\gamma(S,T)=\wid_2$ to prove that
$|Q(t)|<1$ for $t\in T^c$.  \Cref{fig:SparseQPlots} shows upper bounds on
$|Q|$, $Q^{(1)}$ and $Q^{(2)}$ for $\wid\in[\wid_1,\wid_2]$ with
$\Delta(T)\geq\DeltaG$ for the Gaussian kernel and $\Delta(T)\geq\DeltaR$ for
the Ricker wavelet. As a result, the conditions of
\Cref{lem:Regions} are satisfied and the proof is complete.
\begin{figure}[t]
\centering 
\includegraphics{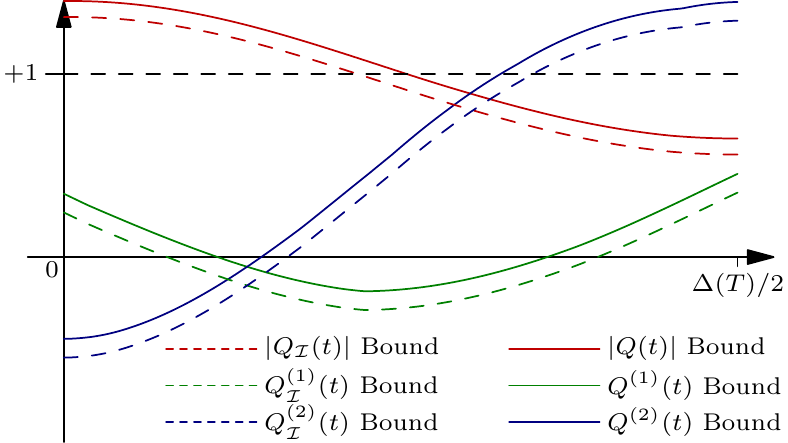}
\includegraphics{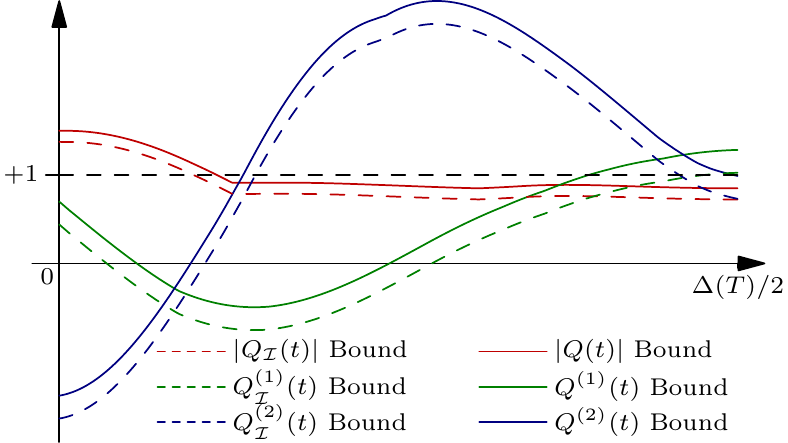}
  \caption{Upper bounds on $Q$, $Q^{(1)}$ and $Q^{(2)}$ for $\wid\in[\wid_1,\wid_2]$ with
$\Delta(T)\geq\DeltaG$ for the Gaussian kernel and $\Delta(T)\geq\DeltaR$ for
the Ricker wavelet.  The difference between the dashed and undashed
curves shows the contribution of the dampened noise.}
  \label{fig:SparseQPlots}
\end{figure}

\subsection{Proof of \Cref{lem:SparseqBound}: 
  Satisfying Condition \eqref{eq:SparseqBound}}
\label{sec:SparseqBoundProof}
Here we prove that $|q_i|<\lambda$ for $s_i\in\Noi^c$.
We begin by establishing that $q_i$ is bounded when $s_i\in\cC$.
\begin{lemma}
  \label{lem:VqBound}
  For the Gaussian and Ricker kernels, if $\wid \leq 0.6$ 
  then $|b_1|<1$ and $|b_2|<1$ where
  \begin{equation}
    B(t,-\wid,\wid) = b_1K(-\wid-t)+b_2K(\wid-t),
  \end{equation}
  is the bump function centered at the origin with samples at $-\wid,\wid$.
\end{lemma}
\begin{proof}
  For the Gaussian kernel we have, by \Cref{lem:BumpExists,lem:DenomFormula},
  \begin{equation}
    |b_1| = \frac{|(\KG)^{(1)}(\wid)|}{2\wid\exp(-\wid^2)}
    = \frac{\exp(\wid^2/2)}{2},
  \end{equation}
  with the same formula for $|b_2|$.  As $e^{0.6^2/2}/2\approx0.5986$
  the result follows for the Gaussian.  
  For the Ricker, 
  by \Cref{lem:BumpExists,lem:DenomFormula}
  \begin{equation}
    |b_1| = \frac{|(\KR)^{(1)}(\wid)|}{2\wid|3-4\wid^2+\wid^4|\exp(-\wid^2)}
    = \frac{\wid|3-\wid^2|\exp(\wid^2/2)}{2|3-4\wid^2+\wid^4|},
  \end{equation}
  with the same formula for $|b_2|$.  For $\wid\leq0.6$ 
  \begin{equation}
    \frac{\wid(3-\wid^2)\exp(\wid^2/2)}{2(3-4\wid^2+\wid^4)}
    \leq \frac{0.6(3)\exp(0.6^2/2)}{2(3-4(0.6)^2)} \approx 0.6907,
  \end{equation}
  completing the proof for the Ricker wavelet.
\end{proof}

If the conditions of \Cref{lem:VqBound} hold then by construction
$|q_{i}|<\lambda$ for all $i\in\cC$. To
complete the proof of \Cref{lem:SparseqBound}, we must show that these
bounds also hold for the pairs of samples closest to each spike.

For the remainder of the section, we assume there is a spike $t_j$ at the
origin with sign $+1$, and $s_i,s_{i+1}$ are the two samples closest to the origin with
$s_{i+1}=s_{i}+\wid$.  We prove that $|q_i|<\lambda$ by showing
that $|q_i|\geq\lambda$ implies $|Q(s_i)|>1$, 
which contradicts \Cref{lem:SparseQBound}. 
By combining \Cref{lem:SparseSchur} and \Cref{lem:BumpExists} we
obtain the following result.
\begin{lemma}
  \label{lem:SparseqPos}
  Suppose $s_i,s_{i+1}$ are the two closest samples to a spike with
  location $t_j=0$.  Assume the spike at
  the origin has sign $+1$, $\alpha_j\geq0$, and that
  $s_i\leq 0\leq s_{i+1}$.
  Then, under the assumptions of \Cref{lem:SparseqBound}, we have
  \begin{equation}
    \min(q_i,q_{i+1})\geq -\frac{\normInf{\beta}e^{\wid^2/2}}{\wid} > -\lambda,
  \end{equation}
  for the Gaussian kernel and 
  \begin{equation}
    \min(q_i,q_{i+1})\geq -\frac{\normInf{\beta}e^{\wid^2/2}}{\wid(3-\wid^2)} > -\lambda,
  \end{equation}
  for the Ricker wavelet.
\end{lemma}
\begin{proof}
  Applying \Cref{lem:BumpExists,lem:DenomFormula} we have
  \begin{equation}
    b_{i,1}
    =\frac{s_{i+1}\KG(s_{i+1})}{(s_{i+1}-s_i)\exp\brac{-\frac{s_i^2+s_{i+1}^2}{2}}},
    \quad
    b_{i,2} =\frac{-s_i\KG(s_i)}{(s_{i+1}-s_i)\exp\brac{-\frac{s_i^2+s_{i+1}^2}{2}}},
  \end{equation}
  for the Gaussian kernel and
  \begin{align}
    b_{i,1}
    &=\frac{s_{i+1}(3-s_{i+1}^2)\KG(s_{i+1})}{(s_{i+1}-s_i)(3-(s_i-s_{i+1})^2+s_i^2s_{i+1}^2)
      \exp\brac{-\frac{s_i^2+s_{i+1}^2}{2}}},\\
    b_{i,2}
    &=\frac{-s_i(3-s_i^2)\KG(s_i)}{(s_{i+1}-s_i)(3-(s_i-s_{i+1})^2+s_i^2s_{i+1}^2)
      \exp\brac{-\frac{s_i^2+s_{i+1}^2}{2}}},
  \end{align}
  for the Ricker wavelet.
  Since $s_1\leq 0\leq s_2$ we see that both bump coefficients are
  non-negative for both kernels.  By a similar calculation we have
  \begin{equation}
    w_{i,1}
    =\frac{-\KG(s_{i+1})}{(s_{i+1}-s_i)\exp\brac{-\frac{s_i^2+s_{i+1}^2}{2}}},
    \quad
    w_{i,2} =\frac{\KG(s_i)}{(s_{i+1}-s_i)\exp\brac{-\frac{s_i^2+s_{i+1}^2}{2}}},
  \end{equation}
  for the Gaussian kernel and
  \begin{align}
    w_{i,1}
    &=\frac{(s_{i+1}^2-1)\KG(s_{i+1})}{(s_{i+1}-s_i)(3-(s_i-s_{i+1})^2+s_i^2s_{i+1}^2)\exp\brac{-\frac{s_i^2+s_{i+1}^2}{2}}},\\
    w_{i,2} &
    =\frac{(1-s_i^2)\KG(s_i)}{(s_{i+1}-s_i)(3-(s_i-s_{i+1})^2+s_i^2s_{i+1}^2)\exp\brac{-\frac{s_i^2+s_{i+1}^2}{2}}},\\
  \end{align}
  for the Ricker wavelet.
  Thus $w_{i,1}\leq 0$ and $w_{i,2}\geq0$ for both kernels.  As
  $\alpha_j\geq0$, we can
  lower bound $\min(q_i,q_{i+1})$ by considering the magnitudes of
  $w_{i,1},w_{i,2}$.  This shows
  \begin{equation}
    \min(q_i,q_{i+1}) \geq -\frac{\normInf{\beta}e^{\wid^2/2}}{\wid}
  \end{equation}
  for the Gaussian kernel.  As $f(x)=e^{x^2/2}/x$ has $f'(x)=e^{x^2/2}(x^2-1)/x^2<0$ for
  $|x|<1$ we can plug in our bound on $\normInf{\beta}$ and $\wid_1=\botwid$ to obtain 
  \begin{equation}
    \frac{\normInf{\beta}e^{\wid_1^2/2}}{\wid_1} = \waveG < \lambda.
  \end{equation}
  For the Ricker wavelet, the result follows from \Cref{lem:RickerWaveBound}, an auxiliary result that we defer to \Cref{sec:RickerWaveBound}.
\end{proof}

To verify that $\alpha_j\geq 0$ note that
\begin{align}
  \alpha_j
  & \geq \psi_j - \normInf{\alpha-\psi}\\
  & \geq 1 - 2\lambda\brac{\sum_{i=0}^5
      \smo{|D|}(i\Delta(T))+\epsilon} - \normInf{\alpha-\psi},
\end{align}
by monotonicity and \Cref{lem:VTailDecay}, where $\epsilon := 10^{-12}/\wid_1$.
This gives $\alpha_j\geq \alphajG$ for the Gaussian and
$\alpha_j\geq \alphajR$ for the Ricker.

By \Cref{lem:SparseqPos} we assume, for contradiction, that
$q_i\geq\lambda$ and
$q_{i+1}\geq-\frac{\normInf{\beta}e^{\wid^2/2}}{\wid}$.
The proof with the bounds on $q_i,q_{i+1}$ swapped is the same.
Applying \Cref{lem:GenericQBounds,lem:VBounds} and monotonicity, we
obtain
\begin{align}
  Q(s_i)
  &= q_iK(0) +q_{i+1}K(s_{i+1}-s_i) + \sum_{l\neq
    j}\alpha_lB_{t_l}(s_i)+\beta_lW_{t_l}(s_i)
  + \sum_{l\in\Noi}D_{s_l}(s_i,\lambda,\rho'_l)\\
  &\geq q_i +q_{i+1}K(\wid)\notag\\ 
  &- 2\brac{3\epsilon +\sum_{i=1}^5
    \normInf{\alpha}\smo{|B|}(i\Delta(T)-\wid_2)
    +\normInf{\beta}\smo{|W|}(i\Delta(T)-\wid_2)+\lambda\sum_{i=0}^5 \smo{|D|}(i\Delta(T))},\label{eq:qLowerBoundTest}
\end{align}
where $\epsilon:=10^{-12}/\wid_1$. Applying \Cref{lem:SparseqPos} gives
\begin{equation}
  q_{i+1}K(s_{i+1}-s_i) \geq -\frac{\normInf{\beta}}{\wid}\geq-\frac{\normInf{\beta}}{\wid_1}
\end{equation}
for the Gaussian and
\begin{equation}
  q_{i+1}K(s_{i+1}-s_i) \geq
  -\frac{\normInf{\beta}}{\wid(3-\wid^2)}\geq -\frac{\normInf{\beta}}{\wid_1(3-\wid_1^2)}
\end{equation}
for the Ricker (differentiate $\wid(3-\wid^2)$).
If the expression on the right-hand side of
\eqref{eq:qLowerBoundTest} is strictly larger than $1$ when
$q_i=\lambda$ then we obtain
$Q(s_i)>1$ which contradicts \Cref{lem:SparseQBound}.  Computing these
bounds we have $Q(s_i)\geq \QsiG$ for the Gaussian and $Q(s_i)\geq
\QsiR$ for the Ricker. This contradiction implies $|q_i|<\lambda$ and
completes the proof.
\subsection{\Cref{lem:RickerWaveBound}}
\label{sec:RickerWaveBound}
\begin{lemma}
  \label{lem:RickerWaveBound}
  Under the assumptions of \Cref{lem:SparseqPos} we have
  \begin{alignat}{4}
    0&\leq\frac{\normInf{\beta}(1-s_{i+1}^2)\KG(s_{i+1})}{(s_{i+1}-s_i)(3-(s_i-s_{i+1})^2+s_i^2s_{i+1}^2)\exp\brac{-\frac{s_i^2+s_{i+1}^2}{2}}}
    & \leq \frac{\normInf{\beta}e^{\wid^2/2}}{\wid(3-\wid^2)} &<\lambda,\label{eq:WaveBound1}\\
    0&\leq\frac{\normInf{\beta}(1-s_i^2)\KG(s_i)}{(s_{i+1}-s_i)(3-(s_i-s_{i+1})^2+s_i^2s_{i+1}^2)\exp\brac{-\frac{s_i^2+s_{i+1}^2}{2}}}
    & \leq \frac{\normInf{\beta}e^{\wid^2/2}}{\wid(3-\wid^2)} &<\lambda,\label{eq:WaveBound2}
  \end{alignat}
  for the Ricker kernel with $\wid\in[\wid_1,\wid_2]$.
\end{lemma}
\begin{proof}
  As $\wid_2=\topwidR$ we immediately obtain that all terms in the
  statement are nonnegative.
  We define $f$ below so
  that the first inequality in \eqref{eq:WaveBound2} reads $0\leq \normInf{\beta}f(s_i)$:
  \begin{equation}
    f(s) = \frac{(1-s^2)\exp((s+\wid)^2/2)}{\wid(3-\wid^2+s^2(s+\wid)^2)}.
  \end{equation}
  
  Differentiating and some algebraic manipulations yield
  \begin{equation}
    \hspace{-.5cm}f'(s) =
    -\frac{e^{\frac{1}{2} (s+\wid)^2} \left(\left(s^2-1\right)^2 \wid^3+s \left(3 s^4-4 s^2+1\right) \wid^2+\left(3 s^6-5 s^4+9 s^2-3\right) \wid+s \left(s^6-3 s^4+7 s^2+3\right)\right)}{\wid \left(\left(s^2+(s-1) \wid\right) (s (s+\wid)+\wid)+3\right)^2}.
  \end{equation}
  We will show this is always positive when
  $\wid\in[\botwidR,\topwidR]$.  As the exponential and the
  denominator are always positive, we focus on the parenthesized
  term in the numerator.
  As $-\wid\leq s \leq 0$ 
  \begin{align}
    (s^2-1)^2\wid^3 & \leq \wid_2^3\\
    s(3s^4-4s^2+1)\wid^2 & \leq 0\\
    (3s^6-5s^4+9s^2-3)\wid & \leq \wid(12s^2-3) \leq
    -\wid_1(3-12\wid_2^2)\\
    s(s^6-3s^4+7s^2+3) & \leq s(3-3s^4) \leq 0.
  \end{align}
  Noting that $\wid_2^3-\wid_1(3-12\wid_2^2) \leq -0.20268$, we see
  that $f'(s)>0$.  This shows that $s=0$ maximizes $f$.  As a result,
  \begin{equation}
    f(0) = \frac{\exp(\wid^2/2)}{\wid(3-\wid^2)}
  \end{equation}
  and
  \begin{equation}
    \frac{\normInf{\beta}\exp(\wid^2/2)}{\wid(3-\wid^2)}
    \leq \frac{\betaR\exp(\wid_2^2/2)}{\wid_1(3-\wid_1^2)} \leq \waveR < \lambda.
  \end{equation}
  Let $(a,b)$ be a valid pair of values for $(s_i,s_{i+1})$, i.e.,
  \begin{equation}
    -\wid \leq a \leq 0 \leq b = a+\wid.
  \end{equation}
  Then $(-b,-a)$ is another valid pair.  Replacing
  $(s_i,s_{i+1})\to(-s_{i+1},-s_i)$ in \eqref{eq:WaveBound2} gives
  \eqref{eq:WaveBound1} and completes the proof.
\end{proof}

\end{document}